\documentclass[a4paper]{amsart}

\usepackage{color,graphicx,enumerate,wrapfig,amssymb,mathtools}
\usepackage{epsfig}

\usepackage{eucal}

\newcommand{\R}{{\mathbb R}}

\newcommand{\uu}{\mathbf{u}}
\newcommand{\vv}{\mathbf{v}}
\newcommand{\w}{\mathbf{w}}
\newcommand{\h}{\mathbf{h}}

\newcommand{\q}{\mathbf{q}}

\usepackage{amsmath}
\usepackage{amsfonts}
\usepackage{caption}
\usepackage{subcaption}
\usepackage{esint}

\newtheorem{theorem}{Theorem}
\theoremstyle{plain}
\newtheorem{corollary}{Corollary}
\newtheorem{definition}{Definition}

\newtheorem{lemma}{Lemma}
\newtheorem{remark}{Remark}
\numberwithin{equation}{section}

\begin{document}

\author{Luis A. Caffarelli} 
\address[Luis A. Caffarelli] 
{Department of Mathematics \newline 
\indent University of Texas at Austin \newline 
\indent 1 University Station, C1200, Austin, TX 78712, USA} 
\email[Luis A. Caffarelli]{caffarel@math.utexas.edu} 

\author{Henrik Shahgholian}
\address[Henrik Shahgholian]
{Department of Mathematics \newline 
\indent KTH Royal Institute of Technology \newline 
\indent 100 44 Stockholm, Sweden} 
\email[Henrik Shahgholian]{henriksh@kth.se}

\author{Karen Yeressian}
\address[Karen Yeressian]
{Department of Mathematics \newline 
\indent KTH Royal Institute of Technology \newline 
\indent 100 44 Stockholm, Sweden} 
\email[Karen Yeressian]{kareny@kth.se}

\title[A minimization problem with free boundary]
{A minimization problem with free boundary related to a cooperative system}

\date{\today}
\keywords{Minimization, Cavitational flow,  Free boundary, System, Regularity} 
\subjclass[2010]{Primary 35R35; Secondary 35J60} 

\begin{abstract} 
We study the minimum problem for the functional 
\begin{equation*}
\int_{\Omega}\bigl( \vert \nabla \uu \vert^{2} + Q^{2}\chi_{\{\vert \uu\vert>0\}} \bigr)dx
\end{equation*}
with the constraint $u_i\geq 0$ for $i=1,\cdots,m$ 
where $\Omega\subset\mathbb{R}^{n}$ is a bounded domain 
and $\uu=(u_1,\cdots,u_m)\in H^{1}(\Omega;\mathbb{R}^{m})$. 

Using an array of technical tools, from geometric analysis for the free boundaries, we reduce the
problem to its scalar counterpart and hence conclude similar results as that of scalar problem. This can also be seen as the most novel part of the paper, that possibly can lead to further developments of free boundary regularity for systems.

\end{abstract}

\maketitle 

\tableofcontents

\section{Introduction}\label{section1000} 

\subsection{Background}\label{subsection500} 
In the last four decades the regularity theory of 
free boundary problems has seen    an unprecedented surge of developments of  new technical devices, that  have  resulted in solving both old and new problems, unfeasible with earlier techniques. Most of these tools, enrooted in the analysis  of minimal surfaces, have been enhanced 
and undergone major changes and in some cases even being reincarnated. 
Cavitational flow, Obstacle problem and Thin obstacles are a few 
among many of those problems, that have been treated successfully with these newly developed tools. It is, however, not until very recently that problems which involve system of equations have been treated from a regularity theory point of view, see 
\cite{JiangLin2007,Andersson2016,AnderssonShahgholianUraltsevaWeiss2015}. 
There seems to be lack of a general methodology and approach for analyzing  the regularity for systems of free boundary problems.\footnote{Competitive systems, which gives rise to disjoint support of limiting solutions,  have been much in focus in the last decade (see e.g. \cite{caff-lin-2008}, \cite{Terracini}).  Competitive system of more than two equations usually give rise to the so-called junction points, where more than two-phases can meet; such points are called multiple junction points. Hence the approach for studying competitive system  differs substantially from that of cooperative systems, where they usually give rise to smooth free boundaries, that are locally graphs.} Our intention with this paper is to initiate the study of Cavitational problems 
where several flows  are involved, and interact whenever there is 
phase transition. 

The mathematical model we have chosen to work with is the by-now classical 
problem of Bernoulli type free boundary, that was treated by  the first author with H. Alt \cite{AltCaffarelli1981}.  
The simplest setting of such a problem asks for properties 
of the minimizers $\uu=(u_1,\cdots,u_m)$ of the functional 
\begin{equation*}
J(\uu)=\int_{\Omega}\bigl(\vert \nabla \uu\vert^{2}
+Q^{2}\chi_{\{\vert \uu\vert>0\}}\bigr)dx
\end{equation*}
over an appropriate Sobolev vector-valued functions, 
domain $\Omega \subset \mathbb{R}^n$, smooth enough $Q$, 
and boundary values. 

Minimizers of this functional describe (optimal) stationary  thermal   insulation, allowing a prescribed  heat  loss from the insulating layer. 
 The heat flows in from the boundary of the domain $\Omega$,  through a vector function 
 $\mathbf{g} \in H^1 (\Omega; \R^m)$ on the boundary (boundary data). 
 Each $g_i$ gives rise to  a potential function $u_i$ describing the heat distribution from  the data $g_i$, and the system has to cost through Dirichlet energy as well as the total volume of heated
 region. Since this is a system, the latter is described by $|\{|\uu| >0 \}| $. If the supports of 
  $g_i$-s stay  far from each other (and data is small enough) then it is reasonable  that the system behaves exactly  like scalar case, for each $i=1, \cdots , m$.  When the supports of $g_i$-s come close (or some  $g_i$-s become large), then naturally the volume of each  support $\{ u_i >0 \} $ increases, and at some stage it is less costly to use same insulation layer, i.e. they prefer  to share support, and hence $\sup u_i = \sup u_j$ for some of these $i,j$.\footnote{A different way of explaining this is to consider two balls $B_1(z)$, and $B_1(z + Re)$ for a  direction $e$,  and a large constant  $R>0$. We set $D_R = B_1(z) \cup B_1(z + Re)$, and minimize our functional  in $\R^n \setminus D_R$ with some  non-negative  boundary data  on $D_R$. For large values of $R$ the insulation layers for each ball is separated, and by decreasing $R$ the supports eventually intersect. But before this happens, it is less costly to share insulation, by having the same support for  all components of the solution vector.} 
  Those $g_i$ that are still small  (and their support stay far from others) will insulate separately. The total heat of the system at each point is given by $\sum u_i$, and this is a major difference between our problem and standard scalar problem.
 A similar model  can appear in  population dynamics where several species coexist, and overflow the patches. In such models (and many others) each $u_i$ may represent a population density (or any quantity given by the system).  
   We refer to Section \ref{supports}  for relation between the supports of $u_i$, and 
 for rigorous  arguments concerning our discussion here.

Other models of such a problem appears as equilibrium state(s) of 
cooperative systems, corresponding to  reaction-diffusion systems, 
with high concentration of energy close to the free  boundary. 
Limit of such singularly perturbed  problems lead to minimization 
of our functional. Other related models may appear in shape 
optimization, where the Dirichlet energy of vector-valued 
functions are to be minimized, subject to volume constraint 
of the type $\vert\{\vert\uu\vert>0\}\vert=A>0 $, with $\vert\Omega\vert > A $, and 
Dirichlet data on $\partial \Omega$. It is noteworthy that our approach in this paper also applies 
to the corresponding two-phase problems, as well as singular perturbations, and 
volume-constrained  maps.

Our results are  in lines of that of \cite{AltCaffarelli1981} and several of the 
succeeding papers \cite{AguileraCaffarelliSpruck1987}, \cite{Weiss1999}, etc.
However, our methodology (besides the obvious preliminary footwork) and strategy 
is somehow new.  For the main regularity theory,  instead of working with the system, 
we use a reduction method 
to the scalar case  with the  cost of loosing the regularity of the 
free boundary condition that is assumed/given in the scalar case.  More exactly, our analysis boils 
down to a weak solution of 
\begin{equation*}
\Delta u_i=w_iQ\mathcal{H}^{n-1}\llcorner(\Omega\cap\partial^{*}\{\vert \uu\vert>0\})
\enskip\text{for}\enskip i=1,\cdots,m\text{,}
\end{equation*}
where (see Notation section for definitions)
\begin{equation*}
w_i(x)=\lim_{y\in\{\vert \uu\vert>0\},y\to x}\frac{u_i(y)}{\vert \uu(y)\vert}\text{.}
\end{equation*}
In this reformulation   the information about the continuity of the Bernoulli boundary condition 
is lost, since a priori we do not know 
how regular $w_i$ are.  The heart of the matter lies in proving the H\"older 
regularity of the functions $w_i$. It should be remarked that this might be seen 
as the most novel part of of our paper; see Section \ref{section1800}.

In a follow up paper \cite{CSY2016a}  we shall consider this problem in a more general setting, allowing sign change as well as more general integrand (anisotropic as well as degenerate/singular) in our functional.

\subsection{Mathematical Setting}\label{subsection600} 
Let $\Omega\subset\mathbb{R}^{n}$ be a bounded domain and $m\geq 1$ an integer. 
Let $Q:\Omega\to\mathbb{R}$ be Lebesgue measurable and there exist constants 
$Q_{max}\geq Q_{min}>0$ such that $Q_{min}\leq Q\leq Q_{max}$ a.e. in $\Omega$. 
For $\uu\in H^{1}(\Omega;\mathbb{R}^{m})$ let us define 
\begin{equation*}
J(\uu)=\int_{\Omega}\bigl(\vert \nabla \uu\vert^{2}+Q^{2}\chi_{\{\vert \uu\vert>0\}}\bigr)dx
\end{equation*}
where 
\begin{equation*}
\vert\nabla \uu\vert^{2}=\vert\nabla u_1\vert^{2}+\cdots+\vert\nabla u_m\vert^{2}\text{,} 
\end{equation*} 
here $\vert\cdot\vert$ denotes the Euclidean length. 

Let $\mathbf{g}\in H^{1}(\Omega;\mathbb{R}^{m})$ such that 
$g_i\geq 0$ a.e. in $\Omega$ for $i=1,\cdots,m$. 
We consider the minimization problem of the functional 
$J$ for $\uu\in H^{1}(\Omega;\mathbb{R}^{m})$ 
under the constraint that $\uu=\mathbf{g}$ on $\partial\Omega$ and the sign constraints 
\begin{equation*}
u_i\geq 0\enskip\text{a.e. in}\enskip\Omega\enskip\text{for}\enskip i=1,\cdots,m\text{.} 
\end{equation*}

\begin{remark} If we change the volume constraint in our functions above, to 
$ \sum_i \chi_{\{u_i >0\}}$ then the components decouple and we fall back to scalar case for each $i=1, \cdots, m$. 
\end{remark}

In Figure \ref{figure500} an example of local minimizer 
(see Definition \ref{definition500}) is depicted. 
In this example we have $\Omega=(-1,1)^{2}$, $m=2$, $Q=1$, 
$g_{1}(x)=x_2^{-}$ and $g_{2}(x)=x_1^{+}$. 
Because in this paper the sum of the components of $\uu$ 
and the length of $\uu$ will play an important role, we have also 
depicted these functions.  

\begin{figure}
        \begin{center}
        
        \vspace{0.4cm}
        
        \begin{subfigure}[b]{6cm}
        	        \centering
                \includegraphics[width=5cm]{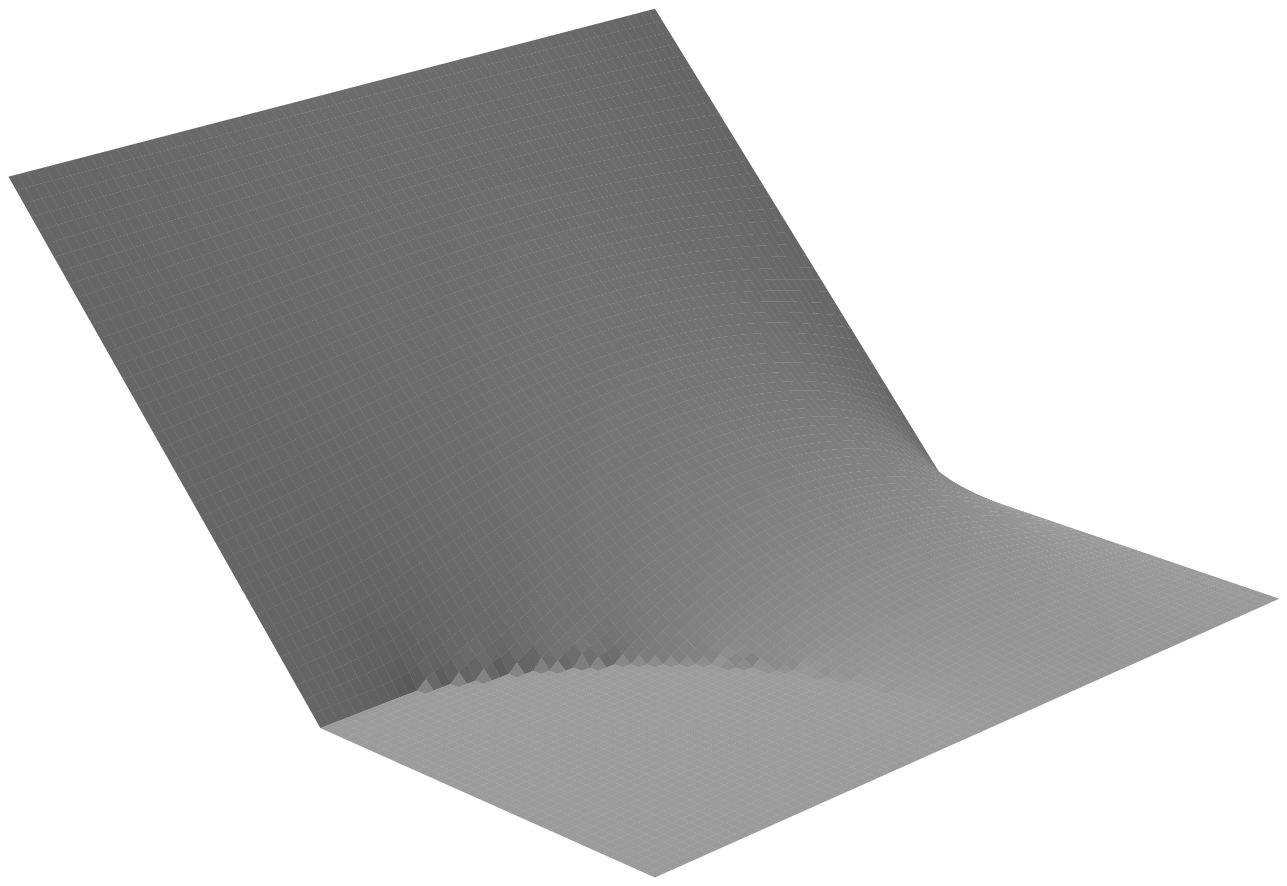}
                \caption{$u_{1}$}
                \label{figure100}
        \end{subfigure}
        \begin{subfigure}[b]{6cm}
	      	\centering
               	\includegraphics[width=5cm]{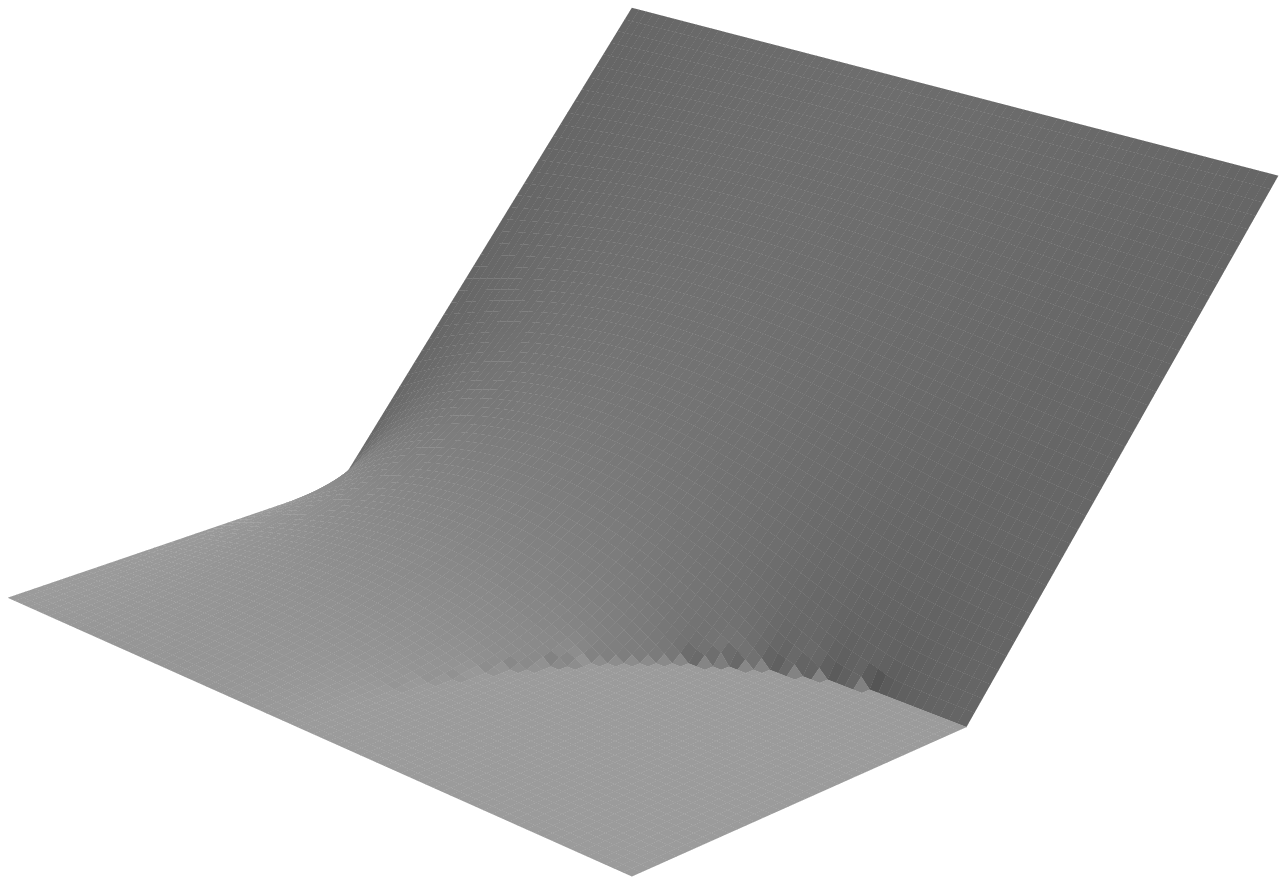}
              	\caption{$u_{2}$}
              	\label{figure200}
        \end{subfigure}
        
         \vspace{0.5cm}
        
         \begin{subfigure}[b]{6cm}
	      	\centering
               	\includegraphics[width=5cm]{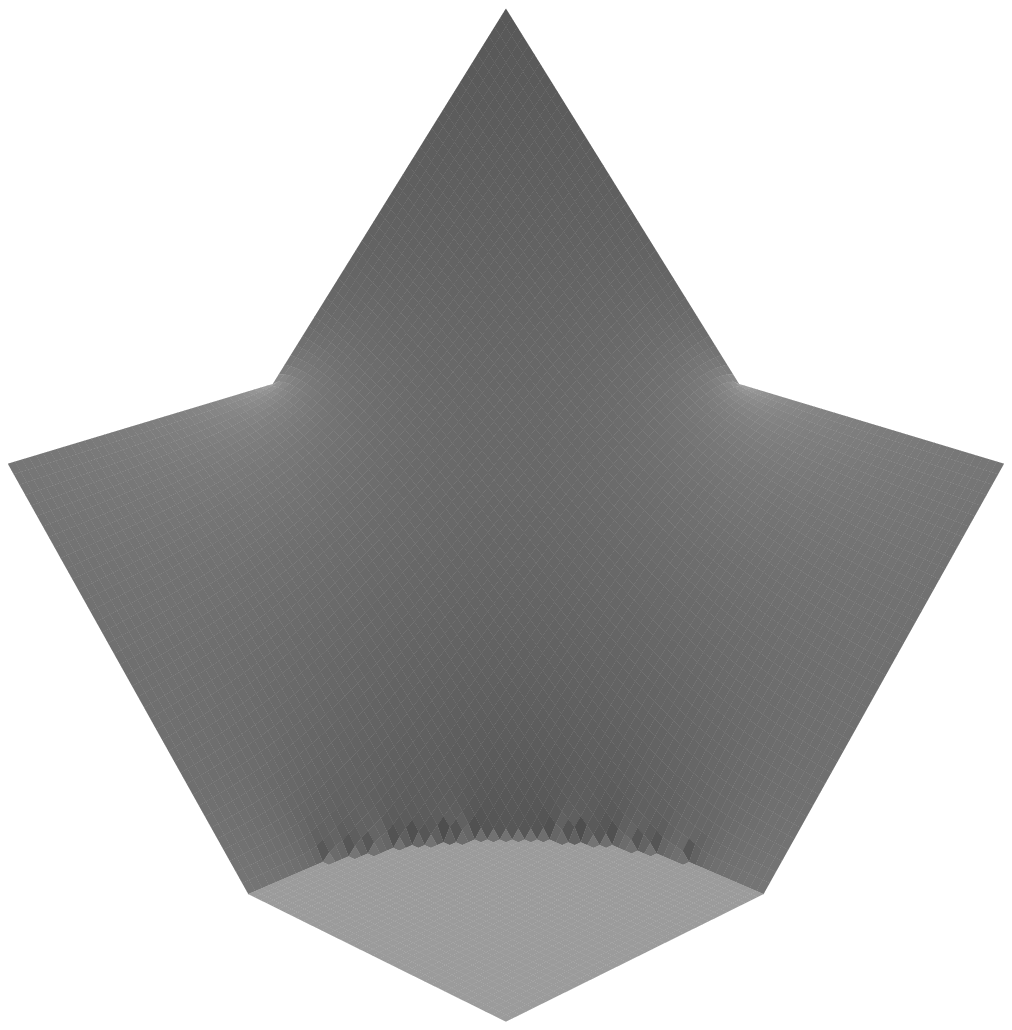}
              	\caption{$u_{1}+u_{2}$}
              	\label{figure300}
        \end{subfigure}
         \begin{subfigure}[b]{6cm}
	      	\centering
               	\includegraphics[width=5cm]{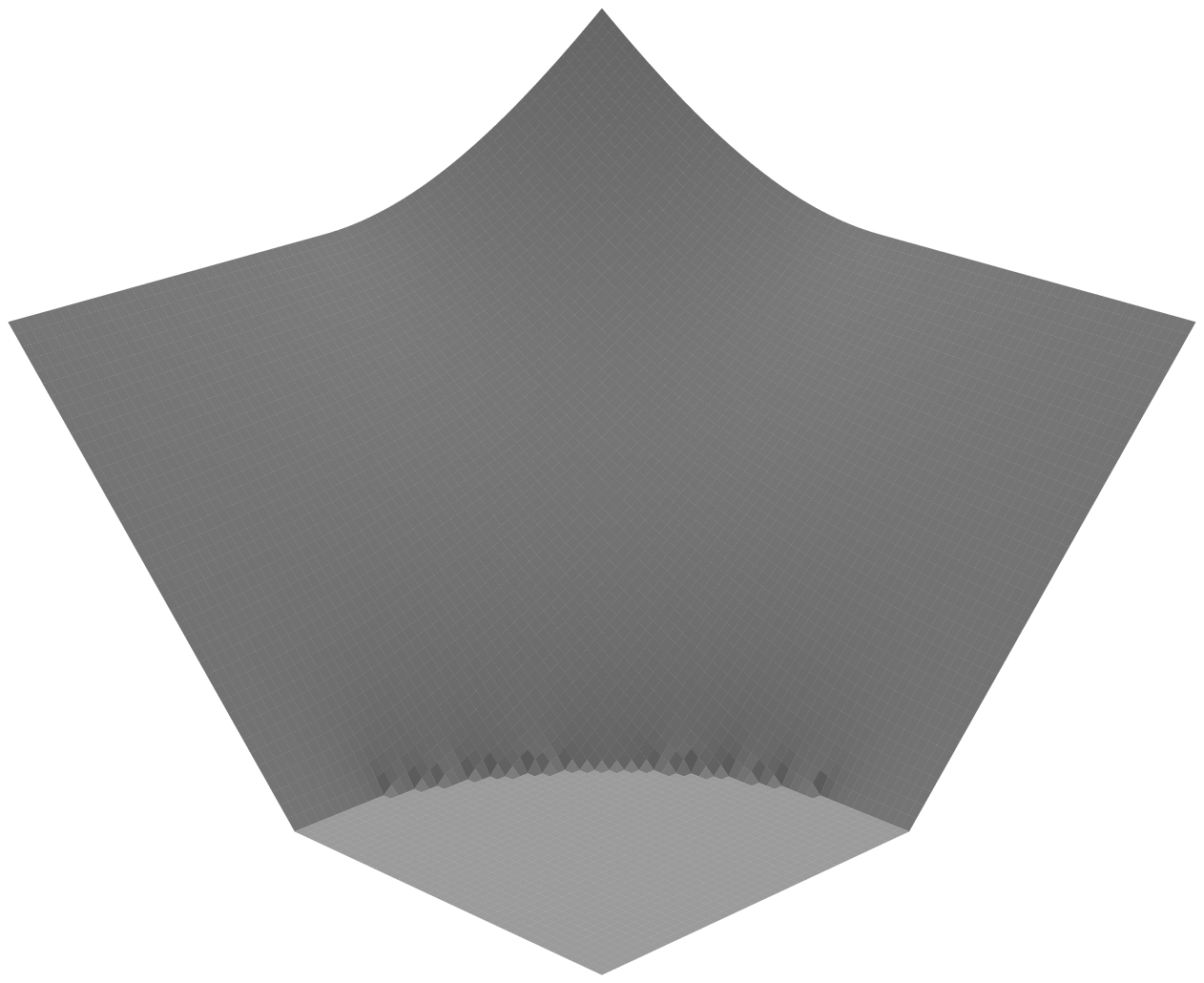}
              	\caption{$\vert \uu\vert$}
              	\label{figure400}
        \end{subfigure}
	\end{center}
        \caption{Components of a local minimizer $\uu$ 
        together with the sum $u_1+u_2$ and its length $\vert \uu\vert$.} 
        \label{figure500}
\end{figure}

\subsection{Notation}\label{subsection700} 
Here we shall line up important notations that are frequently used in this paper.

\begin{tabular}{ll}
  $C$, $C_1$, $C_2$  & generic constants; \\
  $\chi_{D}$ & characteristic function of the set $D$ ($D\subset\mathbb{R}^{n}$); \\
  $\overline{D}$ & the closure of $D$; \\
  $\partial D$ & boundary of $D$; \\
  $\partial^{*} D$ & reduced boundary of $D$; \\
  $\partial_{*} D$ & measure theoretic boundary of $D$; \\
  $D^{\circ}$ & interior of $D$; \\
  $\sigma$ & surface measure; \\
    $\vert\cdot\vert$ & absolute value, euclidean length of a vector, \\ 
                             & norm of a matrix, Lebesgue measure or surface measure; \\
  $\mathcal{H}^{m}$ & $m$-dimensional Hausdorff measure (on $\mathbb{R}^{n}$); \\
  $\Vert\cdot\Vert$ & norm of functions; \\
  $[\cdot]$ & seminorm of functions; \\
  $B_r^{m}(x)$ & $\bigl\{y\in\mathbb{R}^{m}\ \vert\ \vert y-x\vert<r \bigr\}$; \\
  $B_r(x)$, $B_r$, $B$ & $B_r^{n}(x)$, $B_r^{n}(0)$, $B_1^{n}(0)$; \\
  $f^{+}$, $f^{-}$ & $\max(0,f)$, $\max(0,-f)$; \\
  $\subset\subset$ & compactly contained; \\
  $\fint$ & integral mean; \\
  $\nu$ & outer normal; \\
  $\mu\llcorner A$ & measure $\mu$ restricted to the set $A$. 
\end{tabular}

\section{Main Results}\label{section1200} 

Let us denote by $\mathcal{A}$ the set of our admissible functions, i.e. 
\begin{multline}\label{equ300} 
\mathcal{A}=
\Bigl\{
\vv\in H^{1}(\Omega;\mathbb{R}^{m})\enskip\Bigm\vert\enskip 
\vv=g\enskip\text{on}\enskip\partial\Omega \\ 
\enskip\text{and}\enskip 
\vv_i\geq 0\enskip\text{a.e. in}\enskip\Omega\enskip\text{for}\enskip i=1,\cdots,m
\Bigr\}\text{.}
\end{multline}

We call $\uu\in\mathcal{A}$ an absolute minimizer if $J(\uu)\leq J(\vv)$ for all $\vv\in\mathcal{A}$.

\begin{theorem}\label{theorem1000} 
There exists an absolute minimizer of our problem. 
\end{theorem}

For $\uu,\vv\in H^{1}(\Omega;\mathbb{R}^{m})$ let us  define 
the metric  $d$  on $H^{1}(\Omega;\mathbb{R}^{m})$ by
\begin{equation}\label{equ400} 
d(\uu,\vv)
:=\Vert \uu-\vv\Vert_{H^{1}(\Omega;\mathbb{R}^{m})}
+\Vert \chi_{\{\vert \uu\vert>0\}}-\chi_{\{\vert \vv\vert>0\}}\Vert_{L^{1}(\Omega)}\text{.} 
\end{equation}

\begin{definition}\label{definition500}  
We call $\uu\in\mathcal{A}$ a local 
minimizer if there exists $\epsilon>0$ such that 
$J(\uu)\leq J(\vv)$ for $\vv\in\mathcal{A}$ with $d(\uu,\vv)<\epsilon$.  
\end{definition}

\begin{theorem}[Optimal linear growth]\label{theorem1100} 
There exists $C>0$ such that for $\uu$ a (local) minimizer and 
$B_r(x)\subset\subset\Omega$ (small balls) if 
$\overline{B_r(x)}\cap\{\vert \uu\vert=0\}\not=\emptyset$ then 
\begin{equation*}
\fint_{\partial B_{r}(x)} u_i d\sigma
\leq 
CQ_{max}r\enskip\text{for}\enskip i=1,\cdots,m\text{.} 
\end{equation*}
\end{theorem}

In particular, from the linear growth estimate proved in Theorem \ref{theorem1100} 
it follows that $\uu$ is Lipschitz continuous, see Corollary \ref{corollary1000}. 

\begin{theorem}[Optimal linear nondegeneracy]\label{theorem1200} 
Let $\uu$ be a (local) minimizer, 
for $B_{r}(x)\subset\Omega$ (small balls), 
$0<\rho<1$ (in the case when $n=2$ and $u$ is a local minimizer also small enough $\rho$) 
and $\overline{B_{\rho r}(x)}\cap\{\vert \uu\vert>0\}\not=\emptyset$ then 
\begin{equation*}
\sup_{B_{r}(x)}\vert \uu\vert 
\geq  
c Q_{min} r
\end{equation*}
where $c>0$ depends only on $\rho$. 
\end{theorem}

\begin{theorem}[Equation satisfied by each component]\label{theorem1300} 
Let $Q$ be a continuous function, 
then for $\mathcal{H}^{n-1}$ a.e. point $x\in\Omega\cap\partial^{*}\{\vert \uu\vert>0\}$, 
$i=1,\cdots,m$ and $\eta>0$ the non-tangential limit 
\begin{equation}\label{equ500}
w_i(x)=\lim_{y\in\{\vert \uu\vert>0\}\cap\{-(y-x)\cdot\nu_{\{\vert \uu\vert >0\}}(x)\geq\eta\},y\to x}\frac{u_i(y)}{\vert \uu(y)\vert}
\end{equation}
(here $\nu_{\{\vert \uu\vert >0\}}(x)$ is the outer normal to $\{\vert \uu\vert>0\}$ at the point $x$)
exists and we have the equations 
\begin{equation}\label{equ600}
\Delta u_i=w_iQ\mathcal{H}^{n-1}\llcorner(\Omega\cap\partial^{*}\{\vert \uu\vert>0\})
\enskip\text{for}\enskip i=1,\cdots,m\text{.} 
\end{equation} 
\end{theorem}

In the following theorem we prove that around a 
free boundary point, the set $\{\vert \uu\vert>0\}$ is a non-tangentially accessible domain. 
In the Definition \ref{definition4300} a non-tangentially accessible domain, 
with its associated parameters $M$, $\xi$ and $c$, is defined. 

\begin{theorem}[$\{\vert \uu\vert>0\}$ is non-tangentially accessible]\label{theorem1350} 
Let $\uu$ be a (local) minimizer, for $B_{r_{0}}(x_{0})\subset\Omega$ (small enough) with $x_{0}\in\partial\{\vert \uu\vert>0\}$, 
there exists $0<\epsilon_{1}<1$ and $0<\tilde\epsilon_{1}<\epsilon_{1}$ 
such that $B_{\epsilon_{1} r_{0}}(x_{0})\cap\{\vert \uu\vert>0\}$ is 
a non-tangentially accessible domain with parameters 
$M>1$, $\xi=\tilde\epsilon_{1}r_{0}$ and $0<c<1$ 
(where $\epsilon_{1}$, $\tilde\epsilon_1$, $M$ and $c$ depend only on 
$n$, $m$, $\frac{Q_{max}}{Q_{min}}$ and additionally on $u$ in the case when $n=2$ and $u$ is a local minimizer). 
\end{theorem}

\begin{definition}\label{definition600}
For $0<\sigma\leq 1$ and $\nu\in\partial B$ we say that the minimizer 
$\uu$ is $\sigma$-flat in $B_{\rho}(x)\subset\Omega$ in the direction $\nu$ if 
$x\in\partial\{\vert \uu\vert>0\}$ and $\uu=0$ 
in $B_{\rho}(x)\cap\{y\ \vert\ (y-x)\cdot\nu\geq\sigma\rho\}$. 
\end{definition} 

Assume $x_{0}$, $r_{0}$ and $\epsilon_{1}$ to be as in 
Theorem \ref{theorem1350}. 
In Lemma \ref{lemma2200} and \ref{lemma2300} 
using the comparison principle for 
non-tangentially accessible domains 
(see Lemma \ref{lemma3400}) 
we obtain that 
$w_i$ for $i=1,\cdots,m$ are H\"older 
continuous in $B_{\epsilon_{2}\epsilon_1r_0}(x_0)\cap\{\vert\uu\vert>0\}$ where 
$0<\epsilon_{2}<1$ depends on the parameters of 
non-tangentially accessibility which in turn depend on $n$, $m$, $\frac{Q_{max}}{Q_{min}}$ 
and additionally on $u$ in the case when $n=2$ and $u$ is a local minimizer. 

\begin{theorem}[Flatness implies regularity]\label{theorem1360} 
Let $Q$ be H\"older continuous 
and $\uu$ be a minimizer of our functional. 
Then there are constants $\alpha>0$, $\beta>0$, $\sigma_{0}>0$, $\tau_{0}>0$ and $C<\infty$ such that 
if $\uu$ is $\sigma$-flat in $B_{\rho}(x_0)$ in the direction $\nu$ with $\sigma\leq\sigma_0$ 
and $\rho\leq\min(\tau_0\sigma^{\frac{2}{\beta}},\frac{1}{2}\epsilon_{2}\epsilon_1r_0)$ then 
\begin{equation*}
B_{\frac{\rho}{4}}(x_0)\cap\partial\{\vert \uu\vert>0\}\enskip\text{is a}\enskip C^{1,\alpha}
\enskip\text{surface}
\end{equation*}
(a graph in direction $\nu$ of a $C^{1,\alpha}$ function), and for $x_1,x_2$ on this surface 
\begin{equation*}
\vert \nu(x_1)-\nu(x_2)\vert\leq C\sigma\bigl\vert\frac{x_2-x_1}{\rho}\bigr\vert^{\alpha}\text{.} 
\end{equation*}
The constants depend on $n$, $Q_{min}$, $Q_{max}$ 
and the H\"older exponent and norm of $Q$. 
\end{theorem}

\begin{theorem}[Classification of homogenous global minimizers]\label{theorem1400} 
The function $\uu$ is a first order homogenous absolute minimizer in $B$ with connected 
$\{\vert \uu\vert>0\}$ if and only if $u_i=c_i v$ 
where $v$ is a first order homogenous scalar absolute minimizer in $B$ 
with connected $\{v>0\}$, $c\in\mathbb{R}^{m}$, 
$\vert c\vert=1$ and $c_i\geq 0$ for $i=1,\cdots,m$. 
\end{theorem}

\begin{definition}\label{definition700} 
Let $\uu$ be a minimizer in $\Omega$. 
We call $\Sigma=\Omega\cap\bigl(\partial\{\vert \uu\vert>0\}\backslash\partial^{*}\{\vert \uu\vert>0\}\bigr)$ 
the singular set of $\uu$. 
\end{definition} 

Let $k^{*}$ be the critical dimension defined in the Section 3 of \cite{Weiss1999}. 

Let us note that by \cite{CaffarelliJerisonKenig2004,DeSilvaJerison2009,JerisonSavin2015} 
it is known that $5\leq k^*\leq 7$. 

\begin{theorem}[Structure of the free boundary]\label{theorem1500} 
Let $Q$ be H\"older continuous 
and $\uu$ be a minimizer in $\Omega$. Then 
$\Sigma$ is a closed set in the relative topology of $\Omega$. 
The free boundary is $C^{1,\alpha}$ smooth in the open set $\Omega\backslash\Sigma$. 

If $n<k^{*}$ then $\Sigma=\emptyset$. If $n=k^{*}$ then the singular set, i.e. $\Sigma$, 
is at most consisting of isolated points. If $n>k^{*}$ then for $s>n-k^{*}$ we have 
$\mathcal{H}^{s}(\Sigma)=0$, i.e. the 
Hausdorff dimension of the singular set is at most $n-k^{*}$. 
\end{theorem}

\begin{theorem}[Higher regularity of the free boundary]\label{theorem1600} 
If $Q\in C^{1,\gamma}$ for $0<\gamma<1$, 
$Q\in C^{k,\gamma}$ for $k\geq 2$ and $0<\gamma<1$, 
$Q\in C^{\infty}$ or $Q$ is real analytic then 
the free boundary is $C^{2,\min(\alpha,\gamma)}$ 
(with $\alpha$ as in Theorem \ref{theorem1360}), $C^{1+k,\gamma}$, $C^{\infty}$ or real analytic, 
respectively, smooth in the open set $\Omega\backslash\Sigma$. 
\end{theorem}

\subsection{Structure of the paper}\label{subsection800} 

This paper is structured as follows. 
In Section \ref{section1300}, the existence of an absolute minimizer is established. 
In Section \ref{section1400}, general structure and initial regularity 
of minimizers are demonstrated. 
In Section \ref{section1500}, the optimal linear growth of 
minimizers near to the free boundary 
is proved. 

In Section \ref{section1600}, we carry out preliminary local 
analysis of the minimizers and the free boundary. 
We obtain the optimal linear nondegeneracy of 
minimizers near to the free boundary, nonvanishing density of the coincidence set 
$\{\vert \uu\vert=0\}$ and the noncoincidence set $\{\vert \uu\vert>0\}$ 
near to the free boundary, that noncoincidence set $\{\vert \uu\vert>0\}$ has 
locally finite perimeter, a domain variation formula and 
that linear blowup limits at the free boundary are absolute minimizers. 

In Section \ref{section1800}, we derive the equation satisfied by each component 
$u_i$, we prove that the noncoincidence set $\{\vert \uu\vert>0\}$ is a non-tangentially 
accessible domain, using the last property, locally we reduce the problem to a nondegenerate 
scalar one. 

In Section \ref{section1810}, using the reduction to a scalar problem we obtain that 
flatness of the free boundary implies its regularity and also we discuss the equivalence of 
various definitions of regular points of the free boundary. 

In Section \ref{section1900}, after proving a Poho\v{z}aev type identity we obtain 
a Weiss type monotonicity formula. This monotonicity formula establishes the 
homogeneity of blowup limits. 

In Section \ref{section2000}, we classify all possible homogenous global minimizers 
by relating them with those of the scalar problem. 

In Section \ref{section2050}, we obtain the structure of the free boundary and 
its higher regularity close to regular points provided the data of the problem, i.e. $Q$, is 
accordingly regular. 

In the appendix, for ease of reference, we bring the definition of a 
non-tangentially accessible domain and the associated comparison principle. 


\section{Existence of an Absolute Minimizer\\ 
(Proof of Theorem \ref{theorem1000})}\label{section1300}

\begin{proof}[Proof of Theorem \ref{theorem1000}]
We have $g\in\mathcal{A}$ (see \eqref{equ300} for the definition of $\mathcal{A}$) 
thus $\mathcal{A}\not=\emptyset$. 
Let $\uu^k\in\mathcal{A}$ be a minimizing sequence, i.e. 
\begin{equation*}
\inf_{\vv\in\mathcal{A}}J(\vv)=\lim_{k\to\infty}J(\uu^k)\text{.} 
\end{equation*}
Then because $g\in\mathcal{A}$ for large enough $k$ we have 
\begin{equation}\label{equ1000} 
J(\uu^{k})<J(\mathbf{g})+1<\infty\text{.} 
\end{equation}
We might assume that \eqref{equ1000} holds for all $k\geq 1$.  Clearly we have the estimate 
\begin{equation}\label{equ1010} 
\int_{\Omega}\vert\nabla\uu^{k}\vert^{2}dx\leq J(\uu^{k})\text{.} 
\end{equation}
Now because $\uu^{k}=\mathbf{g}$ on $\partial\Omega$ 
by the Poincar\'e inequality, \eqref{equ1000} and \eqref{equ1010} we obtain the uniform bound 
\begin{equation*}
\Vert\uu^{k}\Vert_{H^{1}(\Omega;\mathbb{R}^{m})}
\leq C\enskip\text{for}\enskip k\geq 1\text{.} 
\end{equation*}
Also we have trivially the uniform bound 
$\Vert\chi_{\{\vert \uu^{k}\vert>0\}}\Vert_{L^{\infty}(\Omega)}\leq\vert\Omega\vert$ for $k\geq 1$. 
Thus there exists $\uu\in H^{1}(\Omega;\mathbb{R}^{m})$, $w\in L^{\infty}(\Omega)$ 
and a subsequence $k_{\ell}$ such that $\uu^{k_{\ell}}\to \uu$ weakly 
in $H^{1}(\Omega;\mathbb{R}^{m})$, 
$\uu^{k_{\ell}}\to \uu$ a.e. in $\Omega$ and $\chi_{\{\vert \uu^{k_{\ell}}\vert>0\}}\to w$ 
weak$^{*}$ in $L^{\infty}(\Omega)$. 
We denote the sequence $k_{\ell}$ for simplicity by $k$. 

Because $\mathcal{A}$ is a closed (with respect to the strong topology) 
and convex subset of $H^{1}(\Omega;\mathbb{R}^{m})$, it is 
also closed with respect to the weak topology, therefore $\uu\in\mathcal{A}$. 
Let now $E\subset \Omega$ be a measurable set, then we have 
\begin{equation*}
\int_{E}wdx
=\int_{\Omega}\chi_{E}wdx
=\lim_{k\to\infty}\int_{\Omega}\chi_{E}\chi_{\{\vert \uu^{k}\vert>0\}}dx 
\geq 0
\end{equation*}
by the arbitrariness of $E$ we obtain that $w\geq 0$ a.e. in $\Omega$. 
Since $\uu^{k}\to \uu$ a.e. in $\Omega$ we have 
$\chi_{\{\vert \uu^{k}\vert>0\}}\to 1$ a.e. in $\{\vert \uu\vert>0\}$. 
Let $E\subset\{\vert \uu\vert>0\}$ be a measurable set then 
\begin{multline*}
\int_{E}wdx
=\int_{\Omega}\chi_{E}wdx
=\lim_{k\to\infty}\int_{\Omega}\chi_{E}\chi_{\{\vert \uu^{k}\vert>0\}}dx \\
=\lim_{k\to\infty}\int_{\{\vert \uu\vert>0\}}\chi_{E}\chi_{\{\vert \uu^{k}\vert>0\}}dx 
=\int_{\{\vert \uu\vert>0\}}\chi_{E}dx 
=\vert E\vert 
\end{multline*}
from which by the arbitrariness of $E$ we obtain that $w=1$ a.e. in $\{\vert \uu\vert>0\}$. 
We thus have $w\geq\chi_{\{\vert \uu\vert>0\}}$ a.e. in $\Omega$, and 
\begin{multline*}
J(\uu)
=
\int_{\Omega}\bigl(\vert \nabla \uu\vert^{2}+Q^{2} \chi_{\{\vert \uu\vert>0\}}\bigr)dx 
\leq\int_{\Omega}\bigl(\vert \nabla \uu\vert^{2}+Q^{2} w\bigr)dx \\
=\int_{\Omega}\vert \nabla \uu\vert^{2}dx
+\int_{\Omega}Q^{2} w dx 
\leq 
\varliminf_{k\to\infty}\int_{\Omega}\vert \nabla \uu^{k}\vert^{2}dx
+\lim_{k\to\infty}\int_{\Omega}Q^{2} \chi_{\{\vert \uu^{k}\vert>0\}}dx \\
=
\varliminf_{k\to\infty}
\int_{\Omega}\bigl(\vert\nabla \uu^{k}\vert^{2}+Q^{2} \chi_{\{\vert \uu^{k}\vert>0\}}\bigr)dx 
=
\varliminf_{k\to\infty}
J(\uu^{k})
=
\inf_{\vv\in\mathcal{A}}J(\vv)\text{,} 
\end{multline*} 
which proves that $\uu$ is an absolute minimizer and this 
finishes the proof of the theorem. 
\end{proof}

\section{Local Minimizer}\label{section1400}   

\begin{lemma}\label{lemma1000} 
If $\uu$ is a local minimizer then $u_i$ is subharmonic for all $i=1,\cdots,m$. 
\end{lemma}
\begin{proof}
Let $\vv\in C^{1}_{c}(\Omega;\mathbb{R}^{m})$ with $v_i\geq 0$ for $i=1,\cdots,m$, and 
define $u_{\epsilon,i}(x)=(u_i(x)-\epsilon v_i(x))^{+}$ for $x\in\Omega$ and $i=1,\cdots,m$. 
Then $\uu_{\epsilon}\in\mathcal{A}$ and $\lim_{\epsilon\to 0}d(\uu,\uu_{\epsilon})=0$. 
Thus for small enough $\epsilon>0$ we have $J(\uu)\leq J(\uu_{\epsilon})$, and hence  
\begin{multline*}
\int_{\Omega}\bigl(\vert\nabla \uu\vert^{2}+Q^{2} \chi_{\{\vert \uu\vert>0\}}\bigr)dx 
\leq 
\int_{\Omega}\bigl(\vert\nabla (\uu-\epsilon\vv)^{+}\vert^{2}
+Q^{2} \chi_{\{\vert (\uu-\epsilon\vv)^{+} \vert>0\}}\bigr)dx \\
\leq 
\int_{\Omega}\bigl(\vert\nabla (\uu-\epsilon\vv)\vert^{2}+Q^{2} \chi_{\{\vert \uu\vert>0\}}\bigr)dx\text{,} 
\end{multline*}
from which it follows that 
\begin{equation*}
2\int_{\Omega}
\nabla \uu:\nabla\vv
dx
\leq 
\epsilon
\int_{\Omega}
\vert\nabla\vv\vert^{2}
dx\text{.} 
\end{equation*}
Sending $\epsilon\to0$ we obtain 
\begin{equation*}
\int_{\Omega}
\nabla \uu:\nabla\vv
dx
\leq 0
\end{equation*}
which proves that each component $u_i$ is subharmonic in $\Omega$. 
\end{proof}

Because $u_i$ is subharmonic for $i=1,\cdots,m$, for any $x\in\Omega$ the average  
$\fint_{B_r(x)}u_i dy$ is nondecreasing in $r$ (for small $r>0$). Thus the limit 
$\lim_{r\to 0}\fint_{B_r(x)}u_i dy$ 
exists. Because this limit is equal to $u_i(x)$ for a.e. $x\in\Omega$ we might choose a version of $u_i$ such that 
$u_i(x)=\lim_{r\to 0}\fint_{B_r(x)}u_i dy$ 
for all $x\in\Omega$. 

Also $u_i$ being subharmonic by maximum principle we have that $u_i(x)\leq\sup_{\partial\Omega}g_i$ for all $x\in\Omega$. 
Because the averages $\fint_{B_r(x)}u_i dy$ are continuous functions of $x$ 
and $u_i(x)=\inf_{r>0}\fint_{B_r(x)}u_i dy$ we have that $u_i$ is upper-semicontinuous 
in $\Omega$. 

\begin{definition}\label{definition1100} 
Let $B_{r_0}(x)\subset\Omega$ then for $0<r<r_0$ we define 
\begin{equation*}
\uu_{x,r}(y)=\frac{1}{r}\uu(x+ry)
\end{equation*}
and call $\uu_{x,r}$ the linear blowup of $\uu$ at $x$. 
In the case $x=0$ we denote $\uu_r=\uu_{0,r}$. 
We define further
\begin{equation*}
Q_{x,r}(y)=Q(x+ry)\text{.} 
\end{equation*}
In the case $x=0$ we set $Q_{r}(y)=Q_{0,r}(y)$. 
\end{definition}

\begin{lemma}[Initial regularity of minimizers]\label{lemma1050} 
Let $\uu$ be a local minimizer. Then for each compact subset $K\subset\subset\Omega$ there 
exists $C>0$ (depending on $K$ and $Q_{max}$) such that 
\begin{equation*}
\vert \uu(x)-\uu(y)\vert\leq C\vert x-y\vert\ln(\frac{1}{\vert x-y\vert}) ,
\end{equation*}
for $x,y\in K$ and $\vert x-y\vert<\frac{1}{2}$. 
In particular $\uu\in C^{\beta}_{loc}(\Omega;\mathbb{R}^{m})$ 
for all $0<\beta<1$. 
\end{lemma}
\begin{proof}
Let $B_r(x)\subset\Omega$. 
Let for $i=1,\cdots,m$, $v_i$ be the harmonic function in $B_r(x)$ such that 
$v_i=u_i$ on $\partial B_r(x)$. Let us extend $\vv=(v_1,\cdots,v_m)$ by $\uu$ in $\Omega\backslash B_r(x)$. 
We have $J(\uu)\leq J(\vv)$ (in the case of a local minimizer, $r$ should be small enough). 
It follows that 
\begin{multline*}
\int_{B}\bigl(\vert\nabla \uu_{x,r}\vert^{2}+Q_{x,r}^{2}\chi_{\{\vert \uu_{x,r}\vert>0\}}\bigr)dy 
\leq\int_{B}\bigl(\vert\nabla \vv_{x,r}\vert^{2}
+Q_{x,r}^{2}\chi_{\{\vert \vv_{x,r}\vert>0\}}\bigr)dy \\
=\int_{B}\bigl(\vert\nabla \vv_{x,r}\vert^{2}+Q_{x,r}^{2}\bigr)dy 
\end{multline*}
therefore 
\begin{multline}\label{equ1200} 
\int_{B}\bigl(\vert\nabla \uu_{x,r}\vert^{2}-\vert\nabla \vv_{x,r}\vert^{2}\bigr)dy
\leq 
\int_{B}
Q_{x,r}^2
\chi_{\{\vert \uu_{x,r}\vert=0\}}dy \\
\leq 
Q_{max}^2
\bigl\vert B\cap\{\vert \uu_{x,r}\vert=0\}\bigr\vert\text{.} 
\end{multline}

For each $i=1,\cdots,m$ we compute 
\begin{multline}\label{equ1300} 
\int_{B}\vert\nabla (u_{x,r,i}-v_{x,r,i})\vert^{2}dy \\
=
\int_{B}\bigl(
\vert\nabla u_{x,r,i}\vert^{2}
-2\nabla (u_{x,r,i}-v_{x,r,i})\cdot\nabla v_{x,r,i} 
-\vert\nabla v_{x,r,i}\vert^{2}
\bigr)dy \\
=
\int_{B}\bigl(
\vert\nabla u_{x,r,i}\vert^{2}
-\vert\nabla v_{x,r,i}\vert^{2}
\bigr)dy\text{.} 
\end{multline}
From \eqref{equ1200} and \eqref{equ1300} we obtain 
\begin{equation}\label{equ1400} 
\int_{B}\vert\nabla (\uu_{x,r}-\vv_{x,r})\vert^{2}dy
\leq
Q_{max}^{2}
\bigl\vert B\cap\{\vert \uu_{x,r}\vert=0\}\bigr\vert\text{.} 
\end{equation}
Thus we have 
\begin{equation*}
\int_{B}\vert\nabla (\uu_{x,r}-\vv_{x,r})\vert^{2}dy
\leq
Q_{max}^{2}\vert B\vert\text{.} 
\end{equation*}
It follows that for $i=1,\cdots,m$ we have separately for each component 
\begin{equation*}
\int_{B}\vert\nabla (u_{x,r,i}-v_{x,r,i})\vert^{2}dy
\leq
Q_{max}^{2}\vert B\vert\text{.} 
\end{equation*}

Proceeding as in Theorem 2.1 of \cite{AltCaffarelliFriedman1984b} we compete the proof of the Lemma. 
\end{proof}

\section{Lipschitz Regularity of Minimizers\\
(Proof of Theorem \ref{theorem1100})}\label{section1500}
Let $\psi\in H^{1}(B)$ then for $A\subset B$ we define 
\begin{equation*}
\operatorname{Cap}(A;B,\psi)
=\inf_{u\in H^{1}_{0}(B),\ u\geq \psi\chi_{A}}\int_{B}\vert\nabla u\vert^{2}dx\text{.} 
\end{equation*}

\begin{lemma}\label{lemma1100} 
There exists $c>0$ such that if 
$u\in H^{1}(B)$ is nonnegative and 
$v$ is the harmonic function in $B$ with $v=u$ on $\partial B$ 
we have 
\begin{equation*}
c\operatorname{Cap}\bigl(B\cap\{u=0\};B,1-\vert x\vert\bigr)
\bigl(
\fint_{\partial B}ud\sigma(y)
\bigr)^{2}
\leq\int_{B}\vert \nabla (v-u)\vert^{2}dx\text{.} 
\end{equation*}
\end{lemma}
\begin{proof}
By minimum principle we have that $v\geq 0$ in $B$. 
Let us denote $h=v-u$ then $h=0$ on $\partial B$ and $h=v$ on $B\cap\{u=0\}$. 
Using Poisson formula for unit ball 
there exists a dimensional constant $c_1>0$ such that for $x\in B$ 
\begin{equation*}
v(x)\geq c_{1}(1-\vert x\vert)
\fint_{\partial B}vd\sigma(y)\text{.} 
\end{equation*}
Thus for $x\in B\cap\{u=0\}$ we have 
\begin{equation*}
h(x)\geq c_{1}(1-\vert x\vert)
\fint_{\partial B}ud\sigma(y)
\end{equation*}
since $v=u$ on $\partial B$. 
Let us define 
\begin{equation*}
\tilde h(x)=\Bigl(
c_1\fint_{\partial B}ud\sigma(y)
\Bigr)^{-1}h(x)\text{.} 
\end{equation*}
Then we have $\tilde h=0$ on $\partial B$ and $\tilde h\geq 1-\vert x\vert$ 
on $B\cap\{u=0\}$. Therefore 
\begin{multline*}
\operatorname{Cap}\bigl(B\cap\{u=0\};B,1-\vert x\vert\bigr)
\leq 
\int_{B}\vert\nabla\tilde h\vert^{2}dx \\
=
\Bigl(c_1\fint_{\partial B}ud\sigma(y)\Bigr)^{-2}
\int_{B}\vert\nabla h\vert^{2}dx
\end{multline*}
and this proves the lemma. 
\end{proof}

\begin{lemma}\label{lemma1200} 
There exists $c>0$ such that 
for $\uu$ a (local) minimizer and $B_r(x)\subset\Omega$ (small balls) we have 
\begin{equation*}
\frac{c}{Q_{max}^2}\operatorname{Cap}\bigl(B\cap\{ \vert \uu_{x,r}\vert=0\};B,1-\vert x\vert\bigr)
\bigl\vert
\fint_{\partial B} \uu_{x,r} d\sigma(y)
\bigr\vert_{2}^{2} 
\leq 
\bigl\vert B\cap\{\vert \uu_{x,r}\vert=0\}\bigr\vert\text{.} 
\end{equation*}
\end{lemma}
\begin{proof}
For $B_r(x)\subset\Omega$ 
and $i=1,\cdots,m$, let $v_i$ be the harmonic function in $B_r(x)$ such that 
$v_i=u_i$ on $\partial B_r(x)$. Extending $\vv=(v_1,\cdots,v_m)$ by $\uu$ 
into $\Omega\backslash B_r(x)$ we have $J(\uu)\leq J(\vv)$ (in the case of local minimizer $r$ should be small enough). 
Proceeding as in the proof of Lemma \ref{lemma1050} we obtain \eqref{equ1400}. 

By Lemma \ref{lemma1100} and using $\{\vert \uu\vert=0\}\subset \{u_i=0\}$ 
for $i=1,\cdots,m$ we obtain 
\begin{multline*}
c
\operatorname{Cap}\bigl(B\cap\{ \vert \uu_{x,r}\vert=0\};B,1-\vert x\vert \bigr)
\bigl\vert 
\fint_{\partial B}\uu_{x,r} d\sigma(y)
\bigr\vert_{2}^{2} \\
\leq 
c
\sum_{i=1}^{m}
\operatorname{Cap}\bigl(B\cap\{ u_{x,r,i}=0\};B,1-\vert x\vert \bigr)
\bigl(
\fint_{\partial B} u_{x,r,i} d\sigma(y)
\bigr)^{2} \\
\leq
\int_{B}\vert\nabla(\uu_{x,r}-\vv_{x,r})\vert^2 dy
\leq 
Q_{max}^2
\bigl\vert B\cap\{\vert \uu_{x,r}\vert=0\}\bigr\vert
\end{multline*}
and this proves the lemma. 
\end{proof}

\begin{lemma}\label{lemma1300}
For $A\subset B$ a Borel set we have 
\begin{equation*}
\vert A\vert\leq\operatorname{Cap}\bigl(A;B,1-\vert x\vert\bigr)\text{.} 
\end{equation*}
\end{lemma}
\begin{proof}
Let $v\in H^{1}_{0}(B)$, $v\geq (1-\vert x\vert)\chi_{A}$ a.e. in $B$ 
such that 
\begin{equation*}
\int_{B}\vert \nabla v\vert^{2} dy
=\operatorname{Cap}\bigl(A;B,1-\vert x\vert\bigr)\text{.} 
\end{equation*}

We claim that 
\begin{equation}\label{equ1410} 
v\leq 1-\vert x\vert\enskip\text{a.e. in}\enskip B\text{.}
\end{equation}

Let us denote $w=1-\vert x\vert$. 
Assume \eqref{equ1410} does not hold, then  
\begin{equation*}
\bigl\vert \bigl\{x\in B\bigm\vert v>w\bigr\}\bigr\vert>0\text{.} 
\end{equation*}
Let us define $u=(v-w)^{+}$, then we have $\vert \{u>0\}\vert>0$. 
Because $u\in H^{1}_{0}(B)$ using Poincar\'e inequality we obtain   
\begin{equation}\label{equ1420} 
0<\int_{B}u^{2} dx\leq C\int_{B}\vert \nabla u\vert^{2} dx\text{.} 
\end{equation}
Because $v$ is the minimizer in the definition of capacity, we have 
\begin{multline*}
\int_{B}\vert \nabla v\vert^{2}dx
\leq 
\int_{B}\vert \nabla \min(v,w)\vert^{2}dx \\
=
\int_{B}\chi_{\{v< w\}}\vert \nabla v\vert^{2}dx
+
\int_{B}\chi_{\{v> w\}}\vert \nabla w\vert^{2}dx
\end{multline*}
and thus 
\begin{equation}\label{equ1430} 
\int_{B}\chi_{\{v>w\}}\bigl(\vert \nabla w\vert^{2}-\vert \nabla v\vert^2\bigr)dx
\geq 0\text{.} 
\end{equation}

Next, invoking \eqref{equ1430}, and the fact that $\Delta w\leq 0$ 
in $H^{-1}(B)$, we obtain  
\begin{multline*}
\int_{B}\vert\nabla u\vert^{2}dx
=
\int_{B}\chi_{\{v>w\}}\vert\nabla (v-w)\vert^{2}dx \\
=
\int_{B}\chi_{\{v>w\}}
\bigl(
\vert\nabla v\vert^{2}
-
2\nabla v\cdot\nabla w
+
\vert\nabla w\vert^{2}
\bigr)
dx \\
=
\int_{B}\chi_{\{v>w\}}
\bigl(
\vert\nabla v\vert^{2}
-
2\nabla (v-w)\cdot\nabla w
-
\vert\nabla w\vert^{2}
\bigr)
dx \\
=
\int_{B}\chi_{\{v>w\}}
\bigl(
\vert\nabla v\vert^{2}
-
\vert\nabla w\vert^{2}
\bigr)
dx 
-2
\int_{B}
\nabla u\cdot\nabla w
dx 
\leq 
0
\end{multline*}
which is in contradiction with \eqref{equ1420}. 
This contradiction proves  claim \eqref{equ1410}. Now by \eqref{equ1410}, and that   a.e. in $A$ we have $v\geq 1-\vert x\vert$, 
we obtain $v=1-\vert x\vert$ a.e. in $A$, thus 
\begin{equation*}
\int_{B}\vert \nabla v\vert^{2} dy
\geq 
\int_{A}\vert \nabla v\vert^{2} dy
=
\int_{A}\vert \nabla (1-\vert x\vert)\vert^{2} dy
=\vert A\vert .
\end{equation*}
This proves the lemma. 
\end{proof}

\begin{lemma}\label{lemma1400} 
Let $\uu$ be a (local) minimizer. If 
$B_r(x)\cap\{\vert \uu\vert=0\}\not=\emptyset$ 
then $\vert B_r(x)\cap\{\vert \uu\vert=0\}\vert>0$. 
\end{lemma}

\begin{proof}
Assume by contradiction that $\vert B_r(x)\cap\{\vert \uu\vert=0\}\vert=0$. 
Then $\vert \uu\vert>0$ a.e. in $B_r(x)\cap\Omega$. 

Let $y\in B_{r}(x)\cap\Omega$ and $\eta>0$ small enough 
such that $B_{\eta}(y)\subset B_r(x)\cap\Omega$. 
Let $\vv$ be harmonic in $B_{\eta}(y)$ and $\vv=\uu$ on $\partial B_{\eta}(y)$. 
Extend $\vv$ by $\uu$ in $\Omega\backslash B_{\eta}(y)$. 
Choosing $\eta$ small enough, $\vv$ might be made arbitrarily close to 
$\uu$ in metric $d$, see \eqref{equ400}. 
Thus for small enough $\eta$ we have that $J(\uu)\leq J(\vv)$. 
We have $\chi_{\{\vert \uu\vert>0\}}=1$ a.e. in $B_{\eta}(y)$. 

It follows that 
\begin{equation*}
\int_{B_{\eta}(y)}\vert\nabla \uu\vert^{2}dx\leq\int_{B_{\eta}(y)}\vert\nabla \vv\vert^{2}dx ,
\end{equation*}
and thus  $\uu$ is a minimizer of Dirichlet energy in $B_{\eta}(y)$, with its own 
trace as boundary condition, in $B_{\eta}(y)$. 
Hence $\uu$ is harmonic in a neighborhood of any $y\in B_{r}(x)\cap\Omega$. 
It follows that $\uu$ is harmonic in $B_r(x)\cap\Omega$. 
From strong minimum principle it follows that for each component $u_i$, either 
$u_i=0$ in $B_{r}(x)\cap\Omega$ or $u_i>0$ in $B_{r}(x)\cap\Omega$. 
Now because $\vert \uu\vert>0$ a.e. in $B_{r}(x)\cap\Omega$ we obtain that 
there exists $i_0$ such that $u_{i_0}>0$ in $B_{r}(x)\cap\Omega$. 
But this contradicts with $B_r(x)\cap\{\vert \uu\vert=0\}\not=\emptyset$. 
\end{proof}

\begin{proof}[Proof of Theorem \ref{theorem1100}]
Let $\epsilon>0$ be small enough such that $B_{r+\epsilon}(x)\subset\subset\Omega$. 
By Lemma \ref{lemma1200} and \ref{lemma1300} we have 
\begin{equation*}
c\bigl\vert B\cap\{ \vert \uu_{x,r+\epsilon}\vert=0\}\bigr\vert 
\bigl\vert
\fint_{\partial B} \uu_{x,r+\epsilon} d\sigma(y)
\bigr\vert_{2}^{2} 
\leq 
Q_{max}^{2}
\bigl\vert B\cap\{\vert \uu_{x,r+\epsilon}\vert=0\}\bigr\vert , 
\end{equation*}
and by Lemma \ref{lemma1400} 
$$\vert B\cap\{\vert \uu_{x,r+\epsilon}\vert=0\}\vert=(r+\epsilon)^{-n}\vert B_{r+\epsilon}(x)\cap\{\vert \uu\vert=0\}\vert>0.$$ 
Therefore  
\begin{equation*}
c
\bigl\vert
\fint_{\partial B} \uu_{x,r+\epsilon} d\sigma(y)
\bigr\vert_{2}^{2} 
\leq 
Q_{max}^2
\end{equation*}
which by letting $\epsilon\to0$  proves the theorem. 
\end{proof} 

\begin{corollary}\label{corollary1000} 
There exists $C>0$ such that for $\uu$ a (local) minimizer and 
$B_{r}(x)\subset\Omega$ (small enough) such that $\uu(x)=0$ we have 
\begin{equation*} 
[\uu]_{C^{0,1}(B_{\frac{r}{3}}(x))}\leq CQ_{max}\text{.}
\end{equation*}
\end{corollary}
\begin{proof}
We should show that 
\begin{equation*}
\vert \uu(x_2)-\uu(x_1)\vert\leq C Q_{max} \vert x_2-x_1\vert\enskip\text{for}\enskip x_1,x_2\in B_{\frac{r}{3}}(x)\text{.} 
\end{equation*}

Let us denote by $[x_1,x_2]$ the line segment connecting $x_1$ and $x_2$. 
We consider two cases depending on whether $[x_1,x_2]\cap\{\vert \uu\vert=0\}$ is empty or not. 

If $[x_1,x_2]\cap\{\vert \uu\vert=0\}=\emptyset$ 
then for all $z\in [x_1,x_2]$ we have $\eta=d(z,\{\vert \uu\vert=0\})>0$. 
Because $x\in\{\vert \uu\vert=0\}$ we have $\eta<\frac{r}{3}$. 
We compute 
\begin{equation*}
\overline{B_{\eta}(z)}
\subset\overline{B_{\eta+\frac{r}{3}}(x)}
\subset B_{\frac{2}{3}r}(x)
\subset\Omega\text{.}
\end{equation*}
We have that $\uu$ is harmonic in $B_{\eta}(z)$. 
Because $\overline{B_{\eta}(z)}\cap\{\vert \uu\vert=0\}\not=\emptyset$ by 
Theorem \ref{theorem1100} and Poisson representation formula we have for $i=1,\cdots,m$ 
\begin{equation*}
\vert\nabla u_i(z)\vert
\leq 
\frac{C_{1}}{\eta}\fint_{\partial B_{\eta}(z)}u_id\sigma
\leq 
C_{1}C Q_{max} \text{.} 
\end{equation*}
Because this holds for all $z\in[x_1,x_2]$ 
by mean value theorem this proves the claim in this case. 

If $[x_1,x_2]\cap \{\vert \uu\vert=0\}\not=\emptyset$ then 
there exists $z\in [x_1,x_2]$ such that $\uu(z)=0$. 
For $k=1,2$, $z\in\overline{B_{\vert x_2-x_1\vert}(x_{k})}\cap\{\vert \uu\vert=0\}$. 
We have also 
\begin{equation*}
\overline{B_{\vert x_2-x_1\vert}(x_{k})}
\subset 
\overline{B_{\vert x_2-x_1\vert+\vert x_{k}-x\vert}(x)}
\subset 
B_{\frac{2}{3}r+\frac{r}{3}}(x)
=
B_{r}(x)\subset\Omega\text{.} 
\end{equation*}
Because for $i\in\{1,\cdots,m\}$, $u_i$ is subharmonic using 
Theorem \ref{theorem1100} we obtain 
\begin{equation}\label{equ1500} 
u_{i}(x_k)
\leq 
\fint_{\partial B_{\vert x_{2}-x_{1}\vert}(x_k)} u_i d\sigma
\leq 
C Q_{max} \vert x_{2}-x_{1}\vert
\text{.} 
\end{equation}

Using \eqref{equ1500} we have  
\begin{equation*}
\vert \uu(x_2)-\uu(x_1)\vert 
\leq 
\vert \uu(x_2)\vert+\vert \uu(x_1)\vert \\
\leq 
C_{1} Q_{max} \vert x_{2}-x_{1}\vert
\end{equation*}
which  completes the proof of the corollary. 
\end{proof}

\section{Preliminary Local Analysis\\
(Proof of Theorem \ref{theorem1200})}\label{section1600} 

\subsection{Nondegeneracy}\label{subsection1000} 

Let us define for $\eta>0$ 
\begin{equation*}
\phi(\eta)
=
\left\{
\begin{aligned}
& -\frac{1}{n-2}\frac{1}{\eta^{n-2}}\enskip\text{for}\enskip n\geq 3\text{,} \\
& \ln(\eta)\enskip\text{for}\enskip n=2\text{.} 
\end{aligned}
\right.
\end{equation*}
Then $\phi(\vert x\vert)$ as a function of $x$ is 
radially symmetric, radially increasing and harmonic 
function in $\mathbb{R}^{n}\backslash\{0\}$ (constant multiple 
of the fundamental solution). 

For $0<\rho<1$  we define,  
\begin{equation*}
\psi_{\rho}(x)=\frac{\bigl(\phi(\vert x\vert)-\phi(\rho)\bigr)^{+}}{\phi(1)-\phi(\rho)} ,  
\end{equation*}
which is  radially symmetric, radially nondecreasing, varnishes in $B_{\rho}$, is
harmonic in $\mathbb{R}^{n}\backslash\overline B_{\rho}$, and equals to $1$ on $\partial B$. This will be used in the text below.

\begin{proof}[Proof of Theorem \ref{theorem1200}] 
For $0<\rho<1$ we define 
\begin{equation}\label{equ1600} 
v_i(y)=\min\bigl(u_i(y),rM_{x,r}\psi_{\rho}(\frac{y-x}{r})
\bigr)\enskip\text{for}
\enskip i=1,\cdots,m\enskip\text{and}\enskip y\in B_r(x)\text{,} 
\end{equation}
where $M_{x,r}=\frac{1}{r}\sup_{B_{r}(x)}\vert \uu\vert$ 
and we extend $\vv=(v_1,\cdots,v_m)$ by $\uu$ 
in $\Omega\backslash B_r(x)$. 

In the case $\uu$ is a local minimizer 
we should also have that $\vv$ is close enough 
in the metric $d$, see \eqref{equ400}, to $\uu$.
In the case $n\geq 3$ by choosing $r$ small enough 
we might achieve this. In the case $n=2$ by choosing 
both $r$ and $\rho$ small enough we achieve this. 
Thus we have $J(\uu)\leq J(\vv)$, therefore 
\begin{equation*}
\int_{B}\bigl(
\vert\nabla \uu_{x,r}\vert^{2}
+
Q_{x,r}^2
\chi_{\{\vert \uu_{x,r}\vert>0\}}
\bigr)dy
\leq 
\int_{B}\bigl(
\vert\nabla \vv_{x,r}\vert^{2}
+
Q_{x,r}^2
\chi_{\{\vert \vv_{x,r}\vert>0\}}
\bigr)dy\text{.}
\end{equation*}

Since $\vv_{x,r}=0$ in $B_{\rho}$ 
and $\{\vert \vv_{x,r}\vert>0\}=\{\vert \uu_{x,r}\vert>0\}$ in $B\backslash B_{\rho}$ 
we have 
\begin{multline}\label{equ1700} 
\int_{B_{\rho}}\bigl(
\vert\nabla \uu_{x,r}\vert^{2}
+
Q_{x,r}^2
\chi_{\{\vert \uu_{x,r}\vert>0\}}
\bigr)dy \\
=
\int_{B}\bigl(
\vert\nabla \uu_{x,r}\vert^{2}
+
Q_{x,r}^2
\chi_{\{\vert \uu_{x,r}\vert>0\}}
\bigr)dy \\
-
\int_{B\backslash B_{\rho}}\bigl(
\vert\nabla \uu_{x,r}\vert^{2}
+
Q_{x,r}^2
\chi_{\{\vert \uu_{x,r}\vert>0\}}
\bigr)dy \\
\leq 
\int_{B}\bigl(
\vert\nabla \vv_{x,r}\vert^{2}
+
Q_{x,r}^2
\chi_{\{\vert \vv_{x,r}\vert>0\}}
\bigr)dy \\
-
\int_{B\backslash B_{\rho}}\bigl(
\vert\nabla \uu_{x,r}\vert^{2}
+
Q_{x,r}^2
\chi_{\{\vert \uu_{x,r}\vert>0\}}
\bigr)dy \\
=
\int_{B\backslash B_{\rho}}\bigl(
\vert\nabla \vv_{x,r}\vert^{2}
+
Q_{x,r}^2
\chi_{\{\vert \vv_{x,r}\vert>0\}}
\bigr)dy \\
-
\int_{B\backslash B_{\rho}}\bigl(
\vert\nabla \uu_{x,r}\vert^{2}
+
Q_{x,r}^2
\chi_{\{\vert \uu_{x,r}\vert>0\}}
\bigr)dy \\
= 
\int_{B\backslash B_{\rho}}
\bigl(
\vert\nabla \vv_{x,r}\vert^{2}
-\vert\nabla \uu_{x,r}\vert^{2}
\bigr)
dy \\
=
\sum_{i=1}^{m}
\int_{B\backslash B_{\rho}}
\bigl(
\vert\nabla v_{x,r,i}\vert^{2}
-\vert\nabla u_{x,r,i}\vert^{2}
\bigr)
dy\text{,} 
\end{multline}
and for each $i=1,\cdots,m$  
\begin{multline}\label{equ1800} 
\int_{B\backslash B_{\rho}}
\bigl(\vert\nabla v_{x,r,i}\vert^{2}
-\vert\nabla u_{x,r,i}\vert^{2}
\bigr)dy \\
=
\int_{B\backslash B_{\rho}}
\Bigl(
-2\nabla (M_{x,r}\psi_{\rho})
\cdot \nabla (u_{x,r,i}-M_{x,r}\psi_{\rho})^{+} 
-\vert \nabla (u_{x,r,i}-M_{x,r}\psi_{\rho})^{+}\vert^{2}
\Bigr)dy \\
\leq 
-2
\int_{B\backslash B_{\rho}}
\nabla (M_{x,r}\psi_{\rho})
\cdot \nabla (u_{x,r,i}-M_{x,r}\psi_{\rho})^{+} 
dy \\
=
2M_{x,r}
\int_{\partial B_{\rho}}
u_{x,r,i} 
\partial_{\nu}\psi_{\rho}
d\sigma(y) 
=
C_{\rho}M_{x,r}
\int_{\partial B_{\rho}}
u_{x,r,i} 
d\sigma(y)\text{.} 
\end{multline}

Putting \eqref{equ1700} and \eqref{equ1800} together we obtain 
\begin{multline}\label{equ1900} 
\int_{B_{\rho}}\bigl(
\vert\nabla \uu_{x,r}\vert^{2}+Q_{x,r}^2\chi_{\{\vert \uu_{x,r}\vert>0\}}
\bigr)dy
\leq 
C_{\rho} M_{x,r}
\int_{\partial B_{\rho}}
\sum_{i=1}^{m}u_{x,r,i}
d\sigma(y) \\
\leq 
C_{1} C_{\rho}  M_{x,r}
\int_{\partial B_{\rho}}\vert \uu_{x,r}\vert d\sigma(y)\text{.} 
\end{multline}

It is easy to check that for each $0<\rho<1$ there exists $\tilde C_{\rho}>0$ such that 
for all $w\in W^{1,1}(B_{\rho})$ we have 
\begin{equation}\label{equ1910} 
\int_{\partial B_{\rho}}\vert w\vert d\sigma(y)
\leq 
\tilde C_{\rho}
\int_{B_{\rho}}\bigl(
\vert w\vert+\vert\nabla w\vert
\bigr)dy\text{.} 
\end{equation}

Using \eqref{equ1910} we estimate 
\begin{multline}\label{equ2000} 
\int_{\partial B_{\rho}} \vert \uu_{x,r}(y)\vert d\sigma(y) 
\leq 
\tilde C_{\rho}\int_{B_{\rho}} \bigl(\vert \uu_{x,r}\vert+\vert\nabla \vert \uu_{x,r}\vert \vert\bigr)dy \\
= 
\tilde C_{\rho}
\int_{B_{\rho}} \bigl(\vert \uu_{x,r}\vert \chi_{\{\vert \uu_{x,r}\vert>0\}} 
+\vert\nabla \vert \uu_{x,r}\vert \vert \chi_{\{\vert \uu_{x,r}\vert>0\}} \bigr)dy \\
\leq 
\tilde C_{\rho}
\int_{B_{\rho}} \bigl( M_{x,r} \chi_{\{\vert \uu_{x,r}\vert>0\}} +\vert\nabla \uu_{x,r} \vert 
\chi_{\{\vert \uu_{x,r}\vert>0\}} \bigr) dy \\
\leq 
\tilde C_{\rho}
\int_{B_{\rho}} \bigl( M_{x,r} \chi_{\{\vert \uu_{x,r}\vert>0\}} 
+
\frac{1}{2}
Q_{min}
\chi_{\{\vert \uu_{x,r}\vert>0\}} 
+
\frac{1}{2Q_{min}}
\vert\nabla \uu_{x,r} \vert^{2} 
\bigr) dy \\
= 
\tilde C_{\rho}
\Bigl(
M_{x,r}
\int_{B_{\rho}}\chi_{\{\vert \uu_{x,r}\vert>0\}} dy
+
\frac{1}{2Q_{min}}
\int_{B_{\rho}} 
\bigl( 
Q_{min}^2
\chi_{\{\vert \uu_{x,r}\vert>0\}} 
+
\vert\nabla \uu_{x,r} \vert^{2} 
\bigr) dy
\Bigr) \\
\leq  
\tilde C_{\rho}
\Bigl(
\frac{M_{x,r}}{Q_{min}^2}
\int_{B_{\rho}}Q^{2}_{x,r}\chi_{\{\vert \uu_{x,r}\vert>0\}} dy \\
+
\frac{1}{2Q_{min}}
\int_{B_{\rho}} 
\bigl( 
Q_{x,r}^2
\chi_{\{\vert \uu_{x,r}\vert>0\}} 
+
\vert\nabla \uu_{x,r} \vert^{2} 
\bigr) dy
\Bigr) \\
\leq 
\tilde C_{\rho}
\frac{1}{Q_{min}}
\bigl( 
\frac{M_{x,r}}{Q_{min}}+\frac{1}{2}
\bigr)
\int_{B_{\rho}} 
\bigl( 
Q_{x,r}^2\chi_{\{\vert \uu_{x,r}\vert>0\}} 
+\vert\nabla \uu_{x,r} \vert^{2} 
\bigr) dy
\text{.} 
\end{multline} 

Because $\overline{B_{\rho r}(x)}\cap\{\vert \uu\vert>0\}\not=\emptyset$ we have 
\begin{equation}\label{equ2100} 
\int_{\partial B_{\rho}} \vert \uu_{x,r}\vert d\sigma(y)>0\text{.} 
\end{equation}

From \eqref{equ1900}, \eqref{equ2000} and \eqref{equ2100} we obtain that 
\begin{equation*} 
1
\leq 
C_1 C_{\rho} \tilde C_{\rho}
\frac{M_{x,r}}{Q_{min}}
\bigl(
\frac{M_{x,r}}{Q_{min}}
+
\frac{1}{2}
\bigr) ,
\end{equation*}
 which in turn implies  
$M_{x,r}\geq\min(1,\frac{2}{3}\frac{1}{C_{1} C_{\rho} \tilde C_{\rho}}) Q_{min}$. This   completes
  the proof.  
\end{proof}

\begin{remark}\label{remark1000} 
As noted in the proof of Theorem \ref{theorem1200}, 
in the case of a local minimizer $\uu$ and $n=2$, to have 
$\vv$ close enough in metric $d$, see \eqref{equ400}, to $\uu$ 
we should choose 
both $r$ and $\rho$ small enough. 

To see that this is necessary 
(with $\vv$ as in \eqref{equ1600}) 
one might consider $u=1$, in the scalar case. 
Then we have 
\begin{equation*}
\int_{B_r(x)}\vert \nabla (u-v)\vert^{2}dy=\int_{B}\vert \nabla \psi_{\rho}\vert^2 dz
\end{equation*}
where the right hand side is independent of $r$, but converges to $0$ as $\rho\to 0$.

\end{remark}

\subsection{Density of $\{\vert\uu\vert=0\}$ and $\{\vert\uu\vert>0\}$}
\label{subsection1100} 
\begin{lemma}\label{lemma1500} 
For any (local) 
minimizer $\uu$, (small ball) $B_{r}(x)\subset\Omega$ and 
$x\in\overline{\{\vert\uu\vert>0\}}$ there exists $x_0\in\partial B_{\frac{r}{2}}(x)$ such that 
$B_{\kappa r}(x_0)\subset\{\vert u\vert>0\}$ where 
$\kappa=\frac{c}{2}\frac{Q_{min}}{Q_{max}}$ and $0<c<1$ is universal 
except in the case when $u$ is a local minimizer and $n=2$, in which case $c$ depends on $u$. 
\end{lemma}
\begin{proof}
Let $0<\rho<1$. Because $x\in\overline{\{\vert\uu\vert>0\}}$ we have 
$\overline{B_{\frac{\rho}{2}r}(x)}\cap\{\vert\uu\vert>0\}\not=\emptyset$ 
thus by Theorem \ref{theorem1200} we have  
\begin{equation*}
\sup_{B_{\frac{r}{2}}(x)}\vert\uu\vert 
\geq  
c_{\rho}Q_{min}\frac{r}{2}\text{.} 
\end{equation*}
where in the case when $u$ is a local minimizer, $r$ should be small enough 
and in the case when $n=2$ additionally $\rho$ should be small enough.

Now because $\vert\uu\vert$ is a subharmonic function there exists 
$x_{0}\in\partial B_{\frac{r}{2}}(x)$ such that 
\begin{equation}\label{equ2200} 
\vert\uu(x_{0})\vert\geq c_{\rho}Q_{min}\frac{r}{2}\text{.}
\end{equation} 

For $0<\kappa<\frac{1}{2}$ we have 
\begin{equation*}
\overline{B_{\kappa r}(x_{0})}
\subset 
\overline{
B_{\kappa r+\vert x-x_{0}\vert}(x)
}
=
\overline{
B_{\kappa r+\frac{r}{2}}(x)
}
\subset 
B_{r}(x)
\subset\Omega\text{.} 
\end{equation*}

We claim that for small enough $0<c_{1}<1$ and $\kappa=\frac{c_1}{2}\frac{Q_{min}}{Q_{max}}$ we have that $\vert\uu\vert>0$ in 
$\overline{B_{\kappa r}(x_{0})}$. 

Assume this is not the case, then we have 
$\overline{B_{\kappa r}(x_{0})}\cap\{\vert\uu\vert=0\}\not=\emptyset$ 
thus by Theorem \ref{theorem1100} we would have  
\begin{equation}\label{equ2300} 
\fint_{\partial B_{\kappa r}(x_{0})} u_i d\sigma
\leq 
CQ_{max}\kappa r\enskip\text{for}\enskip i=1,\cdots,m\text{.} 
\end{equation}
Because $u_i$ is subharmonic we have 
\begin{equation}\label{equ2400} 
u_i(x_{0})
\leq 
\fint_{\partial B_{\kappa r}(x_{0})} u_i d\sigma\text{.} 
\end{equation}
From \eqref{equ2200}, \eqref{equ2300} and \eqref{equ2400} we obtain that 
\begin{equation*} 
c_{\rho}Q_{min}\frac{r}{2}
\leq 
\vert\uu(x_{0})\vert
\leq 
\bigl\vert
\fint_{\partial B_{\kappa r}(x_{0})}\uu d\sigma
\bigr\vert
\leq 
CQ_{max}\sqrt{m}\kappa r\text{.}
\end{equation*} 
By choosing $c_1=\min(\frac{1}{2},\frac{c_{\rho}}{C\sqrt{m}})$ we arrive at a contradiction. 
\end{proof}

\begin{lemma}\label{lemma1600} 
For any (local) minimizer $\uu$ and (small) balls 
$B_r(x)\subset\Omega$ with $x\in\partial\{\vert\uu\vert>0\}$ we have 
\begin{equation*}
\vert B_r(x)\cap\{\vert\uu\vert=0\}\vert\geq c(\frac{Q_{min}}{Q_{max}})^{n+2}
\vert B_r\vert
\end{equation*}
where $0<c<1$ is universal 
except in the case when $u$ is a local minimizer and $n=2$, in which case $c$ depends on $u$. 
\end{lemma}
\begin{proof}
By the Poincar\'e inequality in $B_r(x)$ and the inequality \eqref{equ1400} 
proved in Lemma \ref{lemma1050} we have 
\begin{multline}\label{equ2500} 
\int_{B_r(x)}\vert\uu-\vv\vert^{2}dy
\leq 
Cr^2\int_{B_r(x)}\vert\nabla (\uu-\vv)\vert^{2}dy \\
\leq  Cr^2  Q_{max}^{2}\vert B_r(x)\cap\{\vert\uu\vert=0\}\vert
\end{multline}
where $\vv$ is as in the proof of Lemma \ref{lemma1050}. 
Since $\overline{B_{\frac{\rho}{2}r}(x)}\cap\{\vert \uu\vert>0\}\not=\emptyset$, 
by Theorem \ref{theorem1200} we have  
\begin{equation}\label{equ2700} 
c_{\rho}Q_{min}\frac{r}{2}
\leq 
\sup_{B_{\frac{r}{2}}(x)}\vert \uu\vert 
\leq 
C_{1}
\fint_{\partial B_{r}(x)}\vert\uu\vert d\sigma
\text{,} 
\end{equation}
where for the last inequality 
we have used the Poisson representation for the subharmonic function $\vert\uu\vert$. 

Let $c_{2}>0$ such that $c_{2}\vert z\vert_{2}\leq\vert z\vert_{1}$ for $z\in\mathbb{R}^{m}$ 
where $\vert z\vert_2$ is the usual Euclidean length of $z$ and 
$\vert z\vert_{1}=\sum_{i=1}^{m}\vert z_i\vert$. 
One may see that if $c_{3}\leq\frac{c_{\rho}}{2C_{1}}\frac{c_{2}}{m}$ then from 
\eqref{equ2700} it follows that there exists $i_0\in\{1,\cdots,m\}$ such that 
\begin{equation}\label{equ2900} 
\fint_{\partial B_r(x)}u_{i_0}d\sigma
\geq 
c_{3}Q_{min}r\text{.}
\end{equation}

As in Lemma \ref{lemma1100} there exists $c_{4}>0$ such that 
\begin{equation*}
v_{i_0}(y)\geq c_{4}(1-\frac{1}{r}\vert y-x\vert)
\fint_{\partial B_r(x)}v_{i_0}d\sigma\enskip\text{for all}\enskip y\in B_r(x)\text{.} 
\end{equation*}
Thus for $0<\kappa<\frac{1}{2}$ we have 
\begin{equation}\label{equ3000} 
v_{i_0}(y)\geq c_{4}(1-\kappa)
\fint_{\partial B_r(x)}u_{i_0}d\sigma\enskip\text{for all}\enskip y\in B_{\kappa r}(x)\text{.} 
\end{equation}

From \eqref{equ2900} and \eqref{equ3000} it follows that 
\begin{equation}\label{equ3100} 
v_{i_0}(y)\geq c_{4}c_{3} Q_{min} (1-\kappa) r\enskip\text{for all}\enskip y\in B_{\kappa r}(x)\text{.} 
\end{equation}
Because $B_{2\kappa r}(x)\subset\subset\Omega$ and 
$x\in\partial\{\vert\uu\vert>0\}$ by Theorem \ref{theorem1100} 
and Poisson representation formula we have 
\begin{equation}\label{equ3200}
\sup_{B_{\kappa r}(x)}u_{i_0}
\leq 
C_{5}
\fint_{\partial B_{2 \kappa r}(x)} u_{i_0} d\sigma
\leq 2 C_{5} C Q_{max}\kappa r
\end{equation}
and hence 
\begin{equation}\label{equ3400} 
u_{i_0}(y)\leq 2C_{5}C Q_{max} \kappa r\enskip\text{for all}\enskip y\in B_{\kappa r}(x)\text{.} 
\end{equation}

From \eqref{equ3100} and \eqref{equ3400} it follows that 
for $y\in B_{\kappa r}(x)$ we have  
\begin{multline*}
v_{i_0}(y)-u_{i_0}(y)
\geq 
c_{4}c_{3}Q_{min} (1-\kappa) r-2C_{5}C Q_{max} \kappa r \\
=
(c_{4}c_{3} Q_{min}+2C_{5}C Q_{max})
\bigl(
\frac{c_{4}c_{3} Q_{min}}{c_{4}c_{3} Q_{min}+2C_{5}C Q_{max}}
- \kappa
\bigr)
r
\text{.} 
\end{multline*}
Choosing 
\begin{equation*}
\kappa
=
\frac{1}{2}
\frac{1}{
1+\frac{2C_{5}C}{c_{4}c_{3}}\frac{Q_{max}}{Q_{min}}
}
\end{equation*}
we obtain 
\begin{equation}\label{equ3500} 
v_{i_0}(y)-u_{i_0}(y)
\geq
\frac{1}{2}c_{4}c_{3} Q_{min}r
\enskip\text{for all}\enskip y\in B_{\kappa r}(x)\text{.} 
\end{equation}

Now from \eqref{equ2500} and \eqref{equ3500} it follows that 
\begin{multline*}
(
\frac{1}{2}c_{4}c_{3} Q_{min} r
)^{2}\vert B_{\kappa r}\vert
\leq 
\int_{B_r(x)}\vert \uu-\vv\vert^{2}dy \\
\leq 
C Q_{max}^2 r^2
\vert B_r(x)\cap\{\vert \uu\vert=0\}\vert\text{.} 
\end{multline*}
Thus 
\begin{equation*}
c_6\frac{Q_{min}^2}{Q_{max}^2}\vert B_{\kappa r}\vert
\leq 
\vert B_r(x)\cap\{\vert\uu\vert=0\}\vert\text{,} 
\end{equation*}
which proves the lemma.  
\end{proof}

\subsection{$\{\vert\uu\vert>0\}$\label{supports} 
has locally finite perimeter in $\Omega$}\label{subsection1200} 
In this subsection $\uu$ is a (local) minimizer in $\Omega$. 
By strong minimum principle it follows that 
in each component of $\{\vert\uu\vert>0\}$, for $i=1,\cdots,m$, either 
$u_i$ is identically vanishing 
or it is positive. It follows that 
\begin{equation*}
\Omega\cap\partial\{u_i>0\}\subset\Omega\cap\partial\{\vert \uu\vert>0\}\enskip\text{for}\enskip 
i=1,\cdots,m\text{.}
\end{equation*} 

Because each $u_i$ is subharmonic in $\Omega$ 
and $u_i\geq 0$ we have that $\lambda_i=\Delta u_i$ is a positive Radon measure 
such that $\lambda_i(D)<\infty$ for all $D\subset\subset\Omega$. 
Since also $u_i$ is harmonic in $\{u_i>0\}$ 
we have that the support of $\lambda_i$ is in $\Omega\cap\partial\{u_i>0\}$. 
Let us define 
\begin{equation}\label{equ3600} 
\mathcal{U}=\sum_{i=1}^{m}u_{i}\text{.}
\end{equation}

\begin{lemma}\label{lemma1700}
\begin{enumerate}[(i)]
\item\label{lemma1700-1} 
For all $D\subset\subset\Omega$ there 
exist $0<c\leq C$ such that for $B_{r}(x)\subset D$ 
with $x\in\Omega\cap\partial\{\vert \uu\vert>0\}$ we have 
\begin{equation*}
cQ_{min}r^{n-1}\leq\sum_{i=1}^{m}\lambda_i(B_r(x))\leq CQ_{max}r^{n-1}\text{.} 
\end{equation*}
\item\label{lemma1700-2} 
For all $D\subset\subset\Omega$ we have $\mathcal{H}^{n-1}(D\cap\partial\{\vert \uu\vert>0\})<\infty$. 
\item\label{lemma1700-3} 
There exist nonnegative Borel functions $q_{i}$ such that 
\begin{equation*}
\lambda_i=q_{i}\mathcal{H}^{n-1}\llcorner{(\Omega\cap\partial\{\vert \uu\vert>0\})}\text{.}
\end{equation*} 
\item\label{lemma1700-4} 
For all $D\subset\subset\Omega$ there exist $0<c\leq C$ such that 
\begin{equation}\label{equ3610} 
cQ_{min}\leq \sum_{i=1}^{m}q_{i}\leq CQ_{max}\enskip\text{in}\enskip D
\end{equation}
and for $B_r(x)\subset D$ with 
$x\in\Omega\cap\partial\{\vert \uu\vert>0\}$ we have 
\begin{equation*}
cQ_{min}r^{n-1}\leq\mathcal{H}^{n-1}(B_r(x)\cap\partial\{\vert \uu\vert>0\})\leq CQ_{max}r^{n-1}\text{.}
\end{equation*}
\end{enumerate}
\end{lemma}
\begin{proof}
Let $D\subset\subset\Omega$. Then by Theorems \ref{theorem1100} and \ref{theorem1200} 
there exist constants $0<c\leq C<\infty$ such that for (small) $B_{r}(x)\subset D$ 
with $x\in\Omega\cap\partial\{\vert \uu\vert>0\}$ we have 
\begin{equation}\label{equ3700}
cQ_{min}\leq\frac{1}{r}\fint_{\partial B_{r}(x)}\mathcal{U}d\mathcal{H}^{n-1}\leq CQ_{max}\text{.} 
\end{equation}

Using \eqref{equ3700} the assertions follow from 
Theorems 4.3 and 4.5 of \cite{AltCaffarelli1981}. 
\end{proof}

\begin{lemma}\label{lemma1800} 
The set $\{\vert \uu\vert>0\}$ has locally finite perimeter in $\Omega$, 
$\Omega\cap\partial^{*}\{\vert \uu\vert>0\}\subset\Omega\cap\partial\{\vert \uu\vert>0\}$ 
and 
\begin{equation*}
\mathcal{H}^{n-1}\bigl(
\Omega\cap
(\partial\{\vert \uu\vert>0\}
\backslash
\partial^{*}\{\vert \uu\vert>0\})
\bigr)
=0\text{.}
\end{equation*}
\end{lemma}
\begin{proof}
For $D\subset\subset\Omega$ 
a compact set by part \eqref{lemma1700-2} 
in Lemma \ref{lemma1700} we have 
\begin{equation*}
\mathcal{H}^{n-1}\bigl(D\cap\partial_{*}\{\vert \uu\vert>0\}\bigr) \\
\leq 
\mathcal{H}^{n-1}\bigl(D\cap\partial\{\vert \uu\vert>0\}\bigr)
<\infty
\text{.} 
\end{equation*}

By a local version of Theorem 1 in Section 5.11 
of \cite{EvansGariepy1992} we have that 
because for all $D\subset\subset\Omega$ compact sets we have 
$\mathcal{H}^{n-1}(D\cap\partial_{*}\{\vert \uu\vert>0\})<\infty$
it follows that $\{\vert \uu\vert>0\}$ has locally finite perimeter in $\Omega$. 

Let $x\in\Omega\cap\partial\{\vert \uu\vert>0\}$. 
By Lemma \ref{lemma1500} and \ref{lemma1600} we obtain respectively 
\begin{equation*} 
\varlimsup_{r\to0}\frac{\vert B_{r}(x)\cap\{\vert \uu\vert>0\}\vert}{r^{n}}>0\enskip\text{and}
\enskip\varlimsup_{r\to0}\frac{\vert B_{r}(x)\backslash\{\vert \uu\vert>0\}
\vert}{r^{n}}>0\text{.} 
\end{equation*}
It follows that $x\in\partial_{*}\{\vert \uu\vert>0\}$. 
Hence 
\begin{equation*} 
\Omega\cap\partial\{\vert \uu\vert>0\}\subset\partial_{*}\{\vert \uu\vert>0\}
\end{equation*}
and together with $\partial_{*}\{\vert \uu\vert>0\}\subset\partial\{\vert \uu\vert>0\}$ we obtain 
\begin{equation}\label{equ3800} 
\Omega\cap\partial\{\vert \uu\vert>0\}=\Omega\cap\partial_{*}\{\vert \uu\vert>0\}\text{.}
\end{equation}

By a local version of Lemma 1 in Section 5.8 of \cite{EvansGariepy1992} we have that 
$\Omega\cap\partial^{*}\{\vert \uu\vert>0\}\subset \Omega\cap\partial_{*}\{\vert \uu\vert>0\}$ 
and $\mathcal{H}^{n-1}(\Omega\cap(\partial_{*}\{\vert \uu\vert>0\}\backslash\partial^{*}\{\vert \uu\vert>0\}))=0$. 

Now by these properties and \eqref{equ3800} the last two claims of the lemma follow. 
\end{proof}

\subsection{Domain variation formula}\label{subsection1300}

\begin{lemma}[Domain Variation Formula]\label{lemma1900} 
Let $Q$ be continuous, $\uu$ be a minimizer and 
$\Psi\in C^{\infty}_{c}(\Omega;\mathbb{R}^{n})$, then we have 
\begin{equation*}
\int_{\Omega}
\Bigl(
2\sum_{i=1}^{m}
(\nabla u_i^{T}
\nabla\Psi
\nabla u_i)
-\vert\nabla \uu\vert^{2}
\operatorname{div}\Psi
\Bigr)dx
+\int_{\Omega}Q^2
\Psi\cdot
d\mu
=0
\end{equation*}
where we have used the notation 
\begin{equation}\label{equ3850} 
\mu=\nabla\chi_{\{\vert \uu\vert>0\}}\text{.} 
\end{equation}
\end{lemma}
\begin{proof}
Let $L$ be the Lipschitz constant of $\Psi$. 
We define for $t>0$ 
\begin{equation*}
\Phi_t(x)=x+t\Psi(x)\enskip\text{and}\enskip\uu_t(x)=\uu(x+t\Psi(x))\text{.}
\end{equation*}

One can show that for $0<t<\frac{1}{L}$, 
$\Phi_t:\Omega\to\Omega$ is bijective, $\Phi_t^{-1}\in C^{\infty}(\Omega;\mathbb{R}^{n})$, 
and moreover $\Phi_t^{-1}(y)-y\in C^{\infty}_{c}(\Omega;\mathbb{R}^{n})$ 
with $\operatorname{supp}(\Phi_t^{-1}(y)-y)\subset\operatorname{supp}(\Psi)$. 

We have $\nabla\Phi_t=I+t\nabla\Psi$ and for $i=1,\cdots,m$  
\begin{equation*}
\nabla u_{t,i}(x)
=
\nabla\Phi_t(x)
\nabla u_i(\Phi_t(x))
\text{.} 
\end{equation*}

It follows that 
\begin{equation*}
\vert \nabla u_{t,i}(x)\vert^{2}
=
(\nabla u_i(\Phi_{t}(x)))^{T}
A_t(x)
\nabla u_i(\Phi_{t}(x))
\end{equation*}
where 
\begin{equation*}
A_t(x)=(\nabla\Phi_t(x))^{T}\nabla\Phi_t(x)
=I+t\bigl((\nabla\psi)^{T}+\nabla\psi\bigr)
+t^2 (\nabla\psi)^{T}\nabla\psi\text{.}
\end{equation*} 

By change of variables 
\begin{multline}\label{equ3900} 
\int_{\Omega}\vert \nabla u_{t,i}\vert^{2}dx
=
\int_{\Omega}
(\nabla u_i(\Phi_{t}(x)))^{T}
A_t(x)
\nabla u_i(\Phi_{t}(x))
dx \\
=
\int_{\Omega}
(\nabla u_i(y))^{T}
A_t(\Phi_t^{-1}(y))
\nabla u_i(y)
\vert\det \nabla_y\Phi_t^{-1}(y)\vert dy\text{.} 
\end{multline}

We also have 
\begin{equation}\label{equ4000} 
\frac{d}{dt}
(A_t(\Phi_t^{-1}(y)))
\big\vert_{t=0}
=
\nabla\Psi(y)+\nabla\Psi(y)^{T}
\end{equation}
and 
\begin{equation}\label{equ4100} 
\frac{d}{dt}\vert\det \nabla_y\Phi_t^{-1}(y)\vert\big\vert_{t=0}
=
-\operatorname{div}\Psi(y)\text{.} 
\end{equation}

By differentiating \eqref{equ3900} with respect to $t$ and using \eqref{equ4000} 
and \eqref{equ4100} we obtain 
\begin{equation}\label{equ4200} 
\frac{d}{dt}
\int_{\Omega}\vert \nabla u_{t,i}\vert^{2}dx\big\vert_{t=0} 
=
\int_{\Omega}
\Bigl(
2\nabla u_i^{T}
\nabla\Psi
\nabla u_i 
-
\vert\nabla u_i\vert^{2}
\operatorname{div}\Psi
\Bigr)
dy\text{.} 
\end{equation}

We differentiate also the second term in our functional  
\begin{multline}\label{equ4300} 
\frac{d}{dt}
\int_{\Omega}Q^2(x)\chi_{\{\vert \uu_t\vert>0\}}(x)dx 
\big\vert_{t=0} \\
=
\frac{d}{dt}
\int_{\Omega}Q^2(x)\chi_{\{\vert \uu\vert>0\}}(\Phi_t(x))dx
\big\vert_{t=0} \\
=
\int_{\Omega}Q^2(x)
\Psi(x)\cdot
d\mu(x)\text{.} 
\end{multline}

By \eqref{equ4200} and \eqref{equ4300} we obtain 
\begin{multline}\label{equ4400} 
\frac{d}{dt}
\int_{\Omega}
\bigl(
\vert \nabla \uu_{t}\vert^{2}
+
Q^2
\chi_{\{\vert \uu_t\vert>0\}}
\bigr)
dx\big\vert_{t=0}  \\
=
\int_{\Omega}
\Bigl(
2
\sum_{i=1}^{m}
(
\nabla u_i^{T}
\nabla\Psi
\nabla u_i
)
-
\vert\nabla \uu\vert^{2}
\operatorname{div}\Psi
\Bigr)
dx
+
\int_{\Omega}Q^2
\Psi\cdot
d\mu
\text{.} 
\end{multline}
This proves the lemma in the case when $\uu$ is an absolute minimizer. 

In the case when $\uu$ is a local minimizer one should also show that 
$d(\uu_t,\uu)\to 0$ as $t\to0$ which together with the equation \eqref{equ4400} 
proves the lemma also in the case when $\uu$ is a local minimizer. 
\end{proof}

\subsection{Blow-up limits}\label{subsection1350} 
\begin{lemma}\label{lemma2000} 
Let $\uu$ be a minimizer in $\Omega$, $0\in\Omega$, $\uu(0)=0$ and 
$r_k\to0$ as $k\to\infty$. 
Assume also that $\vv\in H^{1}_{loc}(\mathbb{R}^{n};\mathbb{R}^{m})$ 
such that as $k\to\infty$ we have the convergences 
\begin{equation*}
\uu_{r_k}\to \vv\enskip\text{in}\enskip 
L^2_{loc}(\mathbb{R}^{n};\mathbb{R}^{m})\text{,}\enskip 
\uu_{r_k}\to \vv\enskip\text{a.e. in}\enskip\mathbb{R}^{n}\text{,} 
\end{equation*}
\begin{equation*} 
\uu_{r_k}\rightharpoonup \vv\enskip\text{weakly in}\enskip H^{1}(B_{R};\mathbb{R}^{m})
\enskip\text{for all}\enskip R>0
\end{equation*} 
and 
\begin{equation}\label{equ5300} 
\varlimsup_{k\to\infty}\int_{B_{R}}\vert Q_{r_k}-Q(0)\vert dx=0
\enskip\text{for all}\enskip R>0
\end{equation}
then $\vv$ is an absolute minimizer in $B_{R}$ for all $R>0$ 
with constant $Q(0)$. 
\end{lemma}
\begin{proof}
Let us fix $R>0$. 
Let $\w\in H^{1}(B_{R};\mathbb{R}^{m})$ 
such that $w_i\geq 0$ a.e. in $B_{R}$ for $i=1,\cdots,m$ 
and $\w=\vv$ on $\partial B_{R}$. 
We should show that 
\begin{equation}\label{equ5400} 
\int_{B_{R}}
\bigl(\vert\nabla \vv\vert^{2}+Q^2(0)\chi_{\{\vert \vv\vert>0\}}\bigr)dy
\leq
\int_{B_{R}}\bigl(\vert\nabla\w\vert^{2}+Q^2(0)\chi_{\{\vert\w\vert>0\}}\bigr)dy\text{.}
\end{equation}

Let $0<r<R$ and $\eta\in C^{\infty}_{c}(B_{R})$ 
such that $\eta=1$ on $B_r$ and $0\leq\eta\leq 1$. 
Let us define 
\begin{equation*}
\w^k
=
\bigl(\w+(1-\eta)(\uu_{r_k}-\vv)\bigr)^{+}
\text{,} 
\end{equation*}
here positive part is taken for each component separately. 
Then $\w^{k}\in H^{1}(B_{R};\mathbb{R}^{m})$, 
$\w^k_{i}\geq 0$ a.e. in $B_{R}$ for $i=1,\cdots,m$ 
and $\w^k=\uu_{r_k}$ on $\partial B_{R}$. 
Define 
\begin{equation*}
\h^k(x)
=
r_k\w^k(\frac{x}{r_k})
\enskip\text{for}\enskip x\in B_{R r_{k}}
\end{equation*}
and extend it by $\uu$ outside $B_{R r_k}$. 

It is easy to see that $d(\uu,\h^k)\to 0$ as $k\to\infty$ where 
the distance $d$ is defined in \eqref{equ400}. 
Because $\uu$ is a (local) minimizer 
it follows that (for small enough $r_k$, i.e. large $k$) 
we have 
\begin{equation*}
\int_{B_{R r_k}}
\bigl(\vert \nabla \uu\vert^{2}+Q^2\chi_{\{\vert \uu\vert>0\}} \bigr)dx
\leq 
\int_{B_{R r_k}}
\bigl(\vert\nabla \h^k\vert^{2}+Q^2\chi_{\{\vert \h^k\vert>0\}} \bigr)dx\text{.} 
\end{equation*}
Thus for large $k$ we have 
\begin{equation*}
\int_{B_{R}}\bigl(\vert\nabla \uu_{r_k}\vert^{2}+Q^2_{r_k}\chi_{\{\vert \uu_{r_k}\vert>0\}} \bigr)dy
\leq 
\int_{B_{R}}\bigl(   \vert\nabla \w^k \vert^{2}+Q^2_{r_k}\chi_{\{\vert \w^k\vert>0\}}   \bigr)dy\text{.} 
\end{equation*}
It follows that 
\begin{multline}\label{equ5500} 
\varliminf_{k\to\infty}
\int_{B_{R}}\bigl(
\vert\nabla \uu_{r_k}\vert^{2}
-\vert\nabla \w^k \vert^{2}
\bigr)dy
+\varliminf_{k\to\infty}
\int_{B_{R}}Q^{2}_{r_k}\chi_{\{\vert \uu_{r_k}\vert>0\}}dy \\
\leq\varliminf_{k\to\infty}
\int_{B_{R}}Q^{2}_{r_k}\chi_{\{\vert \w^k\vert>0\}}dy\text{.} 
\end{multline}

For $i=1,\cdots,m$ we compute 
\begin{multline*}
\int_{B_{R}}\bigl(
\vert \nabla u_{r_k,i}\vert^{2}-\vert \nabla w^k_i\vert^{2}
\bigr)dy \\
\geq 
\int_{B_{R}}\bigl(
\vert \nabla u_{r_k,i}\vert^{2}
-\bigl\vert \nabla (w_i+(1-\eta)(u_{r_k,i}-v_i))\bigr\vert^{2}
\bigr)dy \\
=
\int_{B_{R}}\bigl(
\vert\nabla v_i\vert^{2}-\vert \nabla w_i\vert^{2}
\bigr)dy \\
+
\int_{B_{R}}\bigl(
(1-(1-\eta)^{2})\vert \nabla (u_{r_k,i}-v_i)\vert^{2}
-\vert\nabla\eta\vert^{2}(u_{r_k,i}-v_i)^{2} \\
+2(u_{r_k,i}-v_i)\nabla\phi\cdot\nabla\eta 
+2\nabla v_i\cdot \nabla (u_{r_k,i}-v_i) \\
-(1-\eta)\nabla w_i\cdot\nabla(u_{r_k,i}-v_i)  
+2(u_{r_k,i}-v_i)(1-\eta)\nabla\eta\cdot\nabla(u_{r_k,i}-v_i)
\bigr)dy \\
\geq 
\int_{B_{R}}\bigl(
\vert\nabla v_i\vert^{2}-\vert \nabla w_i\vert^{2}
\bigr)dy \\
+
\int_{B_{R}}\bigl(
-\vert\nabla\eta\vert^{2}(u_{r_k,i}-v_i)^{2} 
+2(u_{r_k,i}-v_i)\nabla\phi\cdot\nabla\eta 
+2\nabla v_i\cdot \nabla (u_{r_k,i}-v_i) \\
-(1-\eta)\nabla\phi\cdot\nabla(u_{r_k,i}-v_i)  
+2(u_{r_k,i}-v_i)(1-\eta)\nabla\eta\cdot\nabla (u_{r_k,i}-v_i)
\bigr)dy
\end{multline*}
thus 
\begin{equation}\label{equ5600} 
\varliminf_{k\to\infty}
\int_{B_{R}}\bigl(
\vert \nabla \uu_{r_k}\vert^{2}
-
\vert \nabla \w^k\vert^{2}
\bigr)dy
\geq 
\int_{B_{R}}\bigl(
\vert\nabla \vv\vert^{2}-\vert \nabla\w\vert^{2}
\bigr)dy\text{.}
\end{equation}
For a.e. $y\in\{\vert\vv\vert>0\}$ and large enough $k$ we have $\vert\uu_{r_k}(y)\vert>0$. 
Thus for a.e. $y\in B_{R}$ we have 
$\chi_{\{\vert \vv(y)\vert>0\}}\leq\varliminf_{k\to\infty}\chi_{\{\vert \uu_{r_k}(y)\vert>0\}}$. 
Using Fatou's Lemma we have 
\begin{multline}\label{equ5700} 
\varliminf_{k\to\infty}\int_{B_{R}} Q^{2}_{r_k} \chi_{\{\vert \uu_{r_k}\vert>0\}}dy \\
=
\varliminf_{k\to\infty}\int_{B_{R}} 
\bigl(
Q^{2}(0) \chi_{\{\vert \uu_{r_k}\vert>0\}}
+(Q^{2}_{r_k}-Q^{2}(0)) \chi_{\{\vert \uu_{r_k}\vert>0\}}
\bigr)
dy \\
\geq 
\varliminf_{k\to\infty}\int_{B_{R}} 
\bigl(
Q^{2}(0) \chi_{\{\vert \uu_{r_k}\vert>0\}}-2Q_{max}\vert Q_{r_k}-Q(0)\vert 
\bigr)
dy \\
\geq 
\varliminf_{k\to\infty}\int_{B_{R}} 
Q^{2}(0) \chi_{\{\vert \uu_{r_k}\vert>0\}}dy 
-2Q_{max} \varlimsup_{k\to\infty}\int_{B_{R}} \vert Q_{r_k}-Q(0)\vert dy \\
=
\varliminf_{k\to\infty}\int_{B_{R}} 
Q^{2}(0) \chi_{\{\vert \uu_{r_k}\vert>0\}}
dy 
\geq 
\int_{B_{R}} 
Q^{2}(0) 
\varliminf_{k\to\infty}\chi_{\{\vert \uu_{r_k}\vert>0\}}
dy \\
\geq 
\int_{B_{R}} 
Q^{2}(0) 
\chi_{\{\vert \vv\vert>0\}}
dy
\text{.}
\end{multline}

Next we estimate 
\begin{multline}\label{equ5800} 
\varliminf_{k\to\infty}
\int_{B_{R}}
Q^{2}_{r_k}
\chi_{\{\vert\w^k\vert>0\}}dy \\
=
\varliminf_{k\to\infty}
\int_{B_{R}}
(Q^{2}(0)+(Q^{2}_{r_k}-Q^{2}(0)))
\chi_{\{\vert\w^k\vert>0\}}dy \\
\leq 
\varliminf_{k\to\infty}
\int_{B_{R}}
Q^{2}(0)
\chi_{\{\vert\w^k\vert>0\}}dy 
+
\varlimsup_{k\to\infty}
\int_{B_{R}}
(Q^{2}_{r_k}-Q^{2}(0))
\chi_{\{\vert\w^k\vert>0\}}dy \\
\leq 
\varliminf_{k\to\infty}
\int_{B_{R}}
Q^{2}(0)
\chi_{\{\vert\w^k\vert>0\}}dy 
+
2Q_{max}
\varlimsup_{k\to\infty}
\int_{B_{R}}
\vert Q_{r_k}-Q(0)\vert 
dy \\
=
\varliminf_{k\to\infty}
\int_{B_{R}}
Q^{2}(0)
\chi_{\{\vert\w^k\vert>0\}}dy \\
=
\varliminf_{k\to\infty}
\int_{B_r\cup(B_{R}\backslash B_r)}
Q^{2}(0)
\chi_{\{
\vert (\w+(1-\eta)(\uu_{r_k}-\vv))^{+}\vert>0
\}}dy \\
\leq 
\int_{B_r}Q^{2}(0)\chi_{\{\vert\w\vert>0\}}dy
+Q^{2}(0)\vert B_{R}\backslash B_r\vert\text{.}
\end{multline}

From \eqref{equ5500}, \eqref{equ5600}, \eqref{equ5700} and \eqref{equ5800} 
it follows that 
\begin{multline*}
\int_{B_{R}}\bigl(\vert\nabla \vv\vert^{2}-\vert\nabla \w\vert^{2}\bigr)dy
+\int_{B_{R}}Q^{2}(0)\chi_{\{\vert \vv\vert>0\}}dy \\
\leq\int_{B_r}Q^{2}(0)\chi_{\{\vert \w\vert>0\}}dy
+Q^{2}(0)\vert B_{R}\backslash B_r\vert 
\end{multline*}
letting $r\to R$ we obtain \eqref{equ5400}. 
\end{proof}

We say that the sequence of sets $D_{k}\subset\mathbb{R}^{n}$ locally in Hausdorff distance 
converges to the set $D\subset \mathbb{R}^{n}$ as $k\to\infty$, if for all $R>0$ we have 
$D_{k}\cup B_R^{c}\to D\cup B_R^c$ in Hausdorff distance as $k\to\infty$. 

\begin{definition}\label{definition1200} 
Let $\uu$ be a minimizer in $\Omega$ with $0\in\Omega$ and $\uu(0)=0$. 
We call $\vv\in C^{0,1}_{loc}(\mathbb{R}^{n};\mathbb{R}^{m})$ 
a blowup limit of $\uu$ at the origin if there 
is a sequence $B_{r_k}\subset\Omega$ with $r_k\to 0$ as $k\to\infty$ such that 
\begin{equation*}
\begin{aligned}
& \uu_{r_k}\to \vv\enskip\text{in}\enskip C^{0,\alpha}_{loc}( \mathbb{R}^{n} ;\mathbb{R}^{m})
\enskip\text{for all}\enskip 0\leq\alpha<1\text{,} \\ 
& \uu_{r_k}\to \vv\enskip\text{in}\enskip W^{1,p}_{loc}(\mathbb{R}^{n};\mathbb{R}^{m})
\enskip\text{for all}\enskip 1\leq p<\infty\text{,} \\ 
& \nabla \uu_{r_k}\to \nabla \vv\enskip\text{a.e. in}\enskip\mathbb{R}^{n}\text{,} \\ 
& \chi_{\{\vert \uu_{r_k}\vert>0\}}\to\chi_{\{\vert \vv\vert>0\}}
\enskip\text{in}\enskip L^{p}_{loc}(\mathbb{R}^{n})\enskip\text{for all}\enskip 1\leq p<\infty\text{,} \\
& \chi_{\{\vert \uu_{r_k}\vert>0\}}\to\chi_{\{\vert \vv\vert>0\}}
\enskip\text{a.e. in}\enskip\mathbb{R}^{n}\text{,} \\
& \{\vert \uu_{r_k}\vert>0\}\to\{\vert \vv\vert>0\}
\enskip\text{locally in Hausdorff distance}\text{,} \\ 
& \{\vert \uu_{r_k}\vert=0\}\to\{\vert \vv\vert=0\}
\enskip\text{locally in Hausdorff distance} \\ 
\end{aligned}
\end{equation*}
and 
\begin{equation*}
\partial\{\vert \uu_{r_k}\vert>0\}\to\partial\{\vert \vv\vert>0\}
\enskip\text{locally in Hausdorff distance}\text{.}
\end{equation*}

\end{definition}

\begin{lemma}\label{lemma2100} 
Let $\uu$ be a minimizer in $\Omega$, $0\in\Omega$, $\uu(0)=0$, 
$r_k\to 0$ as $k\to\infty$ 
and \eqref{equ5300} hold. 
Then there exists a subsequence $k_j$ and $\vv\in C^{0,1}_{loc}(\mathbb{R}^{n};\mathbb{R}^{m})$ 
such that $\vv$ is a blowup limit of $\uu$ at the origin with respect to the sequence 
$\uu_{r_{k_j}}$. 
\end{lemma}
\begin{proof}
Let $\uu$ and $r_k$ as in the statement of the Lemma. 
Let $R>0$, then by Corollary \ref{corollary1000} for 
large enough $k$ we have that 
$\uu_{r_k}$ are uniformly bounded in $C^{0,1}(\overline{B}_{R};\mathbb{R}^{m})$. 
Thus there exists a subsequence $k_j$ 
and $\vv\in C^{0,1}_{loc}(\mathbb{R}^{n};\mathbb{R}^{m})$ 
such that $\uu_{r_{k_j}}\to \vv$ in $C^{0,\alpha}(\overline{B}_{R};\mathbb{R}^{m})$ 
for all $0\leq\alpha<1$ and $R>0$ (using a diagonalization argument) 
and weakly in $H^{1}(B_{R};\mathbb{R}^{m})$ for all $R>0$. 

From Lemma \ref{lemma2000} it follows that $\vv$ is an 
absolute minimizer in $B_{R}$ 
with constant $Q(0)$ for all $R>0$. 
From part \eqref{lemma1700-2} in Lemma 
\ref{lemma1700} it follows that $\vert B_{R}\cap\partial\{\vert \vv\vert>0\}\vert=0$ 
for all $R>0$. 

Then by similar arguments as in \cite{AltCaffarelli1981,AltCaffarelliFriedman1984a} we 
obtain the rest of the convergences. 
\end{proof}

\section{Reduction to Scalar Weak Solution\\ 
(Proofs of Theorems \ref{theorem1300} and \ref{theorem1350}) }\label{section1800}

\subsection{Equation satisfied by $u_i$}\label{subsection1500} 

\begin{proof}[Proof of Theorem \ref{theorem1300}]
For $i=1,\cdots,m$, by Lemma \ref{lemma1700} part \eqref{lemma1700-3} there exists a nonnegative 
Borel function $q_i$ such that 
\begin{equation*}
\Delta u_i=q_i\mathcal{H}^{n-1}\llcorner(\Omega\cap\partial\{\vert \uu\vert>0\})\text{.} 
\end{equation*}

By Lemma \ref{lemma1800} we obtain 
\begin{equation}\label{equ5930} 
\Delta u_i=q_i\mathcal{H}^{n-1}\llcorner(\Omega\cap\partial^{*}\{\vert \uu\vert>0\})\text{.} 
\end{equation}

By Theorem 4.8 and Remark 4.9 in \cite{AltCaffarelli1981} for 
$\mathcal{H}^{n-1}$ a.e. $x\in\Omega\cap\partial^{*}\{\vert \uu\vert>0\}$, 
$x$ is a Lebesgue point of $q_i$ with respect to the measure 
$\mathcal{H}^{n-1}\llcorner(\Omega\cap\partial^{*}\{\vert \uu\vert>0\})$ 
and as $y\to 0$ we have 
\begin{equation}\label{equ6000} 
u_i(x+y)
=
q_{i}(x)\bigl(-\nu_{\{\vert \uu\vert>0\}}(x)\cdot y\bigr)^{+}+o(\vert y\vert)
\text{.} 
\end{equation}

We might assume that for $\mathcal{H}^{n-1}$ a.e. 
$x\in\Omega\cap\partial^{*}\{\vert \uu\vert>0\}$,  
\eqref{equ6000} holds 
for all $i=1,\cdots,m$. 

From \eqref{equ6000} it follows that for $\mathcal{H}^{n-1}$ a.e. 
$x\in\Omega\cap\partial^{*}\{\vert \uu\vert>0\}$ as $y\to0$ we have  
\begin{equation}\label{equ6010} 
\vert \uu(x+y)\vert
=
\vert \q(x)\vert\bigl(-\nu_{\{\vert \uu\vert>0\}}(x)\cdot y\bigr)^{+}+o(\vert y\vert)
\end{equation}
here $\q(x)$ is the vector function with components $q_i(x)$. 

Now using the first inequality in \eqref{equ3610}, 
\eqref{equ6000} and \eqref{equ6010} for $y\in\{\vert \uu\vert >0 \}-x$, $y\to0$ non-tangentially, i.e. 
for some $\eta>0$ we have $-\nu_{\{\vert \uu\vert>0\}}(x)\cdot \frac{y}{\vert y\vert}\geq\eta$, 
we might compute and obtain 
\begin{equation*}
\bigl\vert
\frac{u_{i}(x+y)}{\vert \uu(x+y)\vert}
-\frac{q_{i}(x)}{\vert \q(x)\vert}
\bigr\vert 
\leq 
\frac{
o(1)}{
\vert \q(x)\vert\eta}
\end{equation*}
which  proves \eqref{equ500}. 

From \eqref{equ6000} it follows that 
for $\mathcal{H}^{n-1}$ a.e. 
$x\in\Omega\cap\partial^{*}\{\vert \uu\vert>0\}$ 
the blowup of $\uu$ at $x$ is unique and given by 
\begin{equation}\label{equ6030} 
\uu_{x,r}(y)
\to 
\vv_{x}(y):=\q(x)\bigl(-\nu_{\{\vert \uu\vert>0\}}(x)\cdot y\bigr)^{+}
\enskip\text{as}\enskip r\to 0\text{.} 
\end{equation}

From Lemma \ref{lemma2000} it follows that 
$\vv_{x}$ is a global absolute minimizer with constant $Q(x)$. 

For short notation let us denote $\nu_{\{\vert \uu\vert>0\}}(x)=\nu(x)$. 
Let $\varphi\in C^{\infty}_{c}(B)$ and set $\Psi_{x}(y):=\varphi(y)\nu(x)$, 
then 
\begin{equation*}
\{\vert \vv\vert>0\}
=\{y\in \mathbb{R}^{n}\enskip\vert\enskip \nu(x) \cdot y<0\}
\text{,}\enskip 
\nabla v_{x,i}(y)=-q_i(x) \chi_{\{\nu(x)\cdot y<0\}}\nu(x)\text{,}
\end{equation*}
\begin{equation*}
\nabla\Psi_{x}(y)=\nabla\varphi(y)\nu(x)^{T}\text{,}\qquad 
\operatorname{div}\Psi_{x}(y)=\partial_{\nu(x)}\varphi(y)\text{,} 
\end{equation*}
and 
\begin{multline*}
\Psi_{x}(y)\cdot d\mu(y)
=-\Psi_{x}(y)\cdot \nu(x)d(\mathcal{H}^{n-1}\llcorner\{ z\in\mathbb{R}^{n}\enskip\vert\enskip  \nu(x) \cdot z=0\})(y) \\
=-\varphi(y)d(\mathcal{H}^{n-1}\llcorner\{ z\in\mathbb{R}^{n}\enskip\vert\enskip  \nu(x) \cdot z=0\})(y)
\end{multline*}
(the vector Radon measure $\mu$ is defined in \eqref{equ3850}). 

By domain variation formula proved in Lemma \ref{lemma1900} 
we have 
\begin{multline*}
0
=
\int_{B}
\Bigl(
2\sum_{i=1}^{m}
(\nabla v_i^{T}
\nabla\Psi_{x}
\nabla v_i)
-\vert\nabla \vv\vert^{2}
\operatorname{div}\Psi_{x}
\Bigr)dy
+\int_{B}Q^2(x)
\Psi_{x}(y)\cdot d\mu(y) \\
=
\int_{B\cap\{\nu(x)\cdot y<0\}}
\Bigl(
2\vert \q(x)\vert^{2}\partial_{\nu(x)}\varphi(y)
-\vert \q(x)\vert^{2}\partial_{\nu(x)}\varphi(y)
\Bigr)dy \\
-\int_{B\cap\{\nu(x)\cdot y=0\}}Q^2(x)\varphi(y)d\sigma(y) \\
=
\vert \q(x)\vert^{2}
\int_{B\cap\{\nu(x)\cdot y<0\}}
\partial_{\nu(x)}\varphi(y)dy 
-Q^2(x)
\int_{B\cap\{\nu(x)\cdot y=0\}}\varphi(y)d\sigma(y) \\
=
\vert \q(x)\vert^{2}
\int_{B\cap\{\nu(x)\cdot y=0\}}\varphi(y)d\sigma(y) 
-Q^2(x)
\int_{B\cap\{\nu(x)\cdot y=0\}}\varphi(y)d\sigma(y) \\
=
\bigl(\vert \q(x)\vert^{2}-Q^2(x)\bigr)
\int_{B\cap\{\nu(x)\cdot y=0\}}\varphi(y) d\sigma(y)\text{.} 
\end{multline*}

Now by choosing $\varphi\in C^{\infty}_{c}(B)$ 
such that 
\begin{equation*}
\int_{B\cap\{\nu(x)\cdot y=0\}}\varphi(y) d\sigma(y)
\not=0
\end{equation*}
we obtain that $\vert \q(x)\vert=Q(x)$. 
It follows that 
\begin{equation}\label{equ6100} 
\q(x)=\w(x)Q(x)\text{,}
\end{equation} 
where $\w$ is defined in \eqref{equ500}. 

By \eqref{equ5930} and \eqref{equ6100} we obtain \eqref{equ600} 
and this finishes the proof of the theorem. 
\end{proof}

\subsection{$\{\vert \uu\vert>0\}$ is a Non-tangentially Accessible Domain}
\label{subsection1600}  

\begin{proof}[Proof of Theorem \ref{theorem1350}]
We might assume that $x_{0}=0$. 
By Corollary \ref{corollary1000} we have 
\begin{equation*}
[\uu]_{C^{0,1}(B_{\frac{1}{3}r_{0}})}\leq CQ_{max}\text{.}
\end{equation*}

Thus for $K=C_1CQ_{max}$ we have 
\begin{equation*}
[\mathcal{U}]_{C^{0,1}(B_{\frac{1}{3}r_{0}})}\leq K
\end{equation*}
where $\mathcal{U}$ is defined in \eqref{equ3600}. 

For $x\in B_{\frac{1}{3}r_{0}}\cap\{\vert \uu\vert>0\}$ 
let $\delta=\operatorname{dist}(x,\{\vert \uu\vert=0\})$. 
Then $\delta\leq\vert x\vert<\frac{1}{3}r_{0}$, thus 
\begin{equation*}
B_{\delta}(x)\subset B_{\frac{1}{3}r_{0}}(x)\subset B_{\frac{1}{3}r_{0}+\vert x\vert}\subset B_{\frac{2}{3}r_{0}}\text{.} 
\end{equation*}
Because $u_i$ is harmonic in $B_{\delta}(x)$ using Poisson representation and 
Theorem \ref{theorem1200} we have 
\begin{equation*}
\mathcal{U}(x)
=
\fint_{\partial B_{\delta}(x)}\mathcal{U}d\sigma
\geq c_{2}\sup_{B_{\frac{1}{2}\delta}(x)}\mathcal{U}
\geq c_{2}c_{3}\sup_{B_{\frac{1}{2}\delta}(x)}\vert \uu\vert 
\geq c_{2}c_3c_{\rho} Q_{min} \frac{1}{2}\delta
>k\delta
\end{equation*}
where $0<\rho<1$ (and in the case when $n=2$ and $u$ is a local minimizer then $\rho$ is small enough) 
and we have denoted $k=c_{2}c_3c_{\rho} Q_{min} \frac{1}{4}$. 
Thus 
\begin{equation*}
\mathcal{U}(x)>k\operatorname{dist}(x,\{\vert \uu\vert=0\})
\enskip\text{for}\enskip x\in B_{\frac{1}{3}r_{0}}\cap\{\vert \uu\vert>0\}\text{.}
\end{equation*}

Now one follows the proof given in Section 4 of \cite{AguileraCaffarelliSpruck1987} 
by considering the function $\mathcal{U}$ in the domain 
$B_{\frac{1}{3}r_{0}}\cap\{\vert \uu\vert>0\}$. 
\end{proof}

\subsection{Reduction to nondegenerate scalar weak solution}\label{subsection1700}

In the following $u$ is a minimizer and 
$x_{0},r_{0},\epsilon_1,\tilde\epsilon_1,M$ and $c$ are 
as in Theorem \ref{theorem1350}. 

In this subsection let us denote 
\begin{equation*}
D_{1}=B_{\epsilon_{1} r_{0}}(x_{0})\cap\{\vert \uu\vert>0\}\text{.} 
\end{equation*} 

By the definition of a non-tangentially accessible domain 
(see Definition \ref{definition4300} in the appendix) 
for any $y\in \partial D_{1}$ 
and $0<r<\xi=\tilde\epsilon_1 r_0$ there exists 
$a_{r}(y)\in D_{1}\cap B_{r}(y)$ such that 
$\operatorname{dist}\bigl(a_{r}(y),\partial D_{1}\bigr)>\frac{r}{M}$. 

Let the constants $0<\lambda<1$, $0<c_1<1$, $0<c_2<1$, 
$C_{3}>1$ and $C_{4}>0$ be as in Lemma \ref{lemma3400} for the 
non-tangentially accessible domain $D_{1}$.

\begin{lemma}\label{lemma2200} 
There exists $0<\epsilon_{2}<1$ such that 
\begin{equation}\label{equ6300} 
\frac{1}{mC_{4}}\mathcal{U}(y)\leq u_{i_{0}}(y)
\enskip\text{for}\enskip y\in B_{\epsilon_{2}\epsilon_{1}r_{0}}(x_{0})\cap\{\vert\uu\vert>0\}
\end{equation}
where $i_{0}\in\{1,\cdots,m\}$ is such that 
\begin{equation}\label{equ6310} 
u_{i_{0}}(a_{r_{1}}(x_0))=\max_{i\in\{1,\cdots,m\}}u_{i}(a_{r_{1}}(x_0))
\end{equation}
with  
\begin{equation}\label{equ6310.3} 
r_{1}=\frac{1}{c_2}\epsilon_{2}\epsilon_1r_0
=\min(\frac{1}{2}c_{1}\tilde\epsilon_{1},\frac{\epsilon_{1}}{2C_3})
r_0
\text{.}
\end{equation}
\end{lemma}
\begin{proof}
Let $r_{1}=\min(\frac{1}{2}c_{1}\tilde\epsilon_{1}r_0,\frac{1}{2C_3}\epsilon_{1}r_0)$. 
We have 
\begin{multline}\label{equ6310.5} 
B_{C_3r_{1}}(x_0)\cap\partial D_{1}
\subset B_{\frac{1}{2}\epsilon_{1} r_0}(x_0)\cap\partial (B_{\epsilon_{1} r_0}(x_0)\cap\{\vert u\vert>0\}) \\
=B_{\frac{1}{2}\epsilon_{1} r_0}(x_0)\cap\partial \{\vert u\vert>0\}\text{.} 
\end{multline}

Thus we have that $\uu$ vanishes on $B_{C_{3}r_{1}}(x_0)\cap\partial D_1$ 
and also we have $r_{1}<c_1\xi$ ($\xi=\tilde\epsilon_{1}r_{0}$). 
From \eqref{equ6310.5}, because of non-tangential accessibility and the first inequality in \eqref{equ7700} 
for all $i\in\{1,\cdots,m\}$ we have 
\begin{equation}\label{equ6400}
\frac{1}{C_{4}}\frac{u_{i}(a_{r_{1}}(x_0))}{\mathcal{U}(a_{r_{1}}(x_0))}\mathcal{U}(y)
\leq u_{i}(y)\enskip\text{for}\enskip y\in B_{c_{2}r_{1}}(x_0)\cap D_{1}\text{.} 
\end{equation}

Let us define 
\begin{equation*}
\epsilon_{2}=\frac{c_2 r_{1}}{\epsilon_1 r_0}
\end{equation*}
it is easy to see that $0<\epsilon_{2}<\frac{1}{2}c_2<1$. 
Now for $i_{0}\in\{1,\cdots,m\}$ such that \eqref{equ6310} holds, we have 
\begin{equation}\label{equ6500}
u_{i_{0}}(a_{r_{1}}(x_0))\geq\frac{1}{m}\mathcal{U}(a_{r_{1}}(x_0))\text{.}
\end{equation}
From \eqref{equ6400} and \eqref{equ6500} the assertion follows. 
\end{proof}

In the following $\epsilon_{2}$ and $r_{1}$ are as in Lemma \ref{lemma2200}. 

\begin{lemma}\label{lemma2300} 
If $j\in\{1,\cdots,m\}$ such that $u_j>0$ in $D_{1}$ then we have 
\begin{equation*}
[\frac{u_{j}}{\vert \uu\vert}]_{C^{\lambda}(B_{\epsilon_{2}\epsilon_1 r_0}(x_0)\cap \{\vert \uu\vert>0\})}
\leq 
\frac{C_{5}}{r_0^{\lambda}}
\Bigl(
\sum_{i\not=j}
\bigl(
\frac{u_{i}(a_{r_{1}}(x_0))}{u_{j}(a_{r_{1}}(x_0))}
\bigr)^{2}
\Bigr)^{\frac{1}{2}}
\end{equation*}
where 
\begin{equation*}
C_5=(\frac{c_2}{\epsilon_{2}\epsilon_1})^{\lambda}C_4
\end{equation*}
and $r_{1}$ is defined in \eqref{equ6310.3}. 
\end{lemma}

\begin{proof}
Let $j\in\{1,\cdots,m\}$ be such that $u_j>0$ in $D_{1}$. 
By \eqref{equ6310.5} and  \eqref{equ7800} for $i\in\{1,\cdots,m\}$ we have 
\begin{equation}\label{equ6600}
[\frac{u_{i}}{u_{j}}]_{C^{\lambda}(B_{c_{2}r_{1}}(x_0)\cap D_1)}
\leq\frac{C_{4}}{r_{1}^{\lambda}}\frac{u_{i}(a_{r_{1}}(x_0))}{u_{j}(a_{r_{1}}(x_0))}
\text{.}
\end{equation}

Next from 
\begin{equation*}
\vert \uu(y)\vert^2
=\sum_{i=1}^{m}u_{i}^{2}(y)
=u_{j}^{2}(y)
\Bigl(
1+\sum_{i\not=j}(\frac{u_{i}(y)}{u_{j}(y)})^{2}
\Bigr)
\end{equation*}
we have 
\begin{equation*}
\frac{u_{j}(y)}{\vert \uu(y)\vert}
=\Bigl(1+\sum_{i\not=j}(\frac{u_{i}(y)}{u_{j}(y)})^{2}
\Bigr)^{-\frac{1}{2}}
=
\phi\Bigl(
\bigl(\sum_{i\not=j}(\frac{u_{i}(y)}{u_{j}(y)})^{2}\bigr)^{\frac{1}{2}}
\Bigr)
\end{equation*}
where 
\begin{equation*}
\phi(z)=(1+z^{2})^{-\frac{1}{2}}\enskip\text{for}\enskip z\in\mathbb{R}\text{.}
\end{equation*}
The function $\phi$ is Lipschitz continuous in $\mathbb{R}$ with Lipschitz constant bounded by $1$. 
For $y_{1},y_{2}\in B_{c_{2}r_{1}}(x_0)\cap D_1$, using \eqref{equ6600} we have  
\begin{multline*}
\bigl\vert 
\frac{u_{j}(y_{2})}{\vert \uu(y_{2})\vert}
-\frac{u_{j}(y_{1})}{\vert \uu(y_{1})\vert}
\bigr\vert 
=
\Bigl\vert 
\phi\Bigl(
\bigl(\sum_{i\not=j}(\frac{u_{i}(y_{2})}{u_{j}(y_{2})})^{2}\bigr)^{\frac{1}{2}}
\Bigr)
-
\phi\Bigl(
\bigl(\sum_{i\not=j}(\frac{u_{i}(y_{1})}{u_{j}(y_{1})})^{2}\bigr)^{\frac{1}{2}}
\Bigr)
\Bigr\vert \\
\leq 
\Bigl\vert 
\bigl(
\sum_{i\not=j}(\frac{u_{i}(y_{2})}{u_{j}(y_{2})})^{2}
\bigr)^{\frac{1}{2}}
-
\bigl(\sum_{i\not=j}(\frac{u_{i}(y_{1})}{u_{j}(y_{1})})^{2}\bigr)^{\frac{1}{2}}
\Bigr\vert 
\leq 
\Bigl(
\sum_{i\not=j}
(
\frac{u_{i}(y_{2})}{u_{j}(y_{2})}
-
\frac{u_{i}(y_{1})}{u_{j}(y_{1})}
)^{2}
\Bigr)^{\frac{1}{2}} \\
\leq 
\Bigl(
\sum_{i\not=j}
(
\frac{C_{4}}{r^{\lambda}}\frac{u_{i}(a_{r_{1}}(x_0))}{u_{j}(a_{r_{1}}(x_0))}
\vert y_{2}-y_{1}\vert^{\lambda}
)^{2}
\Bigr)^{\frac{1}{2}} 
=
\frac{C_{4}}{r^{\lambda}}
\Bigl(
\sum_{i\not=j}
\bigl(
\frac{u_{i}(a_{r_{1}}(x_0))}{u_{j}(a_{r_{1}}(x_0))}
\bigr)^{2}
\Bigr)^{\frac{1}{2}}
\vert y_{2}-y_{1}\vert^{\lambda}
\end{multline*}
and this proves the lemma. 
\end{proof} 

\begin{corollary}\label{corollary2000} 
We have 
\begin{equation*}
[\frac{u_{i_{0}}}{\vert \uu\vert}]_{C^{\lambda}(B_{\epsilon_{2} \epsilon_{1} r_{0}}(x_{0})\cap\{\vert \uu\vert>0\})}
\leq (m-1)^{\frac{1}{2}}
\frac{C_{5}}{r_0^{\lambda}}\text{.}
\end{equation*}
\end{corollary}

\begin{corollary}\label{corollary3000}
For each $K\subset\subset\Omega$, $\w=(w_1,\cdots,w_m)$ (as defined in \eqref{equ500}) 
is uniformly (H\"older) continuous in $K\cap\{\vert \uu\vert>0\}$. 
\end{corollary}

\begin{lemma}\label{lemma2305} 
Let $Q$ be a continuous function 
then for all $x\in\Omega\cap\partial^{*}\{\vert \uu\vert>0\}$ 
as $y\to 0$ we have 
\begin{equation}\label{equ6605} 
\uu(x+y)=\w(x)Q(x)\bigl(-\nu_{\{\vert \uu\vert>0\}}(x)\cdot y\bigr)^{+}
+o(\vert y\vert)\text{.} 
\end{equation}
\end{lemma}
\begin{proof}
For $i=1,\cdots,m$, by Theorem \ref{theorem1300} 
we have 
\begin{equation*}
\Delta u_i=w_i Q \mathcal{H}^{n-1}\llcorner(\Omega\cap\partial^{*}\{\vert \uu\vert>0\})\text{.} 
\end{equation*}

Let $x\in\Omega\cap\partial\{\vert \uu\vert>0\}$ then 
by Corollary \ref{corollary3000} we have that $\w$ is uniformly 
continuous close to $x$. Also we have assumed that $Q$ is continuous. 
Now as in Theorem \ref{theorem1300} by 
Theorem 4.8 in \cite{AltCaffarelli1981} we obtain \eqref{equ6605}. 
\end{proof}

\begin{lemma}\label{lemma2310} 
$u_{i_{0}}$ is weak solution for $w_{i_{0}}Q$ in $B_{\epsilon_{2}\epsilon_1r_0}(x_0)$ 
as defined in Section 5 of \cite{AltCaffarelli1981}, i.e. 
\begin{enumerate}[(i)]
\item\label{lemma2310-1} 
$u_{i_{0}}$ is continuous and non-negative 
in $B_{\epsilon_{2}\epsilon_1r_0}(x_0)$ and harmonic in 
$B_{\epsilon_{2}\epsilon_1r_0}(x_0)\cap\{u_{i_{0}}>0\}$.  
\item\label{lemma2310-2} 
For $K\subset\subset B_{\epsilon_{2}\epsilon_1r_0}(x_0)$ there are constants $0<c\leq C$ 
such that for balls $B_{s}(y)\subset K$ with $y\in\partial\{u_{i_{0}}>0\}$ 
we have 
\begin{equation}\label{equ6610} 
c\leq\frac{1}{s}\fint_{\partial B_{s}(y)}u_{i_{0}}d\sigma \leq C\text{.} 
\end{equation}
\item\label{lemma2310-3} 
The equation 
$\Delta u_{i_{0}}=w_{i_{0}}Q\mathcal{H}^{n-1}\llcorner(B_{\epsilon_{2}\epsilon_1r_0}(x_0)\cap\partial^{*}\{u_{i_{0}}>0\})$ 
holds in $B_{\epsilon_{2}\epsilon_1r_0}(x_0)$ 
(in the sense of distributions). 
\end{enumerate}
\end{lemma}
\begin{proof}
By Lemma \ref{lemma2200} we have 
\begin{equation*}
B_{\epsilon_{2}\epsilon_1 r_0}(x_0)\cap\{\vert \uu\vert>0\}
=B_{\epsilon_{2}\epsilon_1 r_0}(x_0)\cap\{u_{i_{0}}>0\}\text{.} 
\end{equation*} 

Part \eqref{lemma2310-1} is trivial. 
The first inequality in \eqref{equ6610} follows from the 
first inequality in \eqref{equ3700} and the inequality \eqref{equ6300}. 
The second inequality in \eqref{equ6610} follows directly from 
Theorem \ref{theorem1100} and this finishes the proof of 
Part \eqref{lemma2310-2}. 

Part \eqref{lemma2310-3} follows from Theorem \ref{theorem1300} 
and this finishes the proof of the lemma. 
\end{proof}

\section{Flat Free Boundary Points\\ 
(Proof of Theorem \ref{theorem1360}) }\label{section1810} 

In this section $x_0$, $r_0$, $\epsilon_{1}$, $\epsilon_2$ and $i_0$ 
are as in subsection \ref{subsection1700}.

\subsection{Flatness implies regularity}\label{subsection1800}

\begin{proof}[Proof of Theorem \ref{theorem1360}]
By Lemma \ref{lemma2310}, $u_{i_{0}}$ is weak solution for $w_{i_{0}}Q$ in $B_{\epsilon_{2}\epsilon_1r_0}(x_0)$ 
as defined in Section 5 of \cite{AltCaffarelli1981}. 
By Lemma \ref{lemma2300} we know that $w_{i_{0}}Q$ is H\"older continuous 
in $B_{\epsilon_{2}\epsilon_1r_0}(x_0)$. 
Now the Theorem follows from Theorem 8.1 in \cite{AltCaffarelli1981}. 
\end{proof}

\subsection{Equivalence of reduced, regular, and flat free boundary points}
\label{subsection1900} 

\begin{definition}[Flat free boundary points] 
We call $x\in\Omega\cap\partial\{\vert \uu\vert>0\}$ 
a flat free boundary point if for all $0<\sigma<1$ there 
exist sequences $\rho_j\to 0$ and $\nu_j\in\partial B$ 
such that $u$ is $\sigma$-flat in $B_{\rho_j}(x)$ in the direction $\nu_j$ 
(see Definition \ref{definition600}). 
\end{definition}

\begin{definition}\label{definition1300}  
We call $x\in\Omega\cap\partial\{\vert \uu\vert>0\}$ 
a regular point if $Q$ is continuous at $x$, there exists a sequence $r_j\to 0$ 
such that $\uu_{x,r_j}\to Q(x) (\nu\cdot y)^{+}e$ in $C(B)$ where 
$\nu\in\mathbb{R}^{n}$, $\vert \nu\vert=1$, $e\in \mathbb{R}^{m}$, 
$\vert e\vert=1$ and $e_i\geq 0$ for $i=1,\cdots,m$. 
\end{definition}

\begin{lemma}\label{lemma2320} 
Let $Q$ be H\"older continuous and $u$ be a minimizer of our functional. 
Then the following free boundary points coincide 
\begin{enumerate}[(i)]
\item Reduced free boundary points.
\item Regular free boundary points. 
\item Flat free boundary points. 
\item $C^{1,\alpha}$ free boundary points (i.e. Free boundary points 
in a neighborhood of which the free boundary is $C^{1,\alpha}$ smooth). 
\end{enumerate}
\end{lemma}
\begin{proof}
From Lemma \ref{lemma2305} we obtain that at a reduced free boundary point 
there exists a unique halfspace blowup limit thus in particular such a point is 
a regular point. 

Assume $x$ is a regular free boundary point. Then, by definition, there exists 
$r_k\to 0$, $\nu\in\partial B^{n}$, $e\in\partial B^{m}$ with $e_i\geq 0$ for $i=1,\cdots,m$ 
such that $\uu_{x,r_k}\to Q(x)(y\cdot\nu)^{+}e$ in $C(B;\mathbb{R}^{m})$. 
From Lemma \ref{lemma2100} it follows that there 
exists a subsequence $k_{j}$ such that 
$Q(x)(y\cdot\nu)^{+}e$ is the blowup limit of $\uu$ at $x$ with respect to the 
sequence $r_{k_{j}}$. 
Now by the Hausdorff convergence of the coincidence sets 
it follows that for given $0<\sigma<1$, for $j$ large enough 
we have $\uu=0$ in $B_{\frac{1}{2}r_{k_j}}(x)\cap\{-(y-x)\cdot\nu \geq \frac{1}{2} \sigma r_{k_j}\}$. 
This proves that $x$ is a flat free boundary point. Now
if $x$ is a flat free boundary point then by Theorem \ref{theorem1360}, 
in a neighborhood of $x$ the free boundary is $C^{1,\alpha}$ smooth. 

Finally, if the free boundary is $C^{1,\alpha}$ smooth in a neighborhood of the free 
boundary point $x$ then $x$ is a reduced free boundary point. 
\end{proof}

\begin{corollary} 
Let $Q$ be H\"older continuous and $u$ be a minimizer for $Q$ 
then the reduced boundary is an open subset of the free boundary 
in the relative topology of the free boundary. 
\end{corollary}

\section{Monotonicity Formula}\label{section1900}  

\begin{lemma}[Basic Energy Identity]\label{lemma2400} 
Let $\uu$ be a minimizer and $B_{r_{0}}(x_0)\subset\Omega$ then 
for a.e. $r\in(0,r_{0})$ we have 
\begin{equation}\label{equ6700} 
\int_{B_r(x_0)}\vert\nabla u_i\vert^{2}dx
=\int_{\partial B_r(x_0)}u_i\partial_{\nu}u_i d\sigma
\end{equation}
for $i=1,\cdots,m$. 
\end{lemma}
\begin{proof}
We might assume that $x_0=0$. 
We know that $u_i$ is harmonic in $\{\vert \uu\vert>0\}$ and 
$u_i\in H^{1}(\Omega)$. 
Let $0<r<r_{0}$. For $0<\epsilon<r$ let us define 
\begin{equation}\label{equ6800}
\varphi_{r,\epsilon}(x)
=\max\Bigl(0,\min\bigl(1,\frac{r}{\epsilon}(1-\frac{\vert x\vert}{r})\bigr)\Bigr)\text{.} 
\end{equation}
We compute 
\begin{equation*}
\nabla\varphi_{r,\epsilon}(x)
=
-
\chi_{\{r-\epsilon<\vert x\vert<r\}}
\frac{1}{\epsilon}
\frac{x}{\vert x\vert}
\text{.} 
\end{equation*}
We have 
$u_i\varphi_{r,\epsilon}\in H^{1}_{0}(B_r\cap\{\vert \uu\vert>0\})$. 
We compute  
\begin{multline*}
0
=\int_{B_r\cap\{\vert \uu\vert >0 \}}\nabla u_i\cdot\nabla (\varphi_{r,\epsilon}u_i)dx  \\
=\int_{B_r}\vert\nabla u_i\vert^{2}  \varphi_{r,\epsilon}  dx
-\frac{1}{\epsilon}\int_{B_r\backslash \overline{B_{r-\epsilon}}}
\nabla u_i\cdot\frac{x}{\vert x\vert}u_i dx 
\text{.}
\end{multline*}
Letting $\epsilon\to 0$ we obtain that for a.e. $r\in(0,r_{0})$ 
\begin{equation*}
\int_{B_r}\vert\nabla u_i\vert^{2}dx
=
\int_{\partial B_{r}}
\nabla u_i 
\cdot\frac{x}{\vert x\vert}
u_i d\sigma(x)
=
\int_{\partial B_{r}}
u_i\partial_{\nu}u_i 
d\sigma(x) ,
\end{equation*}
and this finishes the proof of the lemma.
\end{proof}

\begin{lemma}[Poho\v{z}aev Type Identity]\label{lemma2500}  
Let $\uu$ be a minimizer with continuous $Q$ and 
$B_{r_{0}}(x_{0})\subset\Omega$ then for a.e. $r\in(0,r_{0})$ 
the following Poho\v{z}aev type identity holds 
\begin{multline}\label{equ6900} 
n\int_{B_r(x_0)}\vert\nabla \uu\vert^{2}dx  \\
=
2\int_{B_r(x_0)}\vert\nabla \uu\vert^{2}dx  
+
r\int_{\partial B_r(x_0)}
\bigl(
\vert\nabla \uu\vert^{2}
-2\vert\partial_{\nu} \uu\vert^{2}
\bigr)d\sigma(x) \\
+\int_{B_r(x_0)}Q^2(x) (x-x_0)\cdot d\mu(x) ;
\end{multline}
the vector Radon measure $\mu$ is define in \eqref{equ3850}. 
\end{lemma}
\begin{proof}
We might assume that $x_0=0$. 
Let $r\in(0,r_{0})$ and $\varphi_{r,\epsilon}$ be as in \eqref{equ6800}. 
Let $\Psi(x)=x\varphi_{r,\epsilon}(x)$. 
We compute 
\begin{equation*}
\nabla\Psi(x)=\varphi_{r,\epsilon}(x)I+\nabla\varphi_{r,\epsilon}(x)x^{T}
\enskip 
\text{and}
\enskip 
\operatorname{div}\Psi(x)
=
n\varphi_{r,\epsilon}(x)+x\cdot\nabla\varphi_{r,\epsilon}(x)\text{.} 
\end{equation*}
Now by the Lemma \ref{lemma1900} we obtain that 
\begin{multline*}
\int_{\Omega}
\Bigl(
2
\bigl(
\varphi_{r,\epsilon}\vert\nabla \uu\vert^{2}+\sum_{i=1}^{m}\nabla\varphi_{r,\epsilon}\cdot\nabla u_i x\cdot\nabla u_i
\bigr)
-\vert\nabla \uu\vert^{2}
\bigl(
n\varphi_{r,\epsilon}+x\cdot\nabla\varphi_{r,\epsilon}
\bigr)
\Bigr)dx \\
+\int_{\Omega}Q^2
\varphi_{r,\epsilon}
x\cdot
d\mu(x)
=0\text{.}
\end{multline*}

Passing to the limit as $\epsilon\to 0$ we obtain that for a.e. $r\in(0,r_{0})$ 
\begin{multline*}
n\int_{B_r}\vert\nabla \uu\vert^{2}dx 
=
2\int_{B_r}\vert\nabla \uu\vert^{2}dx 
+
r\int_{\partial B_r}
\bigl(
\vert\nabla \uu\vert^{2}
-2\vert\partial_{\nu} \uu\vert^{2}
\bigr)d\sigma(x) \\
+\int_{B_r}Q^2 x\cdot d\mu(x)
\end{multline*}
which proves the lemma. 
\end{proof}

Let us define 
\begin{equation*}
W(r,\uu;Q)
=\frac{1}{r^n}\int_{B_r}\bigl(\vert \nabla \uu\vert^{2}
+Q^2(x)\chi_{\{\vert \uu\vert>0\}}\bigr)dx
-\frac{1}{r^{n+1}}\int_{\partial B_r}\vert \uu\vert^{2}d\sigma(x)\text{.} 
\end{equation*}

\begin{lemma}[Weiss type Monotonicity Formula]\label{lemma2600} 
For $r,s>0$ and $\uu\in H^{1}(B_{rs};\mathbb{R}^{m})$ we have $W(rs,\uu;Q)=W(s,\uu_r;Q_r)$. 
For $\uu\in H^{1}(B_{r_0};\mathbb{R}^{m})$, 
$W(r,\uu;Q)$ as a function of $0<r<r_0$ is locally bounded and 
absolutely continuous. 
Let $\uu$ be a minimizer in $B_{r_0}$. 
If $Q=Q(0)$ is constant then for a.e. $r\in(0,r_{0})$ we have 
\begin{equation}\label{equ7000}
\frac{d}{dr}W(r,\uu;Q(0))
=2r\int_{\partial B_1} \vert \partial_{r}\uu_{r}\vert^{2} d\sigma(x)
\end{equation}
and generally for continuous $Q$ for a.e. $r\in(0,r_{0})$ we have 
\begin{equation}\label{equ7100}
\frac{d}{dr}W(r,\uu;Q)
\geq 
2r\int_{\partial B_1} \vert \partial_{r}\uu_{r}\vert^{2} d\sigma(x)
-CQ_{max}\frac{1}{r}\operatorname{osc}_{B_r}Q\text{.} 
\end{equation}
For $\uu$ a first order homogenous minimizer in $B_1$ we have 
\begin{equation*}
W(r,\uu;Q(0))
=
Q^2(0)
\int_{B}
\chi_{\{\vert \uu\vert>0\}}dx \\
=
\frac{1}{n}Q^2(0)
\int_{\partial B}\chi_{\{\vert \uu\vert>0\}}d\sigma(x)
\text{.}
\end{equation*}
\end{lemma}
\begin{proof}
For $r,s>0$ and $\uu\in H^{1}(B_{rs})$ we compute 
\begin{multline*}
W(rs,\uu;Q)
=\frac{1}{(rs)^n}\int_{B_{rs}}\bigl(\vert \nabla \uu\vert^{2}
+Q^2\chi_{\{\vert \uu\vert>0\}}\bigr)dx
-\frac{1}{(rs)^{n+1}}\int_{\partial B_{rs}}\vert \uu\vert^{2}d\sigma(x) \\
=\frac{1}{s^n}\int_{B_{s}}\bigl(\vert \nabla \uu(rx)\vert^{2}
+Q^2(rx)\chi_{\{\vert \uu(rx)\vert>0\}}\bigr)dx
-\frac{1}{s^{n+1}r^{2}}\int_{\partial B_{s}}\vert \uu(rx)\vert^{2}d\sigma(x) \\
=\frac{1}{s^n}\int_{B_{s}}\bigl(\vert \nabla \uu_r\vert^{2}
+Q^2_r(x)\chi_{\{\vert \uu_r\vert>0\}}\bigr)dx
-\frac{1}{s^{n+1}}\int_{\partial B_{s}}\vert \uu_r\vert^{2}d\sigma(x) 
=W(s,\uu_r;Q_r)
\end{multline*} 
and this proves the first claim. 

Let $\uu\in H^{1}(B_{r_0};\mathbb{R}^{m})$ then for $0<r<r_0$ by direct computation using polar coordinates 
we have 
\begin{multline}\label{equ7200} 
\int_{\partial B_{r}}
u_i^2
d\sigma(x)
=
-
2r^{n-1}\int_{B_{r_0}\backslash B_r}\frac{1}{\vert x\vert^{n}}u_i(x)\nabla u_i(x)\cdot xdx \\
+
(\frac{r}{r_0})^{n-1}
\int_{\partial B_{r_0}}
u_i^2
d\sigma(x)
\end{multline}

The equation \eqref{equ7200} together with the fact that for 
$f\in L^{1}_{loc}(\mathbb{R}^{n})$, $\int_{B_r}fdx$ 
as a function of $r$ is bounded and absolutely continuous 
proves the second claim. 

Let $\uu$ be a minimizer in $B_{r_0}$ and $0<r<r_0$. 
By scaling in the second integral in the definition of $W$ we obtain 
\begin{equation*}
W(r,\uu;Q)=r^{-n}\int_{B_r}\bigl(\vert \nabla \uu\vert^{2}
+Q^2\chi_{\{\vert \uu\vert>0\}}\bigr)dx
-r^{-2}\int_{\partial B_1}\vert \uu\vert^{2}(ry)d\sigma(y)\text{.} 
\end{equation*}
Computing the derivative with respect to $r$, for a.e. $r\in(0,r_{0})$ we obtain 
\begin{multline*}
\frac{1}{2}r^{n+1}\frac{d}{dr}W(r,\uu;Q)
=-\frac{1}{2}n\int_{B_r}\vert \nabla \uu\vert^{2}dx 
-\frac{n}{2}\int_{B_r}Q^2\chi_{\{\vert \uu\vert>0\}}dx \\
+\frac{r}{2}\int_{\partial B_r}\bigl(\vert \nabla \uu\vert^{2}
+Q^2\chi_{\{\vert \uu\vert>0\}}\bigr)d\sigma(x) \\
+\frac{1}{r}\int_{\partial B_r}\vert \uu\vert^{2}d\sigma(x) 
-\int_{\partial B_r}\sum_{i=1}^{m} u_i \partial_{\nu}u_i d\sigma(x)\text{.} 
\end{multline*}

Using the identity \eqref{equ6900} proved in the previous 
lemma for the first integral on 
the right hand side we obtain 
(in the following the vector Radon measure $\mu$ is defined in \eqref{equ3850})
\begin{multline*}
\frac{1}{2}r^{n+1}\frac{d}{dr}W(r,\uu;Q) 
=
-\frac{1}{2}
\Bigl(
2\int_{B_r}\vert\nabla \uu\vert^{2}dx  \\
+
r\int_{\partial B_r} 
\bigl(
\vert\nabla \uu\vert^{2}
-2\vert\partial_{\nu} \uu\vert^{2}
\bigr)d\sigma(x) 
+\int_{B_r}Q^2(x) x\cdot d\mu(x)
\Bigr) \\
-\frac{n}{2}\int_{B_r}Q^2\chi_{\{\vert \uu\vert>0\}}dx 
+\frac{r}{2}\int_{\partial B_r}\bigl(\vert \nabla \uu\vert^{2}
+Q^2\chi_{\{\vert \uu\vert>0\}}\bigr)d\sigma(x) \\
+\frac{1}{r}\int_{\partial B_r}\vert \uu\vert^{2}d\sigma(x) 
-\int_{\partial B_r}\sum_{i=1}^{m} u_i \partial_{\nu}u_i d\sigma(x) \\
=
-\int_{B_r}\vert\nabla \uu\vert^{2}dx 
+r\int_{\partial B_r}\vert\partial_{\nu} \uu\vert^{2} d\sigma(x) 
+\frac{1}{r}\int_{\partial B_r}\vert \uu\vert^{2}d\sigma(x)  \\
-\int_{\partial B_r}\sum_{i=1}^{m} u_i \partial_{\nu}u_i d\sigma(x) 
-\frac{1}{2}\int_{B_r}Q^2(x) x\cdot d\mu(x) \\
-\frac{n}{2}\int_{B_r}Q^2\chi_{\{\vert \uu\vert>0\}}dx 
+\frac{r}{2}\int_{\partial B_r}Q^2\chi_{\{\vert \uu\vert>0\}}d\sigma(x)
\text{.} 
\end{multline*}

Using the identity \eqref{equ6700} proved in the Lemma \ref{lemma2400} for the first integral on 
the right hand side we obtain for a.e. $r\in(0,r_{0})$ 
\begin{multline}\label{equ7300}
\frac{1}{2}r^{n+1}\frac{d}{dr}W(r,\uu;Q) \\
=r\int_{\partial B_r}\vert\partial_{\nu} \uu\vert^{2} d\sigma(x) 
+\frac{1}{r}\int_{\partial B_r}\vert \uu\vert^{2}d\sigma(x) 
-2\int_{\partial B_r}\sum_{i=1}^{m} u_i \partial_{\nu}u_i d\sigma(x) \\
-\frac{1}{2}\int_{B_r}Q^2(x) x\cdot d\mu(x)
-\frac{n}{2}\int_{B_r}Q^2\chi_{\{\vert \uu\vert>0\}}dx 
+\frac{r}{2}\int_{\partial B_r}Q^2\chi_{\{\vert \uu\vert>0\}}d\sigma(x)\text{.}
\end{multline}

Separately we compute 
\begin{multline}\label{equ7400}
r\int_{\partial B_r}\vert\partial_{\nu} \uu\vert^{2} d\sigma(x) 
+\frac{1}{r}\int_{\partial B_r}\vert \uu\vert^{2}d\sigma(x) 
-2\int_{\partial B_r}\sum_{i=1}^{m} u_i \partial_{\nu}u_i d\sigma(x) \\
=
\frac{1}{r}
\sum_{i=1}^{m}
\int_{\partial B_r}
\bigl(
r\partial_{\nu} u_i-u_i
\bigr)^{2}
d\sigma(x) 
=
r^{n+2}\int_{\partial B_1} \vert \partial_{r}\uu_{r}\vert^{2} d\sigma(x)\text{.} 
\end{multline}

One may see that 
\begin{equation}\label{equ7500} 
-\frac{1}{2}\int_{B_r}x\cdot d\mu(x)
-\frac{n}{2}\int_{B_r}\chi_{\{\vert \uu\vert>0\}}dx 
+\frac{r}{2}\int_{\partial B_r}\chi_{\{\vert \uu\vert>0\}}d\sigma(x)
=0\text{.} 
\end{equation}
From \eqref{equ7300}, \eqref{equ7400} and \eqref{equ7500} we obtain 
\eqref{equ7000}. 

Using Lemma \ref{lemma1700} part \eqref{lemma1700-4} we further compute 
\begin{multline}\label{equ7600} 
-\frac{1}{2}\int_{B_r}Q^2(x) x\cdot d\mu(x) \\
-\frac{n}{2}\int_{B_r}Q^2\chi_{\{\vert \uu\vert>0\}}dx 
+\frac{r}{2}\int_{\partial B_r}Q^2\chi_{\{\vert \uu\vert>0\}}d\sigma(x) \\
=
-\frac{1}{2}\int_{B_r}(Q^2(x)-Q^2(0)) x\cdot d\mu(x) \\
-\frac{n}{2}\int_{B_r}(Q^2(x)-Q^2(0))\chi_{\{\vert \uu\vert>0\}}dx  \\
+\frac{r}{2}\int_{\partial B_r}(Q^2(x)-Q^2(0))\chi_{\{\vert \uu\vert>0\}}d\sigma(x) \\
\geq 
-C_1Q_{max}r\operatorname{osc}_{B_r}Q\int_{B_r}d\vert \mu\vert(x) \\
-C_1Q_{max}\operatorname{osc}_{B_r}Q\int_{B_r} \chi_{\{\vert \uu\vert>0\}}dx  \\
-C_1Q_{max}r\operatorname{osc}_{B_r}Q\int_{\partial B_r} \chi_{\{\vert \uu\vert>0\}} d\sigma(x) \\
\geq 
-C_2Q_{max}r^{n}\operatorname{osc}_{B_r}Q\text{.} 
\end{multline} 

By \eqref{equ7300}, \eqref{equ7400} and \eqref{equ7600} we obtain 
\begin{equation*}
\frac{1}{2}r^{n+1}\frac{d}{dr}W(r,\uu;Q)
\geq 
r^{n+2}\int_{\partial B_1} \vert \partial_{r}\uu_{r}\vert^{2} d\sigma(x)
-C_{2}Q_{max}r^{n}\operatorname{osc}_{B_r}Q
\text{;}
\end{equation*}
from which \eqref{equ7100} follows. 

For $\uu$ a minimizer in $B_1$ with constant $Q=Q(0)$ 
using the identity \eqref{equ6700} proved in the Lemma \ref{lemma2400} we compute 
\begin{multline*}
W(1,\uu;Q(0))
=
\int_{B_1}\vert \nabla \uu\vert^{2}dx
+\int_{B_1}Q^2(0)\chi_{\{\vert \uu\vert>0\}}dx
-\int_{\partial B_1}\vert \uu\vert^{2}d\sigma(x) \\
=
\int_{\partial B_1}\sum_{i=1}^{m}u_i\partial_{\nu}u_id\sigma(x)
+\int_{B_1}Q^2(0)\chi_{\{\vert \uu\vert>0\}}dx
-\int_{\partial B_1}\vert \uu\vert^{2}d\sigma(x) \\
=
\int_{B_1}Q^2(0)\chi_{\{\vert \uu\vert>0\}}dx
+\int_{\partial B_1}\sum_{i=1}^{m}\bigl(u_i\partial_{\nu}u_i-u_i^2\bigr)d\sigma(x) \\
=
\int_{B_1}Q^2(0)\chi_{\{\vert \uu\vert>0\}}dx
+\int_{\partial B_1}\sum_{i=1}^{m}(\partial_{\nu}u_i-u_i)u_id\sigma(x) \\
=
\int_{B_1}Q^2(0)\chi_{\{\vert \uu\vert>0\}}dx
+\int_{\partial B_1}(\partial_{\nu}\uu-\uu)\cdot \uu d\sigma(x)\text{.} 
\end{multline*}
For a first order homogenous function we have $\partial_{\nu}\uu=\uu$ thus 
the last integral vanishes and this proves the last claim of the lemma. 
\end{proof}

\begin{lemma}\label{lemma2700} 
Let $\uu$ be a local minimizer with continuous $Q$ 
in $\Omega$ with $0\in\Omega$ and $\uu(0)=0$. 
Assume that $Q$ Dini continuous at origin, i.e. 
\begin{equation}\label{equ7610}
\int_{0+}\frac{1}{r}\operatorname{osc}_{B_r}Q dr<\infty\text{.} 
\end{equation}
Let $B_{r_k}\subset\Omega$ with $r_k\to 0$ as $k\to\infty$. 
Let $\vv$ be a blowup limit of $\uu$ at the origin with respect to the sequence $r_k$. 
Then $\vv$ is first order homogenous.
\end{lemma}
\begin{proof}
Let $\uu$, $Q$, $r_k$ and $\vv$ as in the statement of the Lemma. 

Let us define 
\begin{equation*}
\rho(r)=\int_{0}^{r}\frac{1}{s}\operatorname{osc}_{B_{s}}Q ds\text{.} 
\end{equation*}
By \eqref{equ7610}, $\rho(r)$ is well defined for small enough $r>0$.  

Let $B_{r_{0}}\subset\Omega$. 
By \eqref{equ7100} in Lemma \ref{lemma2600} we have that 
for a.e. $r\in(0,r_{0})$ 
\begin{equation*}
\frac{d}{dr}W(r,\uu;Q)
\geq 
-CQ_{max}\frac{1}{r}\operatorname{osc}_{B_r}Q
=
-CQ_{max}\rho^{\prime}(r)
\text{.} 
\end{equation*}
Thus we have 
\begin{equation*}
\frac{d}{dr}\bigl(W(r,\uu;Q)+CQ_{max}\rho(r)\bigr)\geq 0\text{.} 
\end{equation*}
Therefore $W(r,\uu;Q)+CQ_{max}\rho(r)$ is a nondecreasing and absolutely 
continuous function. 
It follows that the limit $\lim_{r\to0,\ r>0}(W(r,\uu;Q)+CQ_{max}\rho(r))$ exists. 
Because $\lim_{r\to0,\ r>0}\rho(r)=0$ we obtain that 
the limit 
\begin{equation*} 
W(+0,\uu;Q)=\lim_{r\to0,\ r>0}W(r,\uu;Q)
\end{equation*} 
exists. 
Because $\uu(0)=0$ one may see that by regularity results $W(+0,\uu;Q)>-\infty$. 

For $s>0$ we have $W(r_k s,\uu;Q)=W(s,\uu_{r_k};Q_{r_k})$. 
Clearly we have $W(+0,\uu;Q)=\lim_{k\to\infty}W(r_k s,\uu;Q)$, and 
\begin{multline*}
W(s,\uu_{r_k};Q_{r_k}) \\
=
\frac{1}{s^n}\int_{B_s}\bigl(\vert \nabla \uu_{r_k}\vert^{2}
+Q_{r_k}^{2}\chi_{\{\vert \uu_{r_k}\vert>0\}}\bigr)dx
-\frac{1}{s^{n+1}}\int_{\partial B_s}\vert \uu_{r_k}\vert^{2}d\sigma(x) \\
\to 
\frac{1}{s^n}\int_{B_s}\bigl(\vert \nabla \vv\vert^{2}
+Q^{2}(0)\chi_{\{\vert \vv\vert>0\}}\bigr)dx
-\frac{1}{s^{n+1}}\int_{\partial B_s}\vert \vv\vert^{2}d\sigma(x) \\
=
W(s,\vv;Q(0))\text{.} 
\end{multline*}
Thus $-\infty<W(+0,\uu;Q)=W(s,\vv;Q(0))$ for $0<s<1$. 
By Lemma \ref{lemma2000} we know that $\vv$ is an 
absolute minimizer thus by Lemma \ref{lemma2600} 
because $W(s,\vv;Q(0))$ is independent of $s$ it 
follows that $\vv$ is first order homogenous. 
\end{proof}

\section{Homogenous Global Minimizers\\ 
(Proof of Theorem \ref{theorem1400})}\label{section2000} 

In this section we use the notation $S^{n-1}=\partial B$.  
Assume $Q_{0}>0$. Then one may see that 
$\uu$ is a minimizer in $\Omega$ with $Q=Q_{0}$ 
if and only if $\vv=\frac{1}{Q_{0}}\uu$ is a minimizer in $\Omega$ with $Q=1$. 
This allows us in the following to consider only the case $Q=1$. 

\begin{lemma}\label{lemma2800} 
If $v$ is an absolute scalar minimizer in $B$, 
$c\in\mathbb{R}^{m}$, $\vert c\vert=1$, $c_i\geq 0$ for $i=1,\cdots,m$ and 
$u_i=c_iv$ then $\uu$ is an absolute (vector) minimizer in $B$. 
\end{lemma}
\begin{proof}
For $i=1,\cdots,m$ let $w_i\in H^{1}(B;\mathbb{R}^{m})$, $w_i\geq 0$ in $B$, 
$w_i=u_i=c_i v$ on $\partial B$. Then 
\begin{multline*}
\int_{B}\bigl(\vert\nabla \uu\vert^{2}+\chi_{\{\vert \uu\vert>0\}}\bigr)dx 
=
\int_{B}\bigl(\vert\nabla v\vert^{2}+\chi_{\{v>0\}}\bigr)dx \\
=
\sum_{i=1}^{m}c_i^2\int_{B}\bigl(\vert\nabla v\vert^{2}+\chi_{\{v>0\}}\bigr)dx 
\leq 
\sum_{i=1}^{m}c_i^2\int_{B}\bigl(\vert\nabla(\frac{w_i}{c_i})\vert^{2}
+\chi_{\{\frac{w_i}{c_i}>0\}}\bigr)dx \\
=
\int_{B}\bigl(\vert\nabla\w\vert^{2}
+\sum_{i=1}^{m}c_i^2 \chi_{\{w_i>0\}}\bigr)dx 
\leq 
\int_{B}\bigl(\vert\nabla\w\vert^{2}
+\sum_{i=1}^{m}c_i^2 \chi_{\{\vert\w\vert>0\}}\bigr)dx  \\
=
\int_{B}\bigl(\vert\nabla\w\vert^{2}
+ \chi_{\{\vert\w\vert>0\}}\bigr)dx
\end{multline*} 
which finishes the proof of the lemma. 
\end{proof}

\begin{lemma}\label{lemma2900} 
Suppose $\uu$ is a first order homogenous 
absolute minimizer in $B$, $\{\vert \uu\vert>0\}$ is a connected 
open set and $\{\vert \uu\vert>0\}\not=B\backslash\{0\}$. 
Then $u_i=c_iv$ where $c\in\mathbb{R}^{m}$, $\vert c\vert=1$, $c_i\geq 0$ 
for $i=1,\cdots,m$ and $v$ is a scalar first order homogenous absolute 
minimizer with $\{v>0\}=\{\vert \uu\vert>0\}$.  
\end{lemma}
\begin{proof}
Let us define the set 
\begin{equation*}
U=\{\vert \uu\vert>0\}\cap S^{n-1}\text{.} 
\end{equation*}

$U$ is an open and connected strict subset of $S^{n-1}$. 
Because $\uu$ is harmonic in the cone $\{\vert \uu\vert>0\}$ 
and first order homogenous we obtain that 
for all $i=1,\cdots,m$ we have 
\begin{equation*}
\left\{
\begin{aligned}
& -\Delta_{S^{n-1}}u_i=(n-1)u_i\enskip\text{on}\enskip U\text{,} \\
& u_i=0\enskip\text{on}\enskip \partial_{S^{n-1}} U\text{.} 
\end{aligned}
\right.
\end{equation*}
Here $\Delta_{S^{n-1}}$ is the Laplacian on the sphere and 
$\partial_{S^{n-1}} U$ is the boundary of $U$ in the sphere. 
It follows that $u_i$ are in the eigenspace corresponding to the eigenvalue $n-1$. 
Because $u_i$ are nonnegative it follows that $n-1$ is the first eigenvalue. 
Since the first eigenvalue is simple, $u_i$ are in a one dimensional space. 
Let $u_i=a_i w$ for $w$ a fixed eigenfunction corresponding to the first eigenvalue. 
Let us define $c=\frac{a}{\vert a\vert}$ and $v=\vert a\vert w$. 

Now let us show that $v$ is an absolute scalar minimizer in $B$. 

For $\phi\in H^{1}(B)$, $\phi\geq 0$ a.e. in $B$ 
and $\phi=v$ on $\partial B$, define $w_i=c_i\phi$. 
Then $w_i=c_iv=u_i$ on $\partial B$, and 
\begin{multline*}
\int_{B}\bigl(\vert\nabla v\vert^{2}+\chi_{\{v>0\}}\bigr)dx 
=
\int_{B}\bigl(\vert\nabla \uu\vert^{2}+\chi_{\{\vert \uu\vert>0\}}\bigr)dx \\
\leq 
\int_{B}\bigl(\vert\nabla \w\vert^{2}+\chi_{\{\vert \w\vert>0\}}\bigr)dx 
=
\int_{B}\bigl(\vert\nabla\phi\vert^{2}+\chi_{\{\phi>0\}}\bigr)dx\text{,} 
\end{multline*}
which proves that $v$ is an absolute scalar minimizer in $B$ and this finishes the proof of the lemma. 
\end{proof}

In the previous lemma we have considered the cases when 
$\{\vert \uu\vert>0\}\not=B\backslash\{0\}$. 
In the following lemma we consider the case when 
$\{\vert \uu\vert>0\}=B\backslash\{0\}$. 

\begin{lemma}\label{lemma3000} 
There exists no first order homogenous 
absolute minimizer $\uu$ in $B$ such that 
$\{\vert \uu\vert>0\}=B\backslash\{0\}$. 
\end{lemma}
\begin{proof}
Assume that $\uu$ is a first order homogenous 
absolute minimizer in $B$ 
such that $\{\vert \uu\vert>0\}=B\backslash\{0\}$. 
Because $\uu$ is harmonic in $\{\vert \uu\vert>0\}=B\backslash\{0\}$ 
and bounded in a neighborhood of the origin, it follows that 
$\uu$ might be extended as 
a harmonic function at the origin. 
From this it follows that $\{\vert \uu\vert>0\}=B$ a contradiction. 
\end{proof}

\begin{proof}[Proof of Theorem \ref{theorem1400}]
This follows from Lemma \ref{lemma2800}, \ref{lemma2900} and \ref{lemma3000}. 
\end{proof} 

\section{Smoothness of the Free Boundary\\ 
(Proofs of Theorems \ref{theorem1500} and \ref{theorem1600})}\label{section2050}

\begin{proof}[Proof of Theorem \ref{theorem1500}]
This follows as in Sections 3 and 4 of \cite{Weiss1999} by using 
Theorem \ref{theorem1360}, 
Lemma \ref{lemma2700} 
and Theorem \ref{theorem1400}. 
\end{proof}

In the following we prove Theorem \ref{theorem1600}. 
The proof is based on the 
Schauder estimates and the regularity theory of elliptic systems as 
in \cite{Morrey1966} which is a further development of \cite{AgmonDouglisNirenberg1959}. 

\begin{proof}[Proof of Theorem \ref{theorem1600}] 

\hfill

\noindent
\textbf{Step 1)} In this step we outline the partial hodograph transform to straighten the free 
boundary. 

Let $x_0,r_0,\epsilon_1,\epsilon_{2}$ and $i_0$ as 
in subsection \ref{subsection1700}. 
We might assume that $x_0=0$, $i_{0}=1$ and $\nu(0)=-e_n$ where $\nu(x)$ 
is the outward normal of $\partial\{\vert \uu\vert>0\}$ 
at $x$. 
Let $\rho$ be as in Theorem \ref{theorem1360}. 
Thus $B_{\frac{\rho}{4}}\cap\partial\{\vert \uu\vert>0\}$ is the 
graph in direction $-e_n$ of a $C^{1,\alpha}$ function. 
It follows that $u_1$ is $C^{1,\alpha}$ regular in $B_{\frac{\rho}{8}}\cap\overline{\{\vert \uu\vert>0\}}$.  Let us denote $\tilde r_0=\frac{\rho}{8}$. 

We consider the partial hodograph transform defined as the mapping of 
$x\in B_{\tilde r_0}\cap\{\vert \uu\vert>0\}$ 
to $y$ defined by the equations 
\begin{equation*}
\left\{
\begin{aligned}
& y_{n}=u_{1}(x)\text{,} \\
& y^{\prime}=x^{\prime}\text{.} 
\end{aligned}
\right.
\end{equation*}
Here $y^{\prime}=(y_1,\cdots,y_{n-1})$ and $x^{\prime}=(x_1,\cdots,x_{n-1})$. 
One may see that this mapping is injective. 

Let us denote by $D$ the image of $B_{\tilde r_0}\cap\{\vert \uu\vert>0\}$ 
after this mapping. 
The inverse of this mapping is the Legendre transform defined 
as the mapping of $y\in D$ to $x$ given by 
\begin{equation*}
\left\{
\begin{aligned}
& x_{n}=v_1(y)\text{,} \\
& x^{\prime}=y^{\prime}\text{,}
\end{aligned}
\right.
\end{equation*}
where the function $v_1:D\to\mathbb{R}$ satisfies 
\begin{equation}\label{equ7615} 
y_{n}=u_{1}(y^{\prime},v_1(y))\enskip\text{for all}\enskip y\in D\text{.}
\end{equation}

By differentiating equation \eqref{equ7615} with respect to 
$y_n$ we obtain
\begin{equation}\label{equ7620} 
1=\partial_{x_n}u_{1}(y^{\prime},v_1(y))\partial_{y_{n}}v_1(y)
\end{equation}
and by differentiating equation \eqref{equ7615} with respect to 
$y_i$ for $i=1,\cdots,n-1$ we obtain 
\begin{equation}\label{equ7625} 
0=\partial_{x_i}u_{1}(y^{\prime},v_1(y))
+\partial_{x_{n}}u_{1}(y^{\prime},v_1(y))
\partial_{y_i}v_1(y)\text{.} 
\end{equation}

Let $g$ be a function defined in $B_{\tilde r_{0}}\cap\{\vert u\vert>0\}$ and 
$f(y)=g(x)$, then from \eqref{equ7620} and \eqref{equ7625} we obtain that 
\begin{equation}\label{equ7630} 
\partial_{x_j}g(x)
=
\partial_{y_j}f(y)
-\frac{\partial_{y_j}v_1(y)}{\partial_{y_n}v_1(y)}
\partial_{y_n}f(y)
\enskip\text{for}\enskip j=1,\cdots,n-1
\end{equation}
and 
\begin{equation}\label{equ7635} 
\partial_{x_n}g(x)
=
\frac{1}{\partial_{y_n}v_1(y)}\partial_{y_n}f(y)\text{.}
\end{equation}

Let us define $v_k(y)=u_k(x)$ for $k=2,\cdots,m$. 

\noindent
\textbf{Step 2)} In this step we derive the differential equations satisfied by 
$v_k$ for $k=1,\cdots,m$. 

For functions defined on $D$ let us define the second order differential 
operator 
\begin{equation*}
\mathcal{L}(v_1)f
=
\frac{1+\vert \nabla_{y^{\prime}}v_1\vert^{2}}{(\partial_{y_{n}}v_1)^{2}}
\partial_{y_ny_n}f
+\Delta_{y^{\prime}}f
-
2\frac{\nabla_{y^{\prime}}v_1}{\partial_{y_{n}}v_1}
\cdot 
\nabla_{y^{\prime}}\partial_{y_n}f
\text{.} 
\end{equation*}
Let $g(x)=f(y)$ then using \eqref{equ7630} and \eqref{equ7635} we compute 
\begin{equation}\label{equ7637} 
\Delta g(x)
=\mathcal{L}(v_1)f
-\frac{\partial_{y_n}f}{\partial_{y_{n}}v_1}\mathcal{L}(v_1)v_1\text{.} 
\end{equation}
Because $\Delta u_1=0$ in $\{\vert \uu\vert>0\}$ and $u_1(x)=y_n$ 
we have 
\begin{equation*}
0=\Delta u_1
=\mathcal{L}(v_1)y_n
-\frac{\partial_{y_n}y_n}{\partial_{y_{n}}v_1}
\mathcal{L}(v_1)v_1
=-\frac{1}{\partial_{y_{n}}v_1}\mathcal{L}(v_1)v_1
\text{.} 
\end{equation*}
It follows that 
\begin{equation}\label{equ7640} 
\mathcal{L}(v_1)v_1=0\enskip\text{in}\enskip D\text{.} 
\end{equation}
Because $\Delta u_k=0$ in $\{\vert \uu\vert>0\}$ 
from \eqref{equ7637} and \eqref{equ7640} we obtain 
\begin{equation}\label{equ7642} 
0=\Delta u_{k}
=
\mathcal{L}(v_1)v_k
-\frac{\partial_{y_n}v_k}{\partial_{y_{n}}v_1}
\mathcal{L}(v_1)v_1
=
\mathcal{L}(v_1)v_k
\enskip\text{in}\enskip D
\text{.} 
\end{equation}

The free boundary is the graph of the function $v_1(y^{\prime},0)$, 
therefore we have 
\begin{equation}\label{equ7644} 
\nu(x)=\frac{(\nabla_{y^{\prime}} v_1,-1)}
{\sqrt{1+\vert\nabla_{y^{\prime}}v_1 \vert^{2}}}\text{.}
\end{equation}
Using \eqref{equ7630}, \eqref{equ7635}, \eqref{equ7644} 
and $u_1(x)=y_n$ we have 
\begin{equation*}
\partial_{\nu}u_1(x)
=\nu(x)\cdot\nabla u_1(x)
=-\frac{\sqrt{1+\vert\nabla_{y^{\prime}}v_1 \vert^{2}}}{\partial_{y_{n}}v_1}\text{.} 
\end{equation*}

Similarly for $k=2,\cdots,m$, 
using \eqref{equ7630}, \eqref{equ7635}, \eqref{equ7644} 
and $\nabla_{y^{\prime}}v_k(y^{\prime},0)=0$ we have 
\begin{equation*}
\partial_{\nu}u_k(x)
=
\nu(x)\cdot\nabla u_k(x)
=
-
\frac{
\sqrt{1+\vert\nabla_{y^{\prime}}v_1 \vert^{2}}
}{\partial_{y_{n}}v_1}
\partial_{y_n}v_k
\text{.} 
\end{equation*}

Now by the free boundary condition (which follows from \eqref{equ600}) 
\begin{equation*}
Q^{2}(x)
=
\sum_{k=1}^{m}(\partial_{\nu}u_k(x))^{2}
\enskip\text{for}\enskip x\in B_{\tilde r_{0}}\cap\partial \{\vert \uu\vert>0\}
\end{equation*}
we obtain 
\begin{equation}\label{equ7648} 
Q^{2}(x)
=\frac{(1+\vert\nabla_{y^{\prime}}v_1\vert^{2})}{(\partial_{y_{n}}v_1)^{2}}
\bigl(1+\sum_{k=2}^{m}(\partial_{y_n}v_k)^{2}\bigr)
\enskip\text{for}\enskip x\in B_{\tilde r_{0}}\cap\partial \{\vert \uu\vert>0\}
\text{.} 
\end{equation}
Thus $v_1$ satisfies 
\begin{equation}\label{equ7649.100} 
\left\{
\begin{aligned}
& \mathcal{L}(v_1)v_1=0
\enskip\text{in}\enskip D\text{,} \\
& 
\frac{\partial_{y_{n}}v_1}{\sqrt{1+\vert\nabla_{y^{\prime}}v_1\vert^{2}}}
=
\frac{1}{Q(y^{\prime},v_1(y))}
\bigl(1+\sum_{k=2}^{m}(\partial_{y_n}v_k)^{2}\bigr)^{\frac{1}{2}}
\enskip\text{on}\enskip\overline{D}\cap\{y_n=0\}
\end{aligned}
\right.
\end{equation}
and for $k=2,\cdots,m$ we have 
\begin{equation}\label{equ7649.200} 
\left\{
\begin{aligned}
& \mathcal{L}(v_1)v_k=0
\enskip\text{in}\enskip D\text{,} \\
& v_k=0\enskip\text{on}\enskip\overline{D}\cap\{y_n=0\}\text{.}
\end{aligned}
\right.
\end{equation}

\noindent
\textbf{Step 3)} In this step we show that the linear, homogenous second order and 
scalar operator $\mathcal{L}(v_1)$ is uniformly elliptic. 

Let $D_r$ denote the set of those $y$ which correspond 
to those $x\in B_r\cap\{\vert \uu\vert>0\}$. 
Because in $B_{\tilde r_0}\cap\overline{\{\vert \uu\vert>0\}}$ 
we have $u_1\in C^{1,\alpha}$ it follows that 
for small enough $0<r_{2}<\tilde r_0$ we have 
\begin{equation*}
0<c\leq \partial_{y_{n}}v_1\leq C\enskip\text{and}\enskip 
\vert \nabla_{y^{\prime}}v_1\vert\leq\epsilon\enskip\text{in}\enskip D_{r_{2}}
\end{equation*}
where  
\begin{equation*}
0<\epsilon\leq\frac{c}{2\max(1,C^2)}\text{.} 
\end{equation*}
For $\zeta\in\mathbb{R}^{n}$ and $y\in D_{r_{2}}$ we compute 
\begin{multline*}
\frac{1}{(\partial_{y_{n}}v_1)^{2}} 
\bigl(1+\vert \nabla_{y^{\prime}}v_1\vert^{2}\bigr)\zeta_n^2
+\vert\zeta^{\prime}\vert^{2}
-\frac{2}{\partial_{y_{n}}v_1}
(\nabla_{y^{\prime}}v_1\cdot \zeta^{\prime}) \zeta_n \\
\geq 
\frac{1}{(\partial_{y_{n}}v_1)^{2}} \zeta_n^2
+\vert\zeta^{\prime}\vert^{2}
-\frac{2}{\partial_{y_{n}}v_1}
\vert\nabla_{y^{\prime}}v_1\vert \vert \zeta^{\prime}\vert \vert\zeta_n\vert \\
\geq 
\frac{1}{C^{2}} \zeta_n^2
+\vert\zeta^{\prime}\vert^{2}
-\frac{2}{c} 
\epsilon \vert \zeta^{\prime}\vert \vert\zeta_n\vert \\
\geq 
\frac{1}{C^{2}} \zeta_n^2
+\vert\zeta^{\prime}\vert^{2}
-\frac{2}{c} 
\epsilon 
\bigl(
\frac{1}{2}\vert \zeta^{\prime}\vert^{2}
+\frac{1}{2}\zeta_n^{2}
\bigr) \\
=\bigl(\frac{1}{C^{2}}-\frac{\epsilon}{c}\bigr)\zeta_n^2
+\bigl(1-\frac{\epsilon}{c}\bigr)\vert\zeta^{\prime}\vert^{2} 
\geq 
\frac{1}{2\max(1,C^2)}
\vert\zeta\vert^{2}
\end{multline*}
which proves the claim of this step. 

\noindent\textbf{Step 4)} In this step we show that if $Q\in C^{1,\gamma}$ 
then $v_k\in C^{2,\min(\alpha,\gamma)}(D_{\frac{1}{2}r_{2}})$ for $k=1,\cdots,m$. 
Because $v_1\in C^{1,\alpha}(D)$ the coefficients of the operator $\mathcal{L}(v_1)$ 
are $C^{\alpha}(D)$ regular. 
Also from the previous step we have that this operator is 
uniformly elliptic in $D_{r_{2}}$. 

For $k=2,\cdots,m$ because $v_k$ satisfies \eqref{equ7649.200} 
from Schauder estimates it follows that $v_k\in C^{2,\alpha}(D_{\frac{1}{2}r_{2}})$. 
It is easy to see that because $v_1\in C^{1,\alpha}(D)$ and 
$Q\in C^{1,\gamma}$ we have  
$Q(y^{\prime},v_1(y))\in C^{1,\min(\alpha,\gamma)}(D)$. 
Now the right hand side of the second equation in \eqref{equ7649.100} 
is in $C^{1,\min(\alpha,\gamma)}(D_{\frac{1}{2}r_{2}})$. 
Because $v_1$ satisfies \eqref{equ7649.100}, 
from Schauder estimates it follows that 
$v_1\in C^{2,\min(\alpha,\gamma)}(D_{\frac{1}{4}r_{2}})$.

\noindent\textbf{Step 5)} In this step we collect the equations 
satisfied by all $v_k$ for $k=1,\cdots,m$ in a nonlinear system. 

Let us define 
\begin{equation*}
F_k(\vv)=\mathcal{L}(v_1)v_k\enskip\text{for}\enskip k=1,\cdots,m\text{,} 
\end{equation*}
\begin{equation*}
\Phi_1(y^{\prime},\vv)
=\frac{(1+\vert\nabla_{y^{\prime}}v_1\vert^{2})}{(\partial_{y_{n}}v_1)^{2}}
\bigl(1+\sum_{k=2}^{m}(\partial_{y_n}v_k)^{2}\bigr)-Q^{2}(y^{\prime},v_1)
\end{equation*}
and 
\begin{equation*}
\Phi_k(\vv)=v_k\enskip\text{for}\enskip k=2,\cdots,m\text{.}
\end{equation*}

Now by \eqref{equ7649.100} and \eqref{equ7649.200} we obtain the nonlinear system 
\begin{equation}\label{equ7649.300} 
\left\{
\begin{aligned}
& F_k(\vv)=0\enskip\text{in}\enskip D\enskip\text{for}\enskip k=1,\cdots,m\text{,} \\
& \Phi_k(\vv)=0\enskip\text{on}\enskip\overline{D}\cap\{y_n=0\}
\enskip\text{for}\enskip k=1,\cdots,m\text{.}
\end{aligned}
\right.
\end{equation}

\noindent\textbf{Step 6)} In this step we compute the linearization of the nonlinear system. 

We compute 
\begin{multline}\label{equ7649.1000} 
\frac{d}{dt}\mathcal{L}(v_1+t\bar{v}_1)f\big\vert_{t=0} \\
=\frac{2}{(\partial_{y_{n}}v_1)^{2}}
\Bigl(
\bigl(\nabla_{y^{\prime}}v_1\cdot\nabla_{y^{\prime}}\partial_{y_n}f 
-\frac{(1+\vert \nabla_{y^{\prime}}v_1\vert^{2})}{\partial_{y_{n}}v_1}
\partial_{y_{n}y_{n}}f \bigr)
\partial_{y_n}\bar{v}_1 \\
+
\bigl(\partial_{y_{n}y_{n}}f\nabla_{y^{\prime}}v_1
-\partial_{y_{n}}v_1\nabla_{y^{\prime}}\partial_{y_n}f\bigr)
\cdot\nabla_{y^{\prime}}\bar{v}_1
\Bigr)
\text{,} 
\end{multline}
\begin{multline}\label{equ7649.1010} 
\frac{d}{dt}\Phi_1(y^{\prime},v_1+t\bar{v}_1,v_2,\cdots,v_m)\big\vert_{t=0} \\
=
\frac{-2}{(\partial_{y_{n}}v_1)^{2}}
\bigl(1+\sum_{k=2}^{m}(\partial_{y_n}v_k)^{2}\bigr)
\Bigl(
\frac{(1+\vert\nabla_{y^{\prime}}v_1\vert^{2})}{\partial_{y_{n}}v_1}
\partial_{y_{n}}\bar{v}_1-\nabla_{y^{\prime}}v_1\cdot\nabla_{y^{\prime}}\bar{v}_1
\Bigr) \\
-2Q(y^{\prime},v_1)
\partial_{x_n}Q(y^{\prime},v_1)
\bar{v}_{1}
\end{multline}
and for $j=2,\cdots,m$ 
\begin{equation}\label{equ7649.1020} 
\frac{d}{dt}\Phi_1(y^{\prime},v_1,\cdots,v_j+t\bar{v}_j,\cdots,v_m)\big\vert_{t=0} \\
=2\frac{(1+\vert\nabla_{y^{\prime}}v_1\vert^{2})}{(\partial_{y_{n}}v_1)^{2}}
\partial_{y_n}v_j\partial_{y_n}\bar{v}_j\text{.}
\end{equation}

For $k,j=1,\cdots,m$ we denote by $D_{j}F_k(\vv)$ the derivative of $F_k$ in the direction $v_j$, i.e. 
assuming $\vv$ and $\bar{v}_j$ smooth enough 
\begin{equation*}
\frac{d}{dt}F_k(v_1,\cdots,v_j+t\bar{v}_j,\cdots,v_m)\big\vert_{t=0}
=D_{j}F_{k}(\vv)\bar{v}_j\text{.} 
\end{equation*}
Similarly we define $D_{j}\Phi_1(y^{\prime},\vv)$ and $D_{j}\Phi_k(\vv)$ for $k=2,\cdots,m$.

It follows from \eqref{equ7649.1000} that for all $k,j=1,\cdots,m$ 
\begin{multline*}
D_{j}F_{k}(\vv)\bar{v}_j=\delta_{1j}
\frac{2}{(\partial_{y_{n}}v_1)^{2}}
\Bigl(
\bigl(\nabla_{y^{\prime}}v_1\cdot\nabla_{y^{\prime}}\partial_{y_n}v_k
-\frac{(1+\vert \nabla_{y^{\prime}}v_1\vert^{2})}{\partial_{y_{n}}v_1}
\partial_{y_{n}y_{n}}v_k \bigr)
\partial_{y_n}\bar{v}_j \\
+
\bigl(\partial_{y_{n}y_{n}}v_k\nabla_{y^{\prime}}v_1
-\partial_{y_{n}}v_1\nabla_{y^{\prime}}\partial_{y_n}v_k\bigr)
\cdot\nabla_{y^{\prime}}\bar{v}_j
\Bigr) 
+\delta_{kj}\mathcal{L}(v_1)\bar{v}_j\text{.}
\end{multline*}

Also from \eqref{equ7649.1010} and \eqref{equ7649.1020} it follows 
\begin{multline*}
D_{1}\Phi_{1}(y^{\prime},\vv)\bar{v}_{1} \\
=\frac{-2\bigl(1+\sum_{k=2}^{m}(\partial_{y_n}v_k)^{2}\bigr)}{(\partial_{y_{n}}v_1)^{2}}
\Bigl(
\frac{(1+\vert\nabla_{y^{\prime}}v_1\vert^{2})}{\partial_{y_{n}}v_1}\partial_{y_n}\bar{v}_{1}
-\nabla_{y^{\prime}}v_1\cdot\nabla_{y^{\prime}}\bar{v}_{1}
\Bigr) \\
-2Q(y^{\prime},v_1)
\partial_{x_n}Q(y^{\prime},v_1)\bar{v}_{1}\text{,} 
\end{multline*} 
\begin{equation*}
D_{j}\Phi_{1}(\vv)\bar{v}_{j}=2\frac{(1+\vert\nabla_{y^{\prime}}v_1\vert^{2})}{(\partial_{y_{n}}v_1)^{2}}
\partial_{v_n}v_j\partial_{y_n}\bar{v}_{j}\enskip\text{for}\enskip j=2,\cdots,m
\end{equation*}
and 
\begin{equation*}
D_{j}\Phi_{k}(\vv)\bar{v}_{j}=\delta_{kj}\bar{v}_{j}
\enskip\text{for}\enskip k=2,\cdots,m\enskip\text{and}\enskip j=1,\cdots,m\text{.} 
\end{equation*}

\noindent\textbf{Step 7)} In this step we compute the principal part of the linearization. 

The theory outlined in \cite{KinderlehrerNirenbergSpruck1978} 
requires a special structure for the orders of principal parts of the linearized system. 
These orders are described by integers $s_k$, $t_j$ and $r_k$ for $j,k=1,\cdots,m$.  
The order of $D_{j}F_k$ should be less than or equal to $s_k+t_j$ 
with its principal part $D_{j}^{\prime}F_k$ having order $s_k+t_j$. 
Similarly the order of $D_{j}\Phi_k$ should be less than or equal to $r_k+t_j$ 
with its principal part $D_{j}^{\prime}\Phi_k$ having order $r_k+t_j$. 
Let us note that we consider the null operator to be of any integer order. 

For our system we choose 
$t_j=2$, $s_k=0$ and $r_k=(-1)\chi_{\{k=1\}}+(-2)\chi_{\{k\not=1\}}$ 
for $k,j=1,\cdots,m$. 

Then from the expressions derived for the linear parts in the previous step it follows that 
the principal parts are given by 
\begin{equation}\label{equ7649.1030} 
D_{j}^{\prime}F_{k}(\vv)=\delta_{kj}\mathcal{L}(v_1)\text{,} 
\end{equation}
\begin{multline}\label{equ7649.1040} 
D_{1}\Phi_{1}(y^{\prime},\vv)\bar{v}_{1} \\
=\frac{-2\bigl(1+\sum_{k=2}^{m}(\partial_{y_n}v_k)^{2}\bigr)}{(\partial_{y_{n}}v_1)^{2}}
\Bigl(
\frac{(1+\vert\nabla_{y^{\prime}}v_1\vert^{2})}{\partial_{y_{n}}v_1}\partial_{y_n}\bar{v}_{1}
-\nabla_{y^{\prime}}v_1\cdot\nabla_{y^{\prime}}\bar{v}_{1}
\Bigr)\text{,} 
\end{multline} 
\begin{equation}\label{equ7649.1042} 
D_{j}^{\prime}\Phi_{1}(\vv)\bar{v}_{j}=2\frac{(1+\vert\nabla_{y^{\prime}}v_1\vert^{2})}{(\partial_{y_{n}}v_1)^{2}}
\partial_{v_n}v_j\partial_{y_n}\bar{v}_{j}\enskip\text{for}\enskip j=2,\cdots,m
\end{equation}
and 
\begin{equation}\label{equ7649.1044} 
D_{j}^{\prime}\Phi_{k}(\vv)\bar{v}_{j}=\delta_{kj}\bar{v}_{j}
\enskip\text{for}\enskip k=2,\cdots,m\enskip\text{and}\enskip j=1,\cdots,m\text{.} 
\end{equation}

\noindent\textbf{Step 8)} In this step we prove that the principal part of the linearization 
is elliptic and coercive at $y=0$ as defined in \cite{KinderlehrerNirenbergSpruck1978}. 

The principal part of the linearization at $y=0$ is given by 
\begin{equation}\label{equ7649.1045} 
D_{j}^{\prime}F_{k}(\vv)\big\vert_{y=0}\bar{v}_{j}
=\delta_{kj}\Delta_{y}\bar{v}_{j}
\text{,} 
\end{equation}
\begin{equation}\label{equ7649.1046} 
D_{1}\Phi_{1}(y^{\prime},\vv)\big\vert_{y=0}\bar{v}_{1} \\
=-2\bigl(1+\sum_{k=2}^{m}(\partial_{y_n}v_k(0))^{2}\bigr)\partial_{y_n}\bar{v}_{1}
\text{,} 
\end{equation} 
\begin{equation}\label{equ7649.1047} 
D_{j}^{\prime}\Phi_{1}(\vv)\big\vert_{y=0}\bar{v}_{j}
=2\partial_{y_n}v_j(0)\partial_{y_n}\bar{v}_{j} 
\enskip\text{for}\enskip j=2,\cdots,m
\end{equation}
and 
\begin{equation}\label{equ7649.1048} 
D_{j}^{\prime}\Phi_{k}(\vv)\big\vert_{y=0}\bar{v}_{j}=\delta_{kj}\bar{v}_{j}
\enskip\text{for}\enskip k=2,\cdots,m\enskip\text{and}\enskip j=1,\cdots,m\text{.} 
\end{equation}

By \eqref{equ7649.1045} the principal part has a diagonal structure and 
on the diagonal we have Laplacians. By this simple structure the 
ellipticity is easy to check. 

From \eqref{equ7649.1045}, \eqref{equ7649.1046}, \eqref{equ7649.1047} 
and \eqref{equ7649.1048} it follows that the system is coercive if 
$\partial_{y_n}\bar{\varphi}(0)=0$ is coercive for $\Delta\bar{\varphi}=0$. 
But the latter is easy to check and thus we obtain the coercivity of the system. 

\noindent\textbf{Step 9)} In this step we finish the proof of the theorem. 

By step 4 if $Q\in C^{1,\gamma}$ 
then $v\in C^{2,\min(\alpha,\gamma)}(D_{\frac{1}{4}r_{2}})$. 
By this initial regularity, the ellipticity and coercivity at $y=0$ as demonstrated in the 
previous step the proof of the theorem follows from 
Theorem 6.8.2 in \cite{Morrey1966}. 
\end{proof} 

\appendix

\section{Non-Tangentially Accessible Domains}\label{section2140} 

In the following we bring the definition of non-tangentially accessible domains 
and recall the comparison principle. For more on these one may refer to 
\cite{AguileraCaffarelliSpruck1987} and \cite{Kenig1994}. 
Let us note that the standard reference for non-tangentially accessible domains is 
\cite{JerisonKenig1982}. But in \cite{AguileraCaffarelliSpruck1987} and 
\cite{Kenig1994} the definition of a non-tangentially accessible domain is more 
general than the one in \cite{JerisonKenig1982} and this generalization is necessary 
for the results in this paper. 

\begin{definition}[Harnack chain with parameter $M$]\label{definition4100}
Let $D\subset\mathbb{R}^{n}$ be a domain, $M>1$ 
and $x_1,x_2\in D$. 
A Harnack chain with parameter $M$ from $x_1$ to $x_2$ in $D$ is a 
finite sequence $B_{r_i}(x_i)\subset D$ for $i=1,\cdots,\ell$ of balls such that 
\begin{equation*}
\frac{r_i}{M}<\operatorname{dist}(B_{r_i}(x_i),\partial D)<Mr_i
\enskip\text{for all}\enskip i=1,\cdots,\ell\text{,} 
\end{equation*}
the first ball contains $x_1$, the last contains $x_2$ and consecutive balls intersect. 
The number of balls in the chain, i.e. $\ell$, is called the length of the chain. 
\end{definition}

\begin{definition}[Non-tangentially accessible domain 
with parameters $M$, $\xi$ and $c$]\label{definition4300}
A bounded domain $D$ in $\mathbb{R}^{n}$ is called 
a non-tangentially accessible with parameters $M>1$, $\xi>0$ and $0<c<1$ 
when  
\begin{enumerate}[(i)]\label{enumerate1000} 
\item\label{definition4300-1} $D$ satisfies corkscrew condition 
with parameters $M$ and $\xi$, i.e. 
for any $x\in\partial D$ and $0<r<\xi$ there exists 
$a_{r}(x)\in D\cap B_{r}(x)$ such that 
$\operatorname{dist}(a_{r}(x),\partial D)>\frac{r}{M}$. 
\item\label{definition4300-2} $D^{c}$ satisfies 
uniform positive density condition with parameter $c$, 
i.e. for all $x\in D^{c}$ we have $\vert B_r(x)\cap D^{c}\vert\geq c\vert B_r\vert$. 
\item\label{definition4300-3} 
$D$ satisfies Harnack chain condition with parameter $M$, i.e. 
for $\epsilon>0$ and $x_{1},x_{2}\in D$ such that 
$\operatorname{dist}(x_{1},\partial D)>\epsilon$, 
$\operatorname{dist}(x_{2},\partial D)>\epsilon$ 
and $\vert x_{2}-x_{1}\vert<C\epsilon$ then there exists a Harnack 
chain with parameter $M$ from $x_{1}$ to $x_{2}$ 
whose length depends on $C$, but not $\epsilon$. 
\end{enumerate}
\end{definition}

\begin{lemma}[Comparison principle]\label{lemma3400} 
Let $D\subset\mathbb{R}^{n}$ be a non-tangentially accessible 
domain with parameters $M>1$, $\xi>0$ and $0<c<1$. 
Then there exist $0<\lambda<1$, $0<c_1<1$, $0<c_2<1$, 
$C_{3}>1$ and $C_{4}>0$ depending only on $M$ and $c$ such that 
if $r<c_{1}\xi$, $x\in\partial D$, $v_{1}$ and $v_{2}$ 
be positive harmonic functions in $D$ 
vanishing continuously on $B_{C_{3}r}(x)\cap\partial D$. Then 
\begin{equation}\label{equ7700} 
\frac{1}{C_{4}}
\frac{v_{2}(a_{r}(x))}{v_{1}(a_{r}(x))}
\leq 
\frac{v_{2}(y)}{v_{1}(y)}
\leq C_{4}
\frac{v_{2}(a_{r}(x))}{v_{1}(a_{r}(x))}
\enskip\text{for}\enskip y\in B_{c_{2}r}(x)\cap D
\end{equation}
and 
\begin{equation}\label{equ7800}
[\frac{v_{2}}{v_{1}}]_{C^{\lambda}(B_{c_{2}r}(x)\cap D)}
\leq 
\frac{C_{4}}{r^{\lambda}}
\frac{v_{2}(a_{r}(x))}{v_{1}(a_{r}(x))}
\text{.}
\end{equation}
\end{lemma}


\end{document}